\documentclass{gtart}
\usepackage{amsmath, amssymb, graphicx, epsfig, verbatim, psfrag, color}
\usepackage[small,nohug,heads=littlevee]{diagrams} 

\newcounter{bean}

\newtheorem{thm}{Theorem}[section]

\newtheorem{cor}[thm]{Corollary}
\newtheorem{lem}[thm]{Lemma}
\newtheorem{prop}[thm]{Proposition}
\newtheorem{defn}[thm]{Definition}

\newtheorem{exas}[thm]{Examples}
\newtheorem{rem}[thm]{Remark}

\newtheorem{algorithm}[thm]{Algorithm}
\numberwithin{equation}{section}

\newcommand\Field{\mathbb F}

\newcommand\Dual{\mathcal D}
\newcommand\Duality\Dual

\newcommand\ws{\mathbf w}

\newcommand\x{\mathbf x}
\newcommand\w{\mathbf w}
\newcommand\z{\mathbf z}

\newcommand\y{\mathbf y}

\newcommand\ModSphere{\ModFlow\left({\mathbb S}\longrightarrow 
\Sym^{g-1}(\Sigma_{1})\times \Sym^2(\Sigma_{2})\right)}
\newcommand\ModSpheres\ModSphere

\newcommand\Mas{\mu}
\newcommand\UnparModSp{\widehat \ModSp}
\newcommand\UnparModFlow\UnparModSp
\newcommand\Mod\ModSp

\newcommand\ModMaps{\mathcal M}
\newcommand\ModSp\ModMaps

\newcommand\Ta{{\mathbb T}_{\alpha}}
\newcommand\Tb{{\mathbb T}_{\beta}}

\newcommand\alphas{\mbox{\boldmath$\alpha$}}

\newcommand\betas{\mbox{\boldmath$\beta$}}

\newcommand{\bfz}{{\mathbb {Z}}}

\newcommand{\bfq}{{\mathbb {Q}}}

\newcommand{\Zmod}[1]{{\mathbb Z}/{#1}{\mathbb Z}}
\newcommand{\mxy}{{\mathfrak{M}}_{\x, \y}}
\newcommand\whelm{}

\newcommand{\kaa}{\mathbf k}
\newcommand{\laa}{\mathbf l}
\newcommand{\Sym}{{\mathrm {Sym}}}

\newcommand{\DD}{\mathfrak D}
\newcommand{\EE}{\mathcal E} 
 \newcommand{\FF}{\mathcal F}
 \newcommand{\Z}{\mathbb Z}   

\newcommand{\bfc}{\mathbb C}

\newcommand{\partiala}{\widehat{\partial}}

\newcommand{\HFa}{{\widehat {\rm {HF}}}}
\newcommand{\HFaa}{{\widetilde {\rm {HF}}}}
\newcommand{\HFmi}{{\rm {HF}}^-}
\newcommand{\CFa}{{\widehat {\rm {CF}}}}
\newcommand{\CFaa}{{\widetilde {\rm {CF}}}}
\newcommand{\Gen}{\mathcal{S}}
\newcommand{\partialaa}{{\widetilde\partial}}

\newcommand\alphak{\mbox{\boldmath$\alpha$}}
\newcommand\betak{\mbox{\boldmath$\beta$}}
\newcommand\gammak{\mbox{\boldmath$\gamma$}}
\newcommand{\HFast}{{\widehat {\rm {HF}}}_{{\rm {st}}}}

\DeclareMathOperator{\SpinC}{Spin^c}

\begin{document}

\title{Combinatorial Heegaard Floer homology and nice Heegaard diagrams}

\author{Peter Ozsv\'ath}
\address{Department of Mathematics, Princeton University,\\
Princeton, NJ, 08544}
\email{petero@math.princeton.edu}

\author{Andr\'{a}s I. Stipsicz}
\address{R{\'e}nyi Institute of Mathematics\\
Budapest, Hungary and\\
Institute for Advanced Study, Princeton, NJ}
\email{stipsicz@math-inst.hu}

\author{Zolt\'an Szab\'o}
\address{Department of Mathematics, Princeton University\\
Princeton, NJ, 08544}
\email{szabo@math.princeton.edu}


\begin{abstract}
  We consider a stabilized version of $\HFa$ of a 3--manifold $Y$
  (i.e. the $U=0$ variant of Heegaard Floer homology for closed
  3--manifolds).  We give a combinatorial algorithm for
  constructing this invariant, starting from a Heegaard decomposition
  for $Y$, and give a topological proof of its invariance
  properties.
\end{abstract}

\maketitle

\section{Introduction}
\label{sec:intro}
Heegaard Floer homology is an invariant for
$3$--manifolds~\cite{OSzF1, OSzF2}, defined using a Heegaard diagram
for the $3$--manifold. Its definition rests on a suitable adaptation of
Lagrangian Floer homology in a symmetric product of the Heegaard
surface, relative to embedded tori which are associated to the
attaching circles.  These Floer homology groups have several
versions. The simplest version $\HFa (Y)$ is a finitely generated
Abelian group, while $\HFmi (Y)$ admits the algebraic structure of a
finitely generated $\Z [U]$--module. Building on these constructions,
one can define invariants of knots \cite{OSzknot, Rasmussen}
and links \cite{OSzlinks} in $3$--manifolds, invariants of smooth
$4$-manifolds~\cite{OSzFour}, contact structures~\cite{Contact},
sutured $3$--manifolds \cite{Juh}, and 3--manifolds with parameterized
boundary~\cite{Bordered}.

The invariants are computed as homology groups of certain chain
complexes. The definition of these chain complexes uses a choice of a
Heegaard diagram of the given 3--manifold, and various further choices
(e.g., an almost complex structure on the symmetric power of the
Heegaard surface). Both the definition of the boundary map and the
proof of independence of the homology from these choices involves
analytic methods. In \cite{SW} Sarkar and Wang discovered that by
choosing an appropriate class of Heegaard diagrams for $Y$ (which they
called \emph{nice}), the chain complex computing the simplest version
$\HFa (Y)$ can be explicitly computed.  
In addition, Sarkar and Wang also showed that every closed 3-manifold
admits a nice Heegaard diagram.
In a similar spirit, in
\cite{MOS} it was shown that all versions of the link Floer homology
groups for links in $S^3$ admit combinatorial descriptions using grid
diagrams. Indeed, in~\cite{MOST}, the topological invariance of this
combinatorial description of link Floer homology is verified using
direct combinatorial methods (and, in particular, avoiding analysis).

The aim of the present work is to develop a version of Heegaard Floer
homology which uses only combinatorial/topological methods, and in
particular is independent of the theory of pseudo-holomorphic
disks. As part of this, we construct a class of Heegaard diagrams for
closed, oriented 3--manifolds which are naturally associated to
pair-of-pants decompositions. The bulk of this paper is devoted to a
direct, topological proof of the topological invariance of the
resulting Heegaard Floer invariants.  In order to precisely state the
main result of the paper, we first introduce the concept of
\emph{stable Heegaard Floer homology groups}.
\begin{defn}\label{def:equiv}
{\whelm Suppose that $V_1, V_2$ are two finite dimensional vector
  spaces over the field $\Field =\Z /2\Z$ and $b_1\geq b_2$ are
  nonnegative integers. The pair $(V_1, b_1)$ is \emph{equivalent} to
  $(V_2, b_2)$ if $V_1\cong V_2\otimes (\Field \oplus \Field
  )^{(b_1-b_2)}$ as vector spaces. This relation generates an
  equivalence relation on pairs of finite dimensional vector spaces
  and nonnegative integers; the equivalence class represented by the
  pair $(V_1,b_1)$ will be denoted by $[V_1, b_1]$.}
\end{defn}

Suppose now that $Y$ is a closed, oriented 3--manifold, which
decomposes as $Y=Y_1\# n (S^1\times S^2)$ (and $Y_1$ contains no
$(S^1\times S^2)$--summand).  Let $\DD = (\Sigma , \alphak , \betak ,
\w )$ denote a \emph{convenient} Heegaard diagram (a special,
multi-pointed nice Heegaard diagram with basepoint set $\w$, to be defined in
Definition~\ref{def:conv}) for $Y_1$ with $ b(\DD )=
\vert \w \vert$ basepoints. Consider the homology $\HFaa (\DD )$ of the chain complex $(\CFaa
(\DD ), \partialaa_{\DD})$ combinatorially defined from the diagram
(cf. Section~\ref{sec:hfhom} for the definition). Furthermore, 
let $\Field$ denote the field $\Z /2\Z$ with two elements.

\begin{defn}
{\whelm With notations as above, let $\HFaa (\DD , n )$ denote $\HFaa (\DD
  ) \otimes (\Field \oplus \Field )^n$ and define the stable Heegaard
  Floer homology $\HFast (Y)$ of $Y$ as $[\HFaa (\DD , n ), b(\DD )]$.}
\end{defn}

\begin{thm}\label{thm:main}
  The stable Heegaard Floer homology $\HFast (Y)$ is a 3--manifold
  invariant.
\end{thm}

The information encoded in $\HFast(Y)$ and $\HFa(Y)$ are equivalent.
Indeed, 
one can prove Theorem~\ref{thm:main} by identifying $\HFaa(\DD)$ with
a stabilized version of $\HFa(Y)$, i.e.  $\HFaa(\DD)\cong
\HFa(Y)\otimes (\Field\oplus\Field)^{b(\DD)-1}$, and then appealing to
the pseudo-holomorphic proof of invariance (see Theorem~\ref{thm:uaz}
in the Appendix below). By contrast, the bulk of the present paper is
devoted to giving a purely topological proof of the invariance of
$\HFast(Y)$. 

The three primary objectives of this paper are the following:
\begin{enumerate}
\item to give an effective construction of Heegaard diagrams for
  3--manifolds for which a chain complex computing $\HFaa(\DD)$ can be explicitly described
  (compare~\cite{SW});
\item to give some relationship between Heegaard Floer homology with
  more classical objects in 3--manifold topology (specifically,
  pair-of-pants decompositions for Heegaard splittings). We hope that
  further investigations along these lines may shed light on
  topological properties of Heegaard Floer homology;
\item to give a self-contained, topological description of some
  version of Heegaard Floer homology. One might hope that the outlines
  of this approach could be applied to studying other Floer-homological 3--manifold
  invariants.
\end{enumerate}

In a similar manner, we will define Heegaard Floer
homology groups $\HFa _T(Y)$ with twisted coefficients, and verify
their invariance as well. Since for a rational homology sphere $Y$
this group is isomorphic to $\HFa(Y)$ of \cite{OSzF1}, this construction
directly gives a 
purely topological definition of the hat-theory for 3-manifolds with
$b_1(Y)=0$.

The outline of the proof is the following. 
We introduce a special class of Heegaard diagrams which we call
\emph{convenient} (multi-pointed) Heegaard diagrams. These diagrams
are constructed by
augmenting pair-of-pants decompositions compatible with a given
Heegaard splitting. These diagrams have the same combinatorial 
properties as those introduced in~\cite{SW}:
for convenient  diagrams, the boundary map in the chain
complex computing $\HFa (Y)$ can be described by counting empty
rectangles and bigons (see Definition~\ref{def:EmptyPolygon} below). Next
we show that any two convenient diagrams for the same 3--manifold
can be connected by a sequence of elementary moves (which we call
\emph{nice isotopies, handle slides} and \emph{stabilizations})
through nice diagrams. By showing that the above nice moves do not
change the stable Floer homology $\HFast (Y)$, we arrive to the
verification of Theorem~\ref{thm:main}. A simple adaptation of the
same method proves the invariance of the twisted invariant $\HFa _T(Y)$.

In this paper, we treat the simplest version of Heegaard Floer
homology -- $\HFa(Y)$ with coefficients in $\Zmod{2}$, for closed
3--manifolds. In the follow-up articles~\cite{signs, spinc, knots}, we extend this approach 
to some
of the finer structures: $\SpinC$ structures, the corresponding
results for knots and links, and signs.

The paper is organized as follows.  In Sections~\ref{sec:first}
through \ref{sec:convdia} we discuss results concerning certain types
of Heegaard diagrams and moves between them. More specifically,
Section~\ref{sec:first} concentrates on pair-of-pants decompositions,
Section~\ref{sec:second} deals with nice diagrams and nice moves,
Section~\ref{sec:cdiagdef} introduces the concept of convenient
diagrams, and Section~\ref{sec:convdia} shows that convenient diagrams
can be connected by nice moves.  This lengthy discussion in
Section~\ref{sec:convdia} --- relying exclusively on simple
topological considerations related to surfaces and Heegaard diagrams
on them --- will be used later in the proof that our invariants are
indeed independent of the choices made.  In Section~\ref{sec:hfhom} we
introduce the chain complex computing the invariant $\HFaa (\DD)$, and
in Section~\ref{sec:cpx} we show that the homology does not change
under nice isotopies and handle slides, and changes in a simple way
under nice stabilization. This result then leads to the proof of
Theorem~\ref{thm:main}, presented in Section~\ref{sec:hom}.  In
Section~\ref{s:twist} we discuss the twisted version of Heegaard Floer
homologies. For completeness, in an Appendix we identify the homology
group $\HFaa (\DD )$ with an appropriately stabilized version of the
Heegaard Floer homology group $\HFa (Y)$ (as it is defined in
\cite{OSzF1}). In addition, for the sake of completeness, in a further
Appendix we verify a version of the result of Luo
(Theorem~\ref{t:luo}) used in the independence proof. The alert reader
will notice that besides the classical Reidemeister--Singer theorem
(on Heegaard splittings of 3-manifolds) and the Kneser-Milnor theorem
we only refer to a result of \cite{SW} (in the proof of
Proposition~\ref{p:atfog}) and a theorem from \cite{sarka} (given in
Theorem~\ref{thm:sark}), hence the paper is rather self-contained.

The convenient diagrams we consider here are multiply-pointed Heegaard
diagrams, which are closely related to pair-of-pants decompositions.
Although this approach uses more curves and more basepoints (than, for
example,~\cite{SW}), we find these diagrams easier to work with.  In
particular, when trying to connect convenient diagrams the problem localizes
inside three- and four-punctured spheres (see for example Proposition~\ref{p:convdomains} and
Theorem~\ref{t:convconn} below), where the problem of connecting diagrams
reduces to examining finitely many cases.

Of course, it is natural to consider nice diagrams with single
basepoints, as provided by the Sarkar-Wang construction.  It would be
very interesting to give a topological invariance proof from this
point of view. Such an approach has been announced by
Wang~\cite{Wang}.

{\bf Acknowledgements}: PSO was supported by NSF grant number
DMS-0804121.  AS was supported by OTKA NK81203, ZSz was supported by
NSF grant number DMS-0704053. We would like to thank the Mathematical
Sciences Research Institute, Berkeley for providing a productive
research enviroment. We would like to thank Jean-Mathieu Magot and an
anonymous referee for many useful comments and suggestions for
improving the presentation of our results.

\section{Heegaard diagrams}
\label{sec:first}
Suppose that $Y$ is a closed, oriented 3--manifold.  It is a standard fact
(and follows, for example, from the existence of a triangulation or from
simple Morse theory) that $Y$ admits a \emph{Heegaard decomposition}
${\mathfrak {U}}=(\Sigma,U_0,U_1)$; i.e.,
\[
Y =U_0\cup _{\Sigma } U_1 ,
\]
where $U_0$ and $U_1$ are handlebodies whose boundary $\Sigma$ is a closed,
connected, oriented surface of genus $g$, called the {\em Heegaard surface} of
the decomposition. 
(We orient $\Sigma $ as $\partial U_0$, hence $\partial U_1=-\Sigma$.)
 By forming the connected sum of a given Heegaard
decomposition with the standard toroidal Heegaard decomposition of $S^3$ we
get the \emph{stabilization} of the given Heegaard decomposition.  By a
classical result of Reidemeister and Singer \cite{R, S}, any two Heegaard
decompositions of a given 3--manifold become isotopic after suitably many
stabilizations, cf. also \cite{saul}.

A genus--$g$ handlebody $U$ can be described by specifying a
collection $\alphak = \{ \alpha _1, \ldots , \alpha _k \}$
of $k$ disjoint, embedded, simple closed curves in
$\partial U=\Sigma$, chosen so that these curves span a
$g$--dimensional subspace of $H_1 (\Sigma ; \bfz )$, and they bound
disjoint disks (usually called \emph{compressing disks})
in $U$.  Attaching 3--dimensional 2--handles to $\Sigma
\times [-1,1]$ along the curves (when viewed them as subsets of
$\Sigma \times \{ 1\}$), we get a cobordism from the surface to a
disjoint union of $k-g+1$ spheres, and by capping these spherical
boundaries with 3--disks, we get the handlebody $U$ back. We will 
also say that $U$ is \emph{determined by} $\alphak $.

A {\em generalized Heegaard diagram} for a closed three manifold
is a triple $(\Sigma,\alphak,\betak)$ where $\alphak$ and
$\betak$ are $k$-tuples of simple closed curves as above,
specifying a Heegaard decomposition ${\mathfrak {U}}$ for $Y$. We
will always assume that in our generalized Heegaard diagrams the
curves $\alpha_i \in \alphak $ and $\beta_j\in \betak $ intersect
each other transversally, and that the Heegaard diagrams are
\emph{balanced}, that is, $\vert \alphak \vert =\vert \betak
\vert$.

\begin{defn}
{\whelm The components of $\Sigma - \alphak -\betak$ are called
\emph{elementary domains}.}
\end{defn}

Notice that an elementary domain --- as part of 
$\Sigma - \alphak $ (or $\Sigma -\betak$) --- is a planar surface.
Let  $D$ be a simply connected elementary domian (i.e.,
$D$ is homeomorphic to the disk). Let $2m$ denote the number of intersection
points of the $\alphak$-- and $\betak$--curves the closure of $D$
(inside $\Sigma$) contains on its boundary. In this case we say that
$D$ is a $2m$--gon; for $m=1$ it will be also called a \emph{bigon}
and for $m=2$ a \emph{rectangle}.

Next we will describe some specific generalized Heegaard diagrams,
called pair-of-pants diagrams. These diagrams have the advantage that
they have a preferred isotopic model (see
Theorem~\ref{thm:nobigon}). In Subsection~\ref{subsec:StabPairOfPants}
we show how they can be stabilized.

\subsection{Pair-of-pants diagrams}
\label{subsec:PairOfPants}

A system of disjoint curves ${\alphak }= \{ \alpha_i \} _{i=1}^k$ in a
closed surface $\Sigma$ is called a \emph{pair-of-pants decomposition}
if every component of $\Sigma-\alphak$ is diffeomorphic to the
2-dimensional sphere with three disjoint disks removed (the so--called
{\em pair-of-pants}).  A pair-of-pants decomposition of $\Sigma$ is
called a \emph{marking} if all curves in the system are homologically
essential in $H_1 (\Sigma; \bfz /2\bfz)$. If the genus $g$ of the
closed surface $\Sigma$ is at least 2 (i.e., the surface is
hyperbolic), such a marking always exists, and the number $k$ of
curves appearing in the system is equal to $3g-3$.  It is easy to see
that a system of curves determining a pair-of-pants decomposition
spans a $g$--dimensional subspace in homology, and hence determines a
handlebody.  We say that two markings on the surface $\Sigma$
\emph{determine the same handlebody} if the identity map id$_{\Sigma}$
extends to a homeomorphism of the handlebodies determined by the
markings.  (Note that any two markings determine diffeomorphic
handlebodies, but two handlebodies built on a surface $\Sigma$ are
equivalent only if the diffeomorphism between them is isotopic to the
identity on the boundary.)  Alternatively, two markings $\alphak $ and
$\alphak '$ determine the same handlebody if, in the handlebody
determined by $\alphak$, the curves $\alpha_i '\in \alphak'$ bound
disjoint embedded disks. The following theorem describes a method to
transform markings determining the same handlebody into each other.
To state the result, we need a definition.
\begin{defn} \label{def:flip}
{\whelm The pair-of-pants decompositions $\alphak =\alphak _0\cup \{
  \alpha_1 \}$ and $\alphak '=\alphak _0 \cup \{ \alpha_1'\}$ of
  $\Sigma $ differ by a \emph{flip} (called a \emph{Type II move} in
  \cite{Luo}) if $\alpha_1, \alpha_1'$ in the 4--punctured sphere
  component of $\Sigma - \alphak _0$ intersect each other
  transversally in two points (with opposite signs);
  cf. Figure~\ref{f:flip}.  We say that $\alphak =\alphak _0\cup \{
  \alpha_1 \}$ and $\alphak '=\alphak _0 \cup \{ \alpha_1'\}$ of
  $\Sigma $ differ by a \emph{generalized flip} (or \emph{g-flip}) if
  $\alpha _1$ and $\alpha _1'$ are contained by the 4--punctured
  sphere component of $\Sigma -\alphak _0$, i.e., we do not require
  the curves $\alpha _1$ and $\alpha _1'$ to intersect in two points.
For an example of a g-flip, see Figure~\ref{f:gflip}.}
\end{defn}
\begin{figure}[ht]
\begin{center}
\epsfig{file=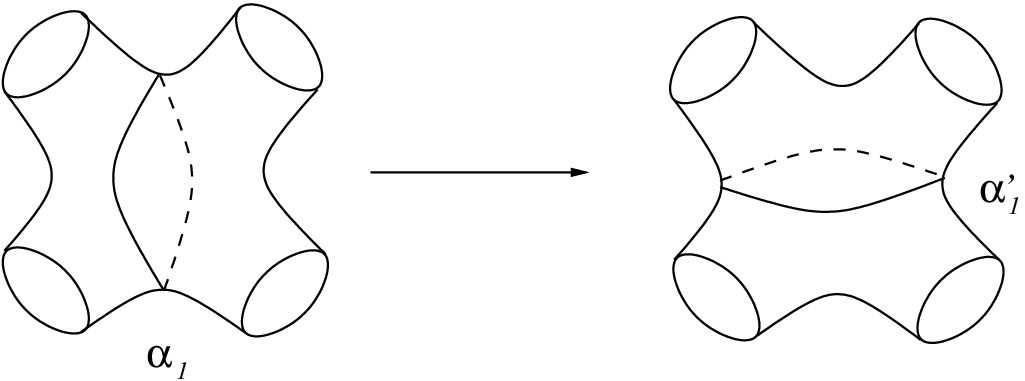, height=4.2cm}
\end{center}
\caption{{\bf The flip (Type II move).}}
\label{f:flip}
\end{figure}

\begin{figure}[ht]
\begin{center}
\epsfig{file=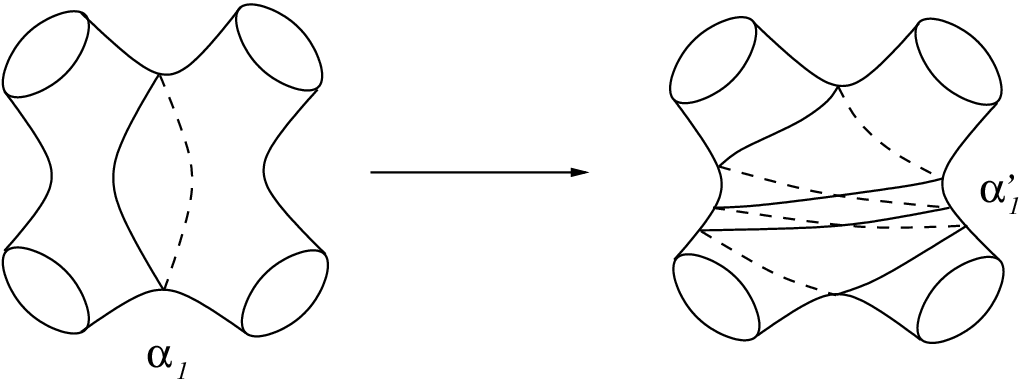, height=4.2cm}
\end{center}
\caption{{\bf An example of a g-flip.}}
\label{f:gflip}
\end{figure}

\begin{thm}(Luo, \cite[Corollary~1]{Luo})\label{t:luo}
Suppose that $\alphak , \alphak '$ are two markings of a given
genus--$g$ surface $\Sigma$.  The two markings determine the same
handlebody if and only if there is a sequence $\{ \alphak _i\}
_{i=1}^n$ of markings such that $\alphak = \alphak _1$, $\alphak
  '=\alphak _n$ and consecutive terms in the sequence $\{ \alphak _i\}
  _{i=1}^n$ differ by a flip or an isotopy. \qed
\end{thm}

\begin{rem}
{\whelm Although the statement of \cite[Corollary~1]{Luo} does not
  state it explicitly, the proof of the main Theorem of \cite{Luo}
  shows that the sequence of flips connecting the two markings
  $\alphak $ and $\alphak '$ can be chosen in such a manner that all
  intermediate curve systems are markings (that is, all curves
  appearing in this sequence are homologically essential). In order to
  make the paper self-contained, we provide a proof of a slightly
  weaker result (namely that the markings determine the same
  handlebody if and only if they can be connected by g-flips) in the
  Appendix, cf. Theorem~\ref{thm:luohely}. In our subsequent
  applications, in fact, the g-flip equivalence is the property that we
  will use.}
\end{rem}

\begin{defn}
{\whelm Let $Y$ be a 3--manifold given by a Heegaard decomposition
  ${\mathfrak {U}}$. Suppose that the two handlebodies are specified
  by pair-of-pants decompositions $\alphak$ and $\betak$ of the
  Heegaard surface $\Sigma$.  Then the triple $(\Sigma , \alphak ,
  \betak )$ is called a \emph{pair-of-pants generalized Heegaard
    diagram}, or simply a \emph{pair-of-pants diagram}, for $Y$.  If
  moreover each of the curves $\alpha_i$ and $\beta_j$ in the systems
  are homologically essential (i.e. $\alphak$ and $\betak$ are both
  markings), then we call the pair-of-pants diagram an {\em essential
    pair-of-pants diagram} for $Y$.}  
\end{defn}

\begin{lem}\label{lem:conn}
Suppose that $(\Sigma , \alphak , \betak )$ and $(\Sigma , \alphak ',
\betak ')$ are two essential pair-of-pants diagrams corresponding to
the Heegaard decomposition ${\mathfrak {U}}=(\Sigma , U_0, U_1)$. Then
there is a sequence $\{ (\Sigma , \alphak _i, \betak _i )\}_{i=1}^m$
of essential pair-of-pants diagrams of ${\mathfrak {U}}$
connecting $(\Sigma , \alphak , \betak )$ and $(\Sigma , \alphak ',
\betak ') $ such that consecutive terms of the sequence differ by a
flip (either on $\alphak $ or on $\betak$).
\end{lem}
\begin{proof}
  Suppose that $\{ \alphak _i \} _{i=1}^{m_1}$ and $\{ \betak _j \}
  _{j=1}^{m_2}$ are sequences of flips connecting $\alphak$ to
  $\alphak '$ and $\betak $ to $\betak '$. Then $\{ (\Sigma , \alphak
  _i , \betak )\} _{i=1}^{m_1}\cup \{ (\Sigma , \alphak ', \betak
  _{i-m_1} )\}_{i=m_1+1}^{m_1+m_2}$ is an appropriate sequence of
  essential diagrams.
\end{proof}

We say that a 3--manifold $Y$ contains no $S^1\times S^2$--summand if
for any connected sum decomposition $Y\cong Y_1 \# n(S^1\times S^2)$ we
have $n=0$.

\begin{lem}\label{l:noiso}
  Suppose that $Y$ contains no $S^1\times S^2$--summand, and
  $(\Sigma , \alphak , \betak )$ is an essential pair-of-pants diagram
  for $Y$.  Then there is no pair $\alpha_i \in \alphak$ and $\beta
  _j \in \betak$ such that $\alpha_i$ is isotopic to $\beta _j$.
\end{lem}

\begin{proof}
  Such an isotopic pair $\alpha_i$ and $\beta_j$ gives an embedded
  sphere $S$ in $Y$, which is homologically nontrivial in $Y$ since
  $\alpha_i$ (as well as  $\beta _j$) is homologically
  essential in the Heegaard surface. Surgery on $Y$ along the sphere
  $S$ results a manifold $Y_1$ with the property that $Y_1\# (S^1\times
  S^2)$ is homeomorphic to $Y$. Therefore by our assumption the
  isotopic pair $\alpha_i$ and $\beta _j$ cannot exist.
\end{proof}
\begin{cor}
Suppose that $Y$ contains no $S^1\times S^2$--summand, and
  $(\Sigma , \alphak , \betak )$ is an essential pair-of-pants diagram
  for $Y$.  Then any $\alphak$--curve is intersected by some $\betak$--curve
(and symmetrically, any
$\betak$--curve is intersected by some $\alphak$--curve).
\end{cor}
\begin{proof}
Suppose that $\alpha_i$ is disjoint from all $\beta _j$. Then $\alpha
_i$ is part of a pair-of-pants component of $\Sigma - \betak$, hence
is parallel to one of the boundary components of the pair-of-pants,
which contradicts the conclusion of Lemma~\ref{l:noiso}. The
symmetric statement follows in the same way.
\end{proof}
\begin{defn}
{\whelm Suppose that $(\Sigma,\alphak, \betak)$ is a Heegaard diagram for
  the $3$--manifold $Y$. We say that the diagram is \emph{bigon--free}
  if there are no elementary domains which are bigons or,
  equivalently, if each $\alpha_i$ intersects each $\beta_j$ a minimal
  number of times.}
\end{defn}
Our aim in this subsection is to prove the following:
\begin{thm}
  \label{thm:nobigon}
  Suppose that $Y$ is a given 3--manifold and $(\Sigma , \alphak ,
  \betak )$ is an essential pair-of-pants diagram for $Y$. Then there is a
  Heegaard diagram $(\Sigma , \alphak ' , \betak ')$ such that
  \begin{itemize}
  \item $\alphak $ and $\alphak '$ (and similarly $\betak $ and $\betak '$)
    are isotopic and
  \item $(\Sigma  , \alphak ' , \betak ')$ is bigon--free.
  \end{itemize}
  If $Y$ contains no $S^1\times S^2$--summand, then the bigon--free
  model is unique up to homeomorphism.  More precisely, if
  $(\Sigma,\alphak',\betak')$ and $(\Sigma,\alphak'',\betak'')$ are
  two bigon--free diagrams for $Y$ for which $\alphak'$ and $\alphak''$
  are isotopic, and $\betak'$ and $\betak''$ are isotopic, then there
  is a homeomorphism $f\colon \Sigma\longrightarrow \Sigma$ isotopic
  to id$_{\Sigma }$ which carries $\alphak'$ to $\alphak''$ and
  $\betak'$ to $\betak''$.
\end{thm}

\begin{rem}
{\whelm In the statement of the above proposition, we assumed that our
  pair-of-pants diagrams were essential. This is, in fact, not needed
  for the existence statement, but it is needed for uniqueness.}
\end{rem}
We return to the proof of the theorem after a definition and a lemma.
\begin{defn}
{\whelm
\begin{itemize}
\item Let ${\DD}$ and ${\DD}'$ be two Heegaard diagrams. We say that
  ${\DD}'$ is obtained from ${\DD}$ by an \emph{elementary
  simplification} if ${\DD}'$ is obtained by eliminating a single
  elementary bigon in ${\DD}$, cf. Figure~\ref{f:dis}(a). (In
  particular, the attaching circles for ${\DD}$ are isotopic to those
  for ${\DD}'$, via an isotopy which cancels exactly two intersection
  points between attaching circles $\alpha_i$ and $\beta_j$ for
  ${\DD}$.)

\item Given a Heegaard diagram ${\DD}$, a \emph{simplifying sequence}
  is a sequence of Heegaard diagrams $\{{\DD}_i\}_{i=0}^n$ with the
  following properties:
  \begin{itemize}
  \item ${\DD}={\DD}_0$
  \item ${\DD}_{i+1}$ is obtained from ${\DD}_i$ by an
    elementary simplification.
  \item ${\DD}_n={\mathcal {E}}$ is bigon--free.
  \end{itemize} 
In this case, we say that ${\DD}={\DD}_0$ \emph{simplifies to}
${\DD}_n={\mathcal {E}}$.

\item If ${\DD}$ is a Heegaard diagram and ${\EE}$ is a
  bigon--free diagram, the \emph{distance} from ${\DD}$ to
  ${\EE}$ is the minimal length of any simplifying sequence
  starting at ${\DD}$ and ending at ${\EE}$. (Of course,
  this distance might be $\infty$; we shall see that this happens only
  if ${\EE}$ is not isotopic to ${\DD}$.)
\end{itemize}
}
\end{defn}

\begin{lem}
  \label{l:simpseq}
  Given a Heegaard diagram ${\DD}$ for a 3--manifold $Y$,
  there exists a simplifying sequence $\{{\DD}_i\}_{i=0}^n$. If
  ${\DD}$ is an essential pair-of-pants diagram, and
  $Y$ contains no $S^1\times S^2$--summand, then  any two
  simplifying sequences starting at ${\DD}$ have the same
  length, and they terminate in the same bigon--free diagram ${\mathcal
    E}$.
\end{lem}

\begin{proof}
  The sequence $\{{\DD}_i\}_{i=0}^n$ is constructed in the following
  straightforward manner.  If the diagram ${\DD}_i$ contains an
  elementary domain which is a bigon, then isotope the $\betak$--curve
  until this bigon disappears, to obtain ${\DD}_{i+1}$ (cf.
  Figure~\ref{f:dis}(a)), and if ${\DD}_i$ does not contain
  any bigons, then stop.
  \begin{figure}[ht]
    \begin{center}
      \epsfig{file=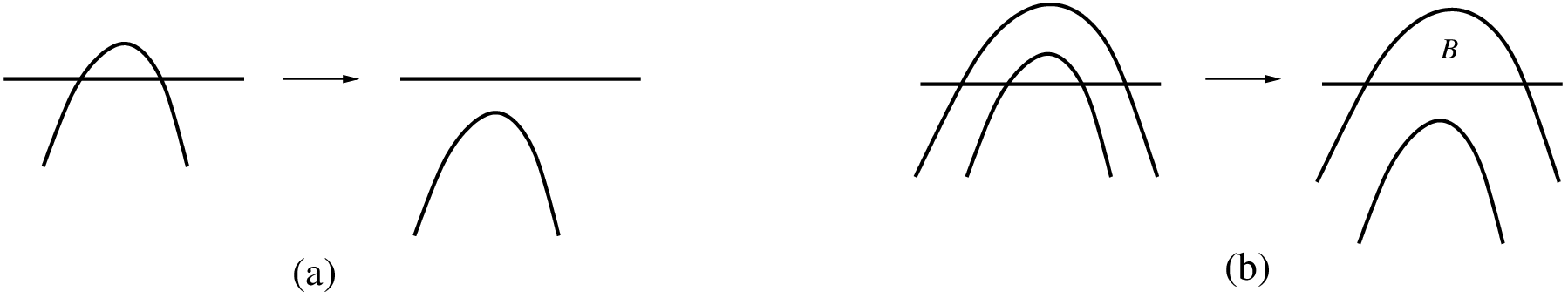, height=2.5cm}
    \end{center}
    \caption{{\bf Elimination of a bigon.} As (b) shows, the elimination of one bigon
might create another one.}
    \label{f:dis}
  \end{figure}
 Although the above isotopy might create new bigons (see $B$ of 
 Figure~\ref{f:dis}(b)), the number of intersection points of the
 $\alphak$-- and $\betak$--curves decreases by two at every elementary
 simplification, hence the
 sequence will eventually terminate in a bigon--free diagram.
 
 Formally, if we define the complexity $K({\DD})$ of a diagram
 ${\DD}$ to be $\sum_{i,j} |\alpha_i\cap\beta_j|$ (where
 $|\cdot|$ denotes the total number of intersection points), then the
 distance $d$ between ${\DD}$ and ${\EE}$ is given by
 $K({\DD})-K({\EE})=2d$. Thus, any two simplifying
 sequences from ${\DD}$ to the same bigon--free diagram
 ${\EE}$ must have the same length.

 Fix now a bigon--free diagram ${\EE}$.  We prove by induction
 on the distance from ${\DD}$ to ${\EE}$ that if
 ${\DD}$ is a diagram with finite distance $d$ from ${\mathcal
   E}$, then any simplifying sequence starting at ${\DD}$
 terminates in ${\EE}$.

The statement is obvious if $d=0$, i.e. if ${\DD}={\EE}$.  By induction,
suppose that we know that every diagram ${\DD}$ with distance $d$ from
the bigon--free ${\EE}$ has the property that each simplifying
sequence starting at ${\DD}$ terminates in ${\EE}$. We must now verify
the following: if $\{{\DD}_i\}_{i=0}^{d+1}$ and $\{{\DD}_i'\}_{i=0}^n$
are two simplifying sequences both starting at
${\DD}={\DD}_0={\DD}_0'$, and with ${\DD}_{d+1}={\EE}$, then in fact
$n=d+1$ and ${\DD}_{n}'={\EE}$.  To see this, note that ${\DD}_1$ is
obtained by eliminating some bigon $B$ in ${\DD}$, and ${\DD}_1'$ is
obtained by eliminating a (potentially) different bigon $B'$ in
${\DD}$. Of course when $B=B'$, induction provides the result.
 
For $B\neq B'$ there are two subcases: either $B$ and $B'$ are  disjoint or
they intersect.  If $B$ and $B'$ are disjoint, we can construct a
third simplifying sequence $\{{\DD}_i''\}_{i=0}^m$ which we construct
by first eliminating the bigon $B$ (so that ${\DD}_1''={\DD}_1$) and
next eliminating $B'$ (and then continuing the sequence arbitarily to
complete these first two steps to a simplifying sequence).  By the
inductive hypothesis applied for ${\DD}_1$, it follows that $m=d+1$
(since the distance from ${\DD}_1$ to ${\EE}$ is $d$), and that
${\DD}_{m}''={\EE}$. We now consider a fourth simplifying sequence
which looks the same as the third, except we eliminate the first two
bigons in the opposite order; i.e. we have $\{{\DD}_i'''\}_{i=0}^m$
with the property that ${\DD}_1'''={\DD}_1'$ and
${\DD}_i'''={\DD}_i''$ for $i\geq 2$. The existence of this sequence
ensures that the distance from ${\DD}_1'$ to ${\EE}$ is $d$, and
hence, by the inductive hypothesis, $n=d+1$, and ${\DD}_n'={\EE}$, as
needed.

Suppose now that the bigons $B$ and $B'$ are not disjooint.
Since the curves in the markings are homologically essential, two distinct elementary
bigons cannot share a side.
Therefore the two bigons share
 at least one corner. In case the two bigons share two corners, we get
 parallel $\alphak$-- and $\betak$--curves, contradicting our
 assumption, cf. Lemma~\ref{l:noiso}. (Recall that we assumed that $Y$
 has no $(S^1\times S^2)$--summands.) If the two bigons share exactly
 one corner, then by a simple local consideration, it follows that
 ${\DD}_2$ and ${\DD}_2'$ are already isotopic.
 cf. Figure~\ref{f:nobigg}. In particular, the inductive hypothesis
 immediately applies, to show that $n=d+1$ and ${\DD}_n'={\mathcal
   E}$.
\end{proof}
\begin{figure}[ht]
\begin{center}
\input{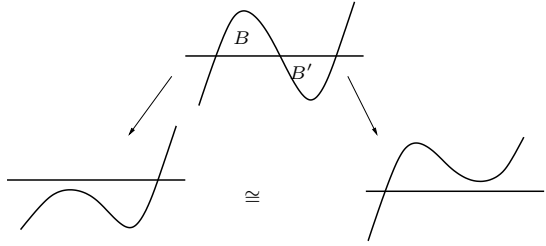}
\end{center}
\caption{{\bf Elimination of bigons with nontrivial intersection in
    different orders.}}
\label{f:nobigg}
\end{figure}
Armed with this lemma, we are ready to give the proof of the theorem:
\begin{proof}[Proof of Theorem~\ref{thm:nobigon}]
  Note first that if ${\DD}_2$ is obtained from ${\DD}_1$
  by an elementary simplification, then both ${\DD}_1$ and
  ${\DD}_2$ simplify to the same bigon--free diagram. To see
  this, take a simplifying sequence starting at ${\DD}_2$
  (whose existence is guaranteed by Lemma~\ref{l:simpseq}), and prepend
  ${\DD}_1$ to the sequence.

  Suppose now that there are two bigon--free diagrams ${\EE}_1$
  and ${\EE}_2$, both isotopic to a fixed, given one.  This, in
  particular, means that the bigon--free diagrams ${\EE}_1$
  and ${\EE}_2$ are isotopic.  Making the isotopy generic,
  and subdividing it into steps, we find a sequence of diagrams
  $\{{\DD}_i\}_{i=1}^m$ where:
  \begin{itemize}
  \item ${\EE}_1={\DD}_1$ and ${\EE}_2={\DD}_m$
  \item  ${\DD}_i$ and ${\DD}_{i+1}$ differ by an elementary
    simplification; i.e. either ${\DD}_{i+1}$ is obtained from ${\DD}_i$ by
    an elementary simplification or vice versa.
  \end{itemize}
  By the above remarks, any two consecutive terms simplify to the same
  bigon--free diagram.  Since by Lemma~\ref{l:simpseq} that
  bigon--free diagram is unique, there is a fixed bigon--free diagram
  ${\FF}$ with the property that any of the diagrams ${\DD}_i$
  simplifies to ${\FF}$. Since ${\DD}_1={\EE}_1$ and
  ${\DD}_n={\EE}_2$ are already bigon--free, it follows
  that ${\EE}_1\cong {\FF}\cong {\EE}_2$.
\end{proof}

In our subsequent discussions the combinatorial shapes of the
components of $\Sigma - \alphak -\betak$ will be of central
importance.  As the next result shows, a bigon--free essential
pair-of-pants decomposition is rather simple in that respect. In fact,
for purposes which will become clear later, we consider the slightly
more general situation where we delete one curve from $\alphak$.

\begin{prop}\label{p:convdomains}
Suppose that $Y$ contains no $S^1\times S^2$--summand, $(\Sigma ,
\alphak , \betak )$ is a bigon--free, essential pair-of-pants Heegaard
diagram for $Y$, and let $\alphak_1 $ be given by deleting an
arbitrary curve from $\alphak$. Then each $\beta _j \in \betak $ is
intersected by some curve in $\alphak_1$, and the components of
$\Sigma -\alphak_1 -\betak $ are either rectangles, hexagons or
octagons. Consequently, the components of $\Sigma - \alphak -\betak $
are also either rectangles, hexagons or octagons.
\end{prop}
\begin{proof}
Suppose that there is a $\betak$--curve (say $\beta_1$) which is
disjoint from all the curves in $\alphak_1$. Any component of $\Sigma
-\alphak_1$ is either a three--punctured or a four--punctured
sphere. By its disjointness, $\beta_1$ must be in one of these
components. If it is in a three--punctured sphere, then it is isotopic
to a boundary component (which is a curve in $\alphak_1$),
contradicting Lemma~\ref{l:noiso}. If $\beta_1$ is in the
four--punctured sphere component, then it is either isotopic to a
boundary curve (contradicting Lemma~\ref{l:noiso} again), or it separates
the component into two pairs-of-pants. Therefore by adding a small
isotopic translate of $\beta_1$ to $\alphak_1$ we would get an essential
pair-of-pants diagram $(\Sigma , \alphak ', \betak )$ for
$Y$ which contradicts Lemma~\ref{l:noiso}. This shows that there is no
$\beta_1$ which is disjoint from all the curves in $\alphak_1$.

Since there are no bigons in $(\Sigma , \alphak , \betak )$, there are
obviously no bigons in $(\Sigma , \alphak_1 , \betak )$ either.
Consider a pair-of-pants component $P$ of $\Sigma - \betak $ and
(a component of) the interesection of $P$ with a curve in $\alphak
_1$. This arc either intersects one or two boundary components. Notice
that since there are no bigons in the decomposition, the $\alphak
_1$--arc cannot be boundary parallel.  Figure~\ref{f:arc} shows the
two possibilities (up to diffeomorphism on the pair-of-pants).
\begin{figure}[ht]
\begin{center}
\epsfig{file=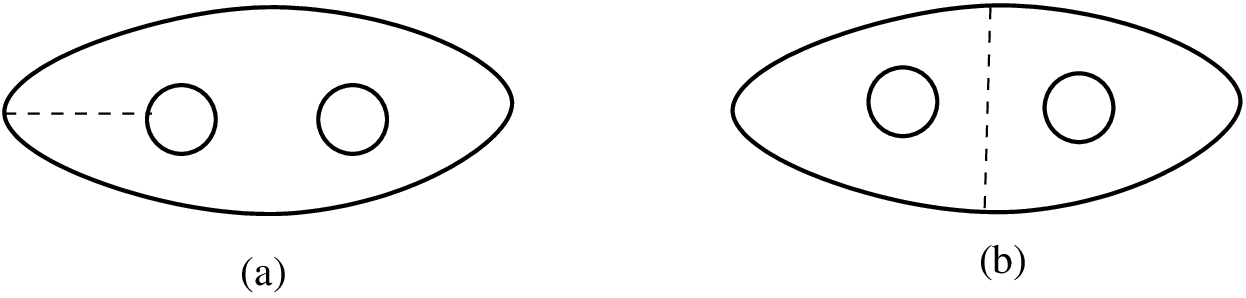, height=2.5cm}
\end{center}
\caption{{\bf The dashed line represents the $\alphak_1$--arc in the
  $\betak$--pair-of-pants $P$.}}
\label{f:arc}
\end{figure}
By denoting a bunch of parallel $\alphak_1$--arcs with a unique
interval we get three possibilities for the $\alphak_1$--curves in a
component of the $\betak$--pair-of-pants, as shown in
Figure~\ref{f:poss}.
\begin{figure}[ht]
\begin{center}
\epsfig{file=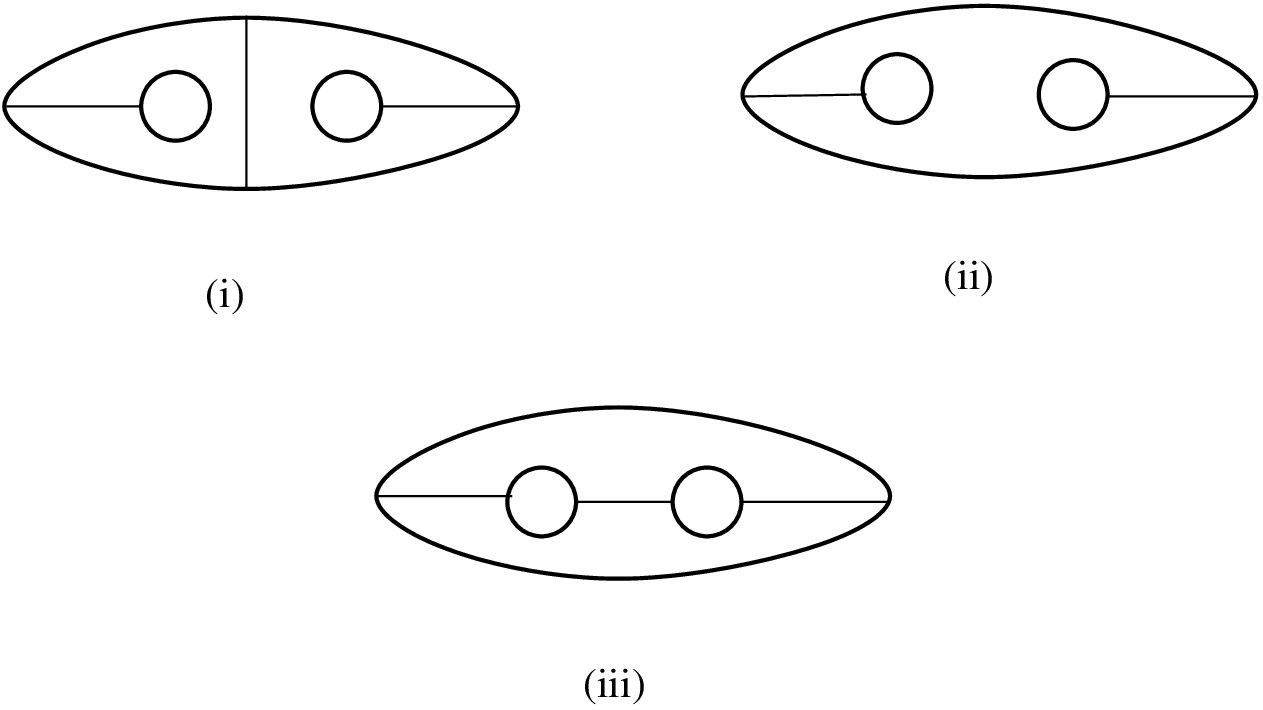, height=5cm}
\end{center}
\caption{{\bf Possible $\alphak $--arcs in a $\betak$--pair-of-pants.}
  Intervals denote parallel copies of $\alphak$--arcs.}
\label{f:poss}
\end{figure}
(Notice that we already showed that any $\betak$--curve is intersected
by some $\alphak$--curve.) Since all the domains in such a
pair-of-pants diagram are $2m$--gons with $m=2,3,4$, this observation
verifies the claim regarding the shape of the domains in $(\Sigma ,
\alphak_1, \betak )$. Obviously, adding the deleted single
$\alphak$--curve back, the same conclusion can be drawn for the
components of $\Sigma -\alphak -\betak $.
\end{proof}

\subsection{Stabilizing pair-of-pants diagrams}
\label{subsec:StabPairOfPants}

Suppose that $(\Sigma , \alphak , \betak )$ is a given essential
pair-of-pants Heegaard diagram for the Heegaard decomposition
${\mathfrak {U}}$.  A pair-of-pants diagram for the stabilized
Heegaard decomposition can be given as follows. Consider a point $x\in
\Sigma$ which is an intersection of $\alpha_1\in \alphak $ and $\beta
_1\in \betak$. Consider a small isotopic translate $\alpha_1'$ (and
$\beta_1 '$) of $\alpha_1$ (and $\beta_1$, resp.)  such that
$\alpha_1 , \alpha '_1$ (and similarly $\beta_1, \beta '_1$) cobound
an annulus $A_{\alpha}$ (and $A_{\beta}$, resp.) in $\Sigma$. 
Stabilize the Heegaard decomposition
${\mathfrak {U}}$ in the elementary rectangle with boundaries $\alpha
_1, \beta_1, \alpha_1', \beta_1'$, containing the chosen $x$ on the
boundary. Add the curves $\alpha , \beta$ of the stabilizing torus and
a further pair $\alpha_1'', \beta_1 ''$ (as shown in
Figure~\ref{f:stab}) to the sets of curves $\alphak$ and $\betak$.
(Notice that the curves in $\alphak$ and $\betak$ can be naturally viewed
as curves in the stabilized Heegaard surface $\Sigma '$.)
\begin{figure}[ht]
\begin{center}
\epsfig{file=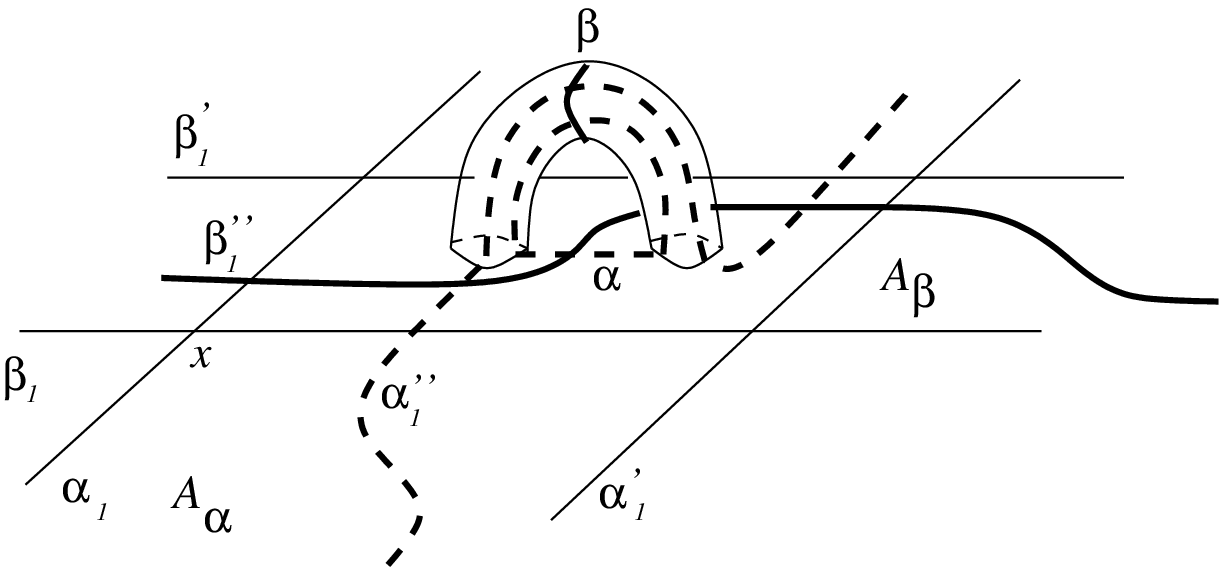, height=4cm}
\end{center}
\caption{{\bf The stabilization of an essential pair-of-pants Heegaard
    diagram.} We stabilize near the intersection point $x$ of $\alpha
  _1$ and $\beta _1$ and introduce a 2--dimensional 1--handle
  (increasing the Heegaard genus by 1), together with the additional
  curves $\alpha , \alpha _1, \alpha _1''$ and $\beta , \beta _1,
  \beta _1''$.}
\label{f:stab}
\end{figure}
\begin{lem}\label{l:stabil}
The procedure above gives an essential pair-of-pants Heegaard diagram
$(\Sigma ' , \alphak ', \betak ')$ for the stabilized Heegaard
decomposition.
\end{lem}
\begin{proof}
Consider components of $\Sigma - \alphak$ outside of the strip $A_{\alpha }$ between
$\alpha_1 $ and $\alpha_1'$. Those are obviously unchanged, hence
are still pairs-of-pants. In the annulus $A_{\alpha}$ between $\alpha_1$ and
$\alpha_1'$ we perform a connected sum operation with a torus
(turning the annulus into a twice punctured torus), cut open the torus
along its generating circles (getting a four--punctured sphere) and
finally introducing an $\alphak$--curve which partitions the
four--punctured sphere into two pairs-of-pants.  Similar argument
applies for the $\betak$--circles and $\betak$--components. The
argument also shows that if we start with a marking then the result of
this procedure will be a marking as well, concluding the proof.
\end{proof}
Notice also that if $(\Sigma , \alphak, \betak) $ was bigon--free then so is
the stabilized diagram $(\Sigma ', \alphak ', \betak ') $.

\section{Nice diagrams and nice moves}
\label{sec:second}
Suppose that ${\mathfrak {U}} =(\Sigma, U_0, U_1)$ is a genus--$g$
Heegaard decomposition of the 3--manifold $Y$, and let
$(\Sigma,\alphak,\betak)$ (with $\alphak = \{ \alpha_1, \ldots ,
\alpha_k \}$, $\betak =\{ \beta_1, \ldots , \beta _k \}$) be a
corresponding generalized Heegaard diagram. Choose furthermore a
$(k-g+1)$--tuple of points $\w =\{ w_1, \ldots , w_{k-g+1}\} \subset
\Sigma -\alphak -\betak$ with the property that each component of
$\Sigma - \alphak$ and each component of $\Sigma -\betak $ contains a
unique element of $\w$. (Notice that this assumption, in fact,
determines the cardinality of $\w$.)  Then $\DD = (\Sigma , \alphak , \betak
, \w )$ is called a \emph{multi-pointed Heegaard diagram}. Points of
$\w$ are called \emph{basepoints}; their number is denoted by 
$b(\DD )=\vert \w \vert$.  Generalizing the corresponding
definition of \cite{SW} to the case of multiple basepoints (in the
spirit of~\cite{OSzlinks}), we have

\begin{defn}
\label{def:Nice}
{\whelm The multi-pointed Heegaard diagram $\DD =(\Sigma , \alphak , \betak ,
  \w )$ is \emph{nice} if an elementary domain (a connected component
  of $\Sigma - \alphak - \betak$) which contains no basepoint is
  either a bigon or a rectangle.}
\end{defn}
According to one of the main results of \cite{SW}, any once-pointed Heegaard
diagram (i.e. a Heegaard diagram with exactly $g$ $\alphak$-- and
$\betak$--curves and hence with $\w =\{ w\}$) can be transformed by isotopies
and handle slides to a once-pointed nice diagram.  A useful lemma for
multi-pointed nice diagrams was proved in \cite{LMW}:
\begin{lem}(\cite[Lemma~3.1]{LMW}) \label{l:mindket}
Suppose that $(\Sigma , \alphak , \betak , \w )$ is a nice Heegaard diagram and
$\alpha_i \in \alphak$. Then there are elementary domains $D_1, D_2$,
both containing basepoints such that $\alpha_i \cap \partial D_1$ and
$\alpha_i \cap \partial D_2$ are both nonempty, and the orientation
induced by $D_1$ on $\alpha_i$ is opposite to the one induced by
$D_2$. (The domains $D_1, D_2$ get their orientation from the Heegaard
surface $\Sigma$.)  In short, $\alpha_i$ contains a basepoint on
either of its sides. \qed
\end{lem}

Next we describe three modifications, isotopies, handle slides and
stabilizations (with two types of the latter) which modify a nice
diagram in a manner that it remains nice. We discuss these moves in
the order listed above.

\noindent{\bf {Nice isotopies.}} An embedded arc in a Heegaard diagram
starting on an $\alphak$--circle (but otherwise disjoint from
$\alphak$), transverse to the $\betak $--circles, and ending in the
interior of a domain naturally defines an isotopy of the circle which
contains the starting point of the arc: apply a finger move along the
arc.  Special types of isotopies can be therefore defined by requiring
special properties of such arcs.
\begin{figure}[ht]
\begin{center}
\epsfig{file=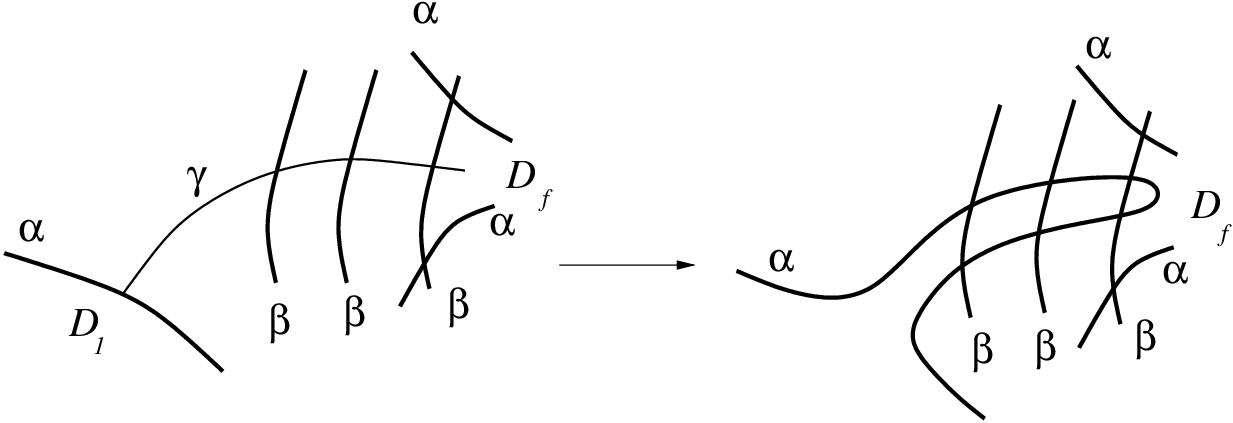, height=4cm}
\end{center}
\caption{{\bf Nice isotopy along the arc $\gamma$.}}
\label{f:niceiso}
\end{figure}

\begin{defn}
\label{def:NiceIsotopy}
{\whelm Suppose that $\DD = (\Sigma , \alphak , \betak , \w )$ is a nice
  diagram. We say that the embedded arc $\gamma=(\gamma (t))_{t\in
    [0,1]}$ is \emph{nice} if
\begin{itemize}
\item The starting point $\gamma (0)$ of $\gamma$ is on an
  $\alphak$--curve $\alpha$, while the endpoint $\gamma (1)$ is in the
  interior of the elementary domain $D_{f}$ which is either a bigon or
  a domain containing a basepoint;
\item $\gamma -\gamma (0)$ is disjoint from all the $\alphak$--curves,
  $\gamma$ intersects any $\betak$--curve transversally,
  and $\gamma$ is transverse to $\alpha$ at $\gamma (0)$;
\item the elementary domain $D_1$ containing $\gamma (0)$
  on its boundary, but not $\gamma (t)$ for small $t$, is either a
  bigon or it contains a basepoint;
\item for any elementary domain $D$, at most one component of
  $D-\gamma$ is not a rectangle or a bigon, and if there is such a
  component, it  contains a basepoint;
\item the component of $D_f-\gamma $ containing $\gamma (1)$
is either a bigon, or it contains a basepoint, and finally
\item  if $D_{1}=D_{f}$ then we assume that the component of 
$D_1-\gamma $ containing $\gamma (1)$ also contains a 
basepoint. 
\end{itemize}
An isotopy defined by a nice arc is called a \emph{nice isotopy}.}
\end{defn}

\noindent{\bf {Nice handle slides.}}  Recall that in a Heegaard
diagram a handle slide of the curve $\alpha_1$ over $\alpha_2$ can be
specified by an embedded arc $\delta$ with one endpoint on $\alpha_1$,
the other on $\alpha_2$ and with the property that $\delta$ (away from
its endpoints) is disjoint from all the $\alphak$--curves.  The result
of sliding $\alpha _1$ over $\alpha _2$ along $\delta$ is a pair of
curves $(\alpha _1', \alpha _2)$, where $\alpha _1'$ is the connected
sum of $\alpha _1$ and $\alpha _2$ along $\delta$,
cf. Figure~\ref{f:slide}.
\begin{figure}[ht]
\begin{center}
\epsfig{file=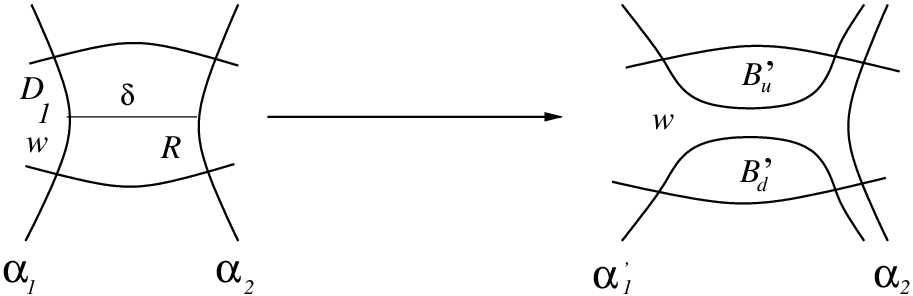, height=4cm}
\end{center}
\caption{{\bf Nice handle slide along the arc $\delta$.}}
\label{f:slide}
\end{figure}

\begin{defn}
{\whelm Suppose that $\DD =(\Sigma , \alphak , \betak , \w )$ is a nice
  diagram. We say that the embedded arc $\delta$ defines a \emph{nice
    handle slide} if the interior of $\delta$ is contained in a single elementary
  rectangle $R$, and the other elementary domain $D_{1}$
  containing $\delta (0)$ on its boundary contains a basepoint.}
\end{defn}

\noindent{\bf {Nice stabilizations.}}  Suppose that $\DD = (\Sigma , \alphak ,
\betak , \w )$ is a nice diagram.  There are two types of stabilizations of
the diagram: type-$b$ stabilizations do not change the Heegaard surface
$\Sigma$, but increase the number of $\alphak$-- and $\betak$--curves, and
also increase the number of basepoints, while type-$g$ stabilizations increase
the genus of $\Sigma$ and the number of $\alphak$-- and $\betak$--curves, but
keep the number of basepoints fixed. In the following we will describe both
types of stabilizations.

We start with the description of nice type-$b$ stabilizations.
Suppose that $D$ is an elementary domain of the diagram $\DD$, which
contains a basepoint $w$. Suppose furthermore that $\alpha ' , \beta
'\subset D$ are embedded, homotopically trivial circles, bounding the
disks $D_{\alpha '}, D_{\beta '}$ respectively, and intersecting each
other in exactly two points.  Assume that the disks $D_{\alpha '}, D
_{\beta '}$ are disjoint from the basepoint of $D$ and consider a new basepoint
$w'\in D_{\alpha '}\cap D_{\beta '}$, cf. Figure~\ref{f:nsta}.
\begin{figure}[ht]
\begin{center}
\epsfig{file=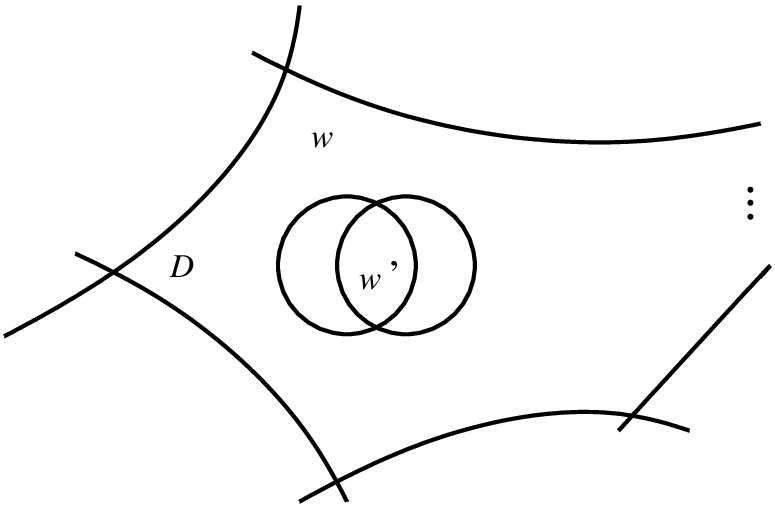, height=4cm}
\end{center}
\caption{{\bf Nice type-$b$ stabilization in the domain $D$ containing
    the basepoint $w$.}}
\label{f:nsta}
\end{figure}

\begin{defn}
{\whelm The multi-pointed Heegaard diagram $\DD '= (\Sigma , \alphak \cup
  \{ \alpha '\}, \betak \cup \{ \beta '\}, \w \cup \{ w'\} )$ is
  called a \emph{nice type-$b$ stabilization} of $\DD =(\Sigma ,
  \alphak , \betak , \w )$.  Conversely, $(\Sigma , \alphak , \betak ,
  \w )$ is a \emph{nice type-$b$ destabilization} of $\DD '$.}
\end{defn}

Suppose now that ${\mathcal {T}}=(T^2, \alpha , \beta )$ is the
standard toric Heegaard diagram of $S^3$, that is, the Heegaard surface is
a genus-1 surface and $\alpha, \beta$ form a pair of simple closed curves
intersecting each other transversely in a single point. 
\begin{defn} {\whelm The connected sum of ${\mathcal {T}}$ with $\DD = (\Sigma
    , \alphak , \betak ,\w )$, performed in a point of $T^2-\alpha -\beta$ and
    in an interior point of an elementary domain $D$ of $\DD$ containing a
    basepoint $w$ is called a \emph{nice type-$g$ stabilization} of $\DD$,
    cf. Figure~\ref{f:nstag}. The inverse of this operation is called a
    \emph{nice type-$g$ destabilization}.}
\end{defn}
\begin{figure}[ht]
\begin{center}
\epsfig{file=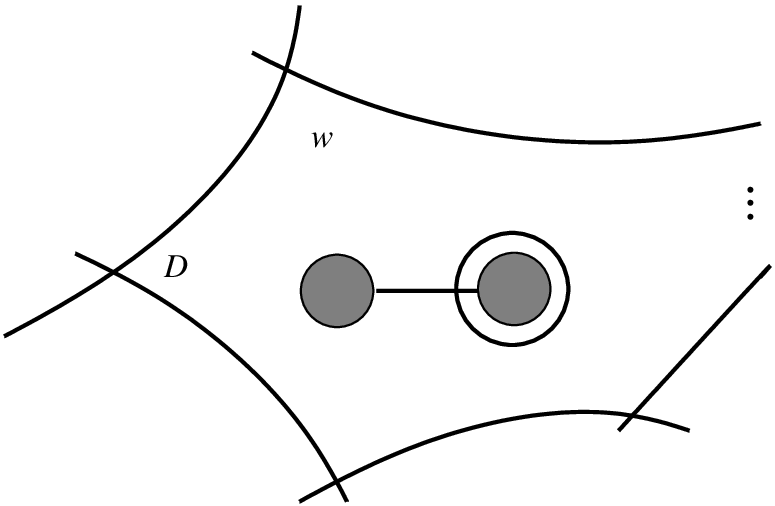, height=4cm}
\end{center}
\caption{{\bf Nice type-$g$ stabilization in the domain $D$.} The two
  full circles indicate the feet of the 1-handle we add to $\Sigma$,
  the contour of one of which is parallel to the new $\alphak$--curve,
  while the interval joining the two disks (which becomes a circle
  when completed in the 1-handle) is the new $\betak$--curve.}
\label{f:nstag}
\end{figure}

The expression ``nice stabilization'' will refer to either of the
above types. 

\begin{rem} {\whelm The two types of nice stabilizations can be regarded as
    taking the connected sum of the multi-pointed Heegaard diagram $\DD$ with
    the diagrams (a) (for type-$b$ stabilization) and (b) (for type-$g$
    stabilization) of Figure~\ref{f:hd}, depicting two diagrams for $S^3$.
    Since we take the connected sum in a domain $D$ containing a basepoint
    $w$, one of the basepoints of Figure~\ref{f:hd} ($w_2$ for (a) and $w$ for
    (b)) should be eliminated.}
\end{rem}
\begin{figure}[ht]
  \begin{center}
    \input{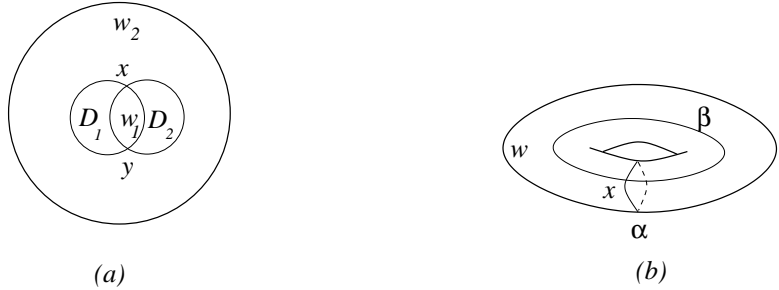}
  \end{center}
  \caption{{\bf Two Heegaard diagrams of $S^3$.} The left diagram is
      a spherical Heegaard diagram for $S^3$, with a single $\alpha$--
      and a single $\beta$--curve and two basepoints. The diagram on
      the right is the standard toroidal Heegaard diagram of $S^3$
      with one basepoint.}
    \label{f:hd}
    \end{figure}

In the sequel a \emph{nice move} will mean either a nice isotopy, a
nice handle slide or a nice stabilization/destabilization. It is an
elementary fact that the result of a nice move on a multi-pointed
Heegaard diagram of a 3--manifold $Y$ is also a multi-pointed Heegaard
diagram of $Y$.

\begin{thm}
Suppose that $\DD '=(\Sigma , \alphak ', \betak ', \w ')$ is given by
a nice move on the nice diagram $\DD =(\Sigma , \alphak , \betak , \w)$.
Then $(\Sigma , \alphak ', \betak ', \w ')$ is nice, in the sense
of Definition~\ref{def:Nice}.
\end{thm}
\begin{proof}
The result of a nice move is a multi-pointed Heegaard diagram, so we
need to check only that $\DD'$ is nice, i.e. if an elementary domain
contains no basepoint then it is either a bigon or a rectangle.

Consider a nice isotopy first. For a domain disjoint from the nice arc
$\gamma$, the shape of the domain remains intact. Similarly, if a
domain does not contain $\gamma (0)$ or $\gamma (1)$ then $\gamma$
splits off bigons and/or rectangles, and (by our assumption) a component
which is not bigon or rectangle, which must contain a basepoint. Finally our
assumptions on the domains $D_{1}$ and $D_{f}$ ensure that the
resulting diagram is nice.

Suppose now that we perform a nice handle slide of $\alpha_1$ over $\alpha_2$.
First consider the diagram which is identical to the nice diagram $\DD$ we started
with, except we replace $\alpha_1$ by a new curve $\alpha_1'$ which is the
connected sum of $\alpha_1$ with $\alpha_2$ along $\delta$. To get the diagram
$\DD'$ (which is the result of the handle slide), we need to add a small
isotopic translate of $\alpha _2$ (still denoted by $\alpha _2$) to this
diagram.  The curves $\alpha_1'$, $\alpha_1$, and $\alpha_2$ bound a
pair-of-pants in the Heegaard surface.  (Notice that $\alpha _1$ is not in the
diagram $\DD ' $.)  The diagram $\DD '$ has a collection of elementary domains
which are rectangles, supported in the region between $\alpha_1'$ and
$\alpha_2$. There are also two bigons $B_u$ and $B_d$ in the new diagram,
which are contained in the rectangle containing (in the old diagram $\DD $)
the arc $\delta$, cf. Figure~\ref{f:slide}. There is a natural one-to-one
correspondence between all other elementary domains in the diagram before and
after the handle slide.  The domain $D_{1}$ in the original diagram $\DD$
acquires four additional corners in the new diagram; all other domains have
the same combinatorial shape before and after the handle slide.  Since $D_{1}$
contains a basepoint, the new diagram $\DD '$ is nice as well. See
Figure~\ref{f:slide} for an illustration.

Finally, a nice type-$b$ stabilization introduces three new bigons
(one of which is with the basepoint $w'$) and changes $D$ only. Since
$D$ contains a basepoint, the resulting diagram is obviously nice.  A
nice type-$g$ stabilization changes only the domain $D$, hence if we
start with a nice diagram, the fact that $D$ contains a basepoint
implies that the result will be nice, concluding the proof.
\end{proof}

\section{Convenient diagrams}
\label{sec:cdiagdef}
Suppose now that $(\Sigma , \alphak , \betak )$ is an essential
pair-of-pants diagram of a 3--manifold $Y$ which contains no $S^1\times
S^2$--summand. In the following we will give an algorithm which
provides a nice diagram from $(\Sigma , \alphak , \betak )$. Any
output of this algorithm will be called a \emph{convenient}
diagram. (The algorithm will require certain choices, and depending on
these choices we will have $\alphak$--, $\betak$-- and symmetric
convenient diagrams.) The algorithm involves seven steps, which we
spell out in detail below.

\begin{algorithm} \label{algo:alg}
{\whelm
The following algorithm provides a nice multi-pointed Heegaard
diagram from an essential pair-of-pants diagram $(\Sigma , \alphak , \betak )$
of a 3--manifold which has no $S^1\times S^2$--summand.

\noindent {\bf {Step 1}} \quad Apply an isotopy on $\betak$ to get the
bigon--free model of $(\Sigma , \alphak , \betak )$.  Recall that by
Theorem~\ref{thm:nobigon} the resulting diagram is unique (up to homeomorphism).
We will henceforth use the notation $(\Sigma ,
\alphak , \betak )$ to denote this bigon-free model.

\noindent {\bf {Step 2}} \quad
Choose one of the curve systems $\alphak$ or $\betak$. Depending on the
choice here, the result of the algorithm will be called
$\alphak$--convenient or $\betak$--convenient. To ease notation, we will
assume that we chose the $\alphak$--curves; for the other choice
the subsequent steps must be modified accordingly.

\noindent {\bf {Step 3}} \quad
Put one basepoint into the interior of each hexagon, and two into
the interior of each octagon of $(\Sigma , \alphak , \betak )$.
Notice that in this way in each component of $\Sigma - \alphak$ (and of 
$\Sigma - \betak$) there will be two basepoints.

\noindent {\bf {Step 4}} \quad Consider a component $P$ of $\Sigma - \alphak$.
Denoting parallel $\betak$--curves in $P$ with a single interval, the
resulting digaram (after a suitable diffeomorphism of $P$) is one of the
diagrams shown in Figure~\ref{f:poss}, together with the two basepoints chosen
above.  In case (i) connect the two basepoints with an oriented arc $a_P$
which crosses each of the vertical $\betak$--arcs once and is disjoint from
all other curves in $P$.  (The orientation of $a_P$ can be chosen arbitrarily.
As we will see, the resulting convenient diagram will depend on the chosen
orientation of $a_P$. For an example, see Figure~\ref{f:new}(a).) 
In case (iii) connect the two basepoints with an
oriented arc $a_P$ which intersects the $\betak$--arcs indicated by one of
the horizontal arcs of Figure~\ref{f:poss}(iii), and which is disjoint from the
$\betak$--arcs corresponding to the other two horizontal arcs, cf.
Figure~\ref{f:new}(b).  Notice that for the arc therefore we have three possible
choices; so, when taking possible orientations into account, altogether we have
six choices in this case for $a_P$.
\begin{figure}[ht]
    \begin{center}
\epsfig{file=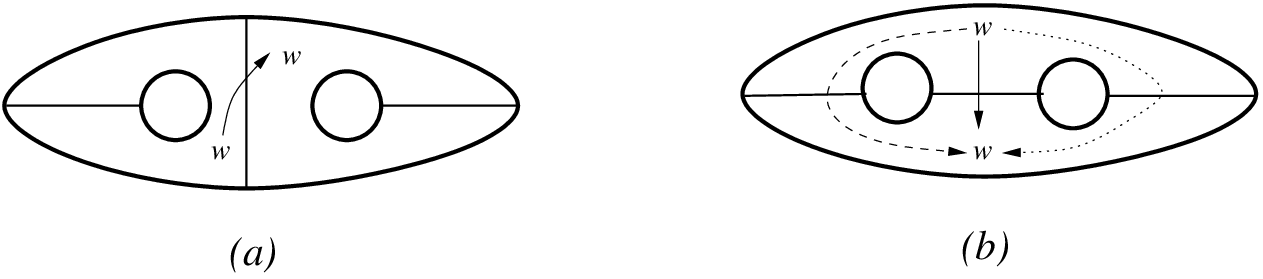, height=2cm}
    \end{center}
    \caption{{\bf Possible oriented arcs.} The left diagram shows the
only possible choice for $a_P$ with one of its possible orientatons.
In the right we show the three possible arcs, with one of their possible
orientations.}
    \label{f:new}
    \end{figure}
Now, for each 
component $P_j$ of $\Sigma - \alphak$ containing hexagons we fix an 
oriented arc $a_{P_j}$ as above.

\noindent {\bf {Step 5}} \quad
Choose a similar set of oriented arcs
$b_{Q_i}$ for the basepoints, now using the components $Q_i$ of
$\Sigma - \betak$.

\noindent {\bf {Step 6}} \quad
Add a new $\alphak$--curve in each pair-of-pants
component of $\Sigma - \alphak$ as indicated by the dashed curves of 
Figure~\ref{f:add}.
\begin{figure}[ht]
\begin{center}
\epsfig{file=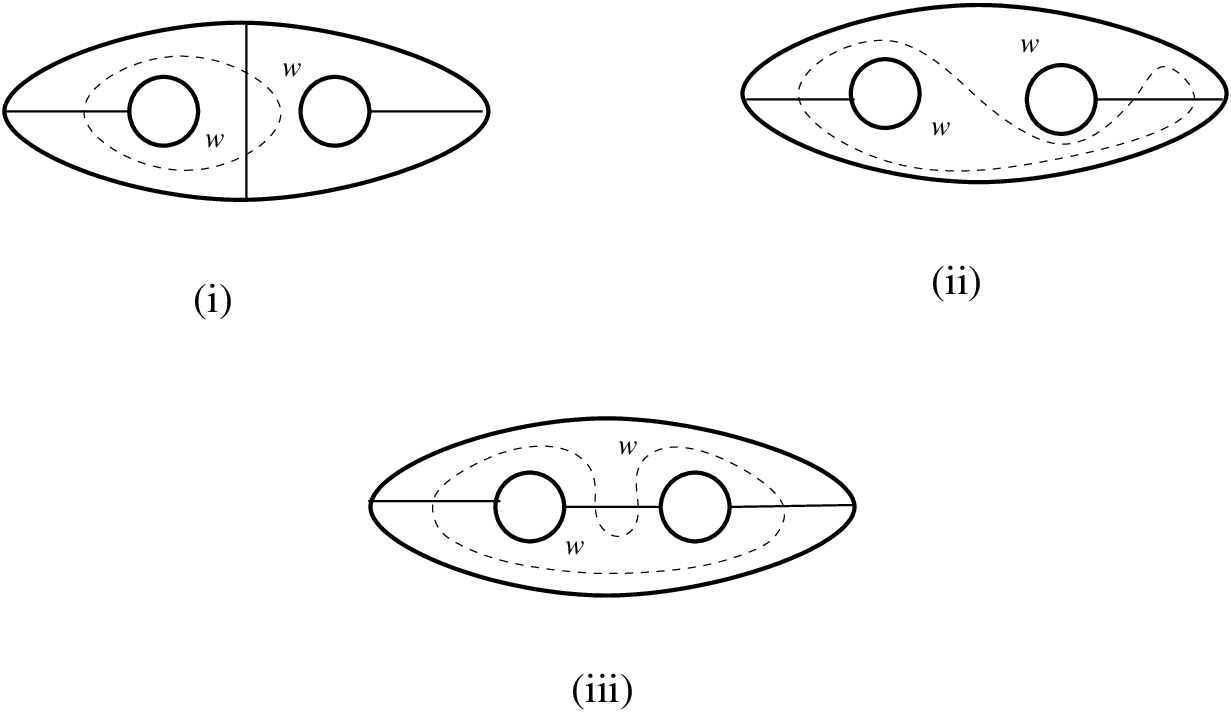, height=6cm}
\end{center}
\caption{{\bf Addition of the new curves separating the two basepoints in a
  pair-of-pants.} The basepoints are denoted by $w$.}
\label{f:add}
\end{figure}
The bigons in Figures~\ref{f:add}(i) and (iii) are placed in the
hexagon pointed into by the chosen oriented arc $a_P$, and in (iii) the
bigon rests on the $\betak$--curve which is intersected by $a_P$.
Although in the situation depicted in (ii) we also have a number of
choices, we do not record them by choosing an arc. Notice that adding
a curve as shown in (ii) in a pair-of-pants containing an octagon, we
cut it into a hexagon, an octagon, a rectangle and a bigon
(and some further rectangles between the parallel $\betak$--curves indicated
by a single arc in the diagram). The union
of the set $\alphak$ with the chosen new curves (a collection of
$5g(\Sigma )-5$ curves altogether) will be denoted by $\alphak ^c$.

\noindent {\bf {Step 7}} \quad Consider now a component $Q$ of $\Sigma
- \betak$. The intersection of $Q$ with $\alphak$ still falls into the
three categories shown by Figure~\ref{f:poss} (after a suitable
diffeomorphism has been applied).  After adding the new
$\alphak$--curves, the patterns slightly change. The diagrams might
contain bigons, and, when disregarding the bigons, we will have
diagrams only of the shape of (i) and (iii) of Figure~\ref{f:poss}
(since after disregarding bigons there is no elementary domain which
is an octagon). For the components where $Q\cap \alphak $ looked like
(i) or (iii) choose the new $\betak$--curve dictated by the chosen
arcs $b_Q$, while in those domains where $Q\cap \alphak$ is of (ii)
(and then $Q\cap \alphak ^c$, after disregarding the bigons, became
(i) or (iii)) we make further choices of oriented arcs and add the new
$\betak$--curves accordingly. We assume that the bigons in the
diagrams are very narrow and almost reach the basepoints --- this
convention helps deciding the intersection patterns between the bigons
and the newly chosen curves. Like before, the completion of $\betak$
with the above choices will be denoted by $\betak ^c$.}
\end{algorithm}

\begin{defn}\label{def:conv} 
{\whelm The resulting multi-pointed diagram $\DD = (\Sigma , \alphak ^c ,
  \betak ^c, \w )$ with $\vert \w \vert =4g(\Sigma ) -4$ and
$\vert \alphak ^c \vert=\vert \betak ^c \vert = 5g(\Sigma ) -5$ will be
called a \emph{convenient} diagram; depending on the choice 
made in Step 2, we call the diagram $\alphak$-- or $\betak$--convenient.}
\end{defn}

A simple variation of Algorithm~\ref{algo:alg} provides a
\emph{symmetric convenient diagram} as follows: skip Step 2, add only
one basepoint to an octagon in Step 3 and then apply Steps 4--7
modified so that in the components of $\Sigma -\alphak$ (and of
$\Sigma - \betak$) described by Figure~\ref{f:poss}(ii) no new curves
are added. The number of basepoints and the number of curves in a symmetric
convenient diagram therefore depend on the genus of the Heegaard
surface \emph{and} the number of octagons in the bigon--free model. An
example of a symmetric convenient diagram with no hexagons was
discussed (and was called \emph{adapted}) in \cite{U2}.

\begin{prop}\label{p:nicemarad}
Any $\alphak$--convenient ($\betak$--convenient or symmetric
convenient) Heegaard diagram is nice.  
\end{prop}
\begin{proof}
We only need to check that after adding both the new $\alphak$-- and
$\betak$--curves we do not create any further $2n$--gons with $n>2$
than the ones containing the basepoints. It is obvious from the
construction that all basepoints will be in different components, and
any $\alphak$-- (and similarly $\betak$--) component contains a
basepoint.

When adding the new curve in the situation of Figure~\ref{f:add}(i), we have
two choices, as encoded by the orientation of the arc connecting the
two basepoints, pointing towards the region where the bigon is created.
(Notice that the boundary circles of a pair-of-pants in (i) are not symmetric:
one of the components is distinguished by the property that it is intersected
by the same $\betak$--arc twice.) We will label the corresponding oriented arc
with (i).  In the case of Figure~\ref{f:add}(ii) there are four possibilities,
according to which boundary the newly added circle is isotopic to (when the
basepoints are disregaded) and from which side it places the bigon.  Since the
modification of (ii) does not affect any other elementary $2n$--gon with $n>2$
besides the octagon we started with, the choice here will be irrelevant as far
as the combinatorics of the other domains go, and (as in the algorithm) we do
not record the choices made. For the case of Figure~\ref{f:add}(iii) there are
six choices, also indicated by an oriented arc connecting the two basepoints.
These oriented arcs will be decorated by (iii).

Now for a given hexagon we must choose from these possibilities for
both the $\alphak$-- and the $\betak$--curves. This amounts to
examining the changes on a hexagon with two oriented arcs pointing to
(or from) the basepoint in the middle of the hexagon. The two oriented
arcs (one corresponding to the fact that the hexagon is in an
$\alphak$-component, the other that it is in a $\betak$-component)
intersect either neighbouring, or opposite sides of the hexagon, and
either can be a type (i) or type (iii), and can point in or out.
Figure~\ref{f:vec} shows the modification of the hexagon in each case.
By drawing all possibilities for the two oriented arcs (taking
symmetries and identities into account, there are 10 of them),
and picturing the result on a given hexagon in 
Figure~\ref{f:eset}, the proof of the proposition is complete.
\begin{figure}[ht]
\begin{center}
\epsfig{file=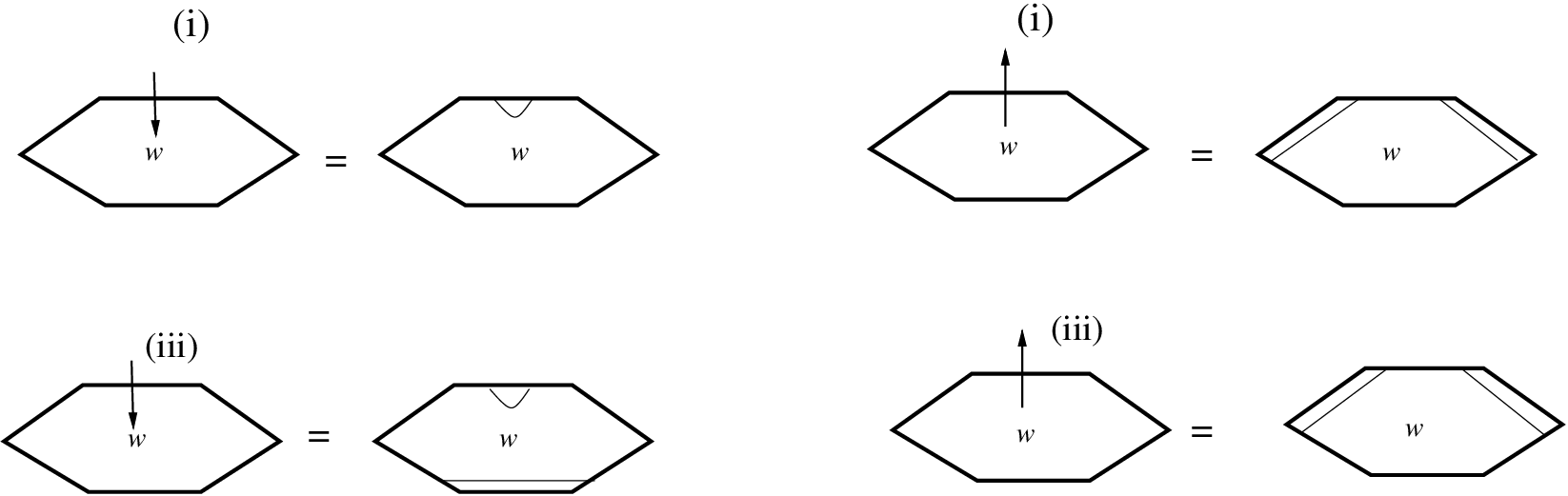, height=4cm}
\end{center}
\caption{{\bf Possible oriented arcs and the effect of adding a new curve 
in a hexagon, dictated by the oriented arc.}}
\label{f:vec}
\end{figure}
\end{proof}

\begin{figure}[ht]
\begin{center}
\epsfig{file=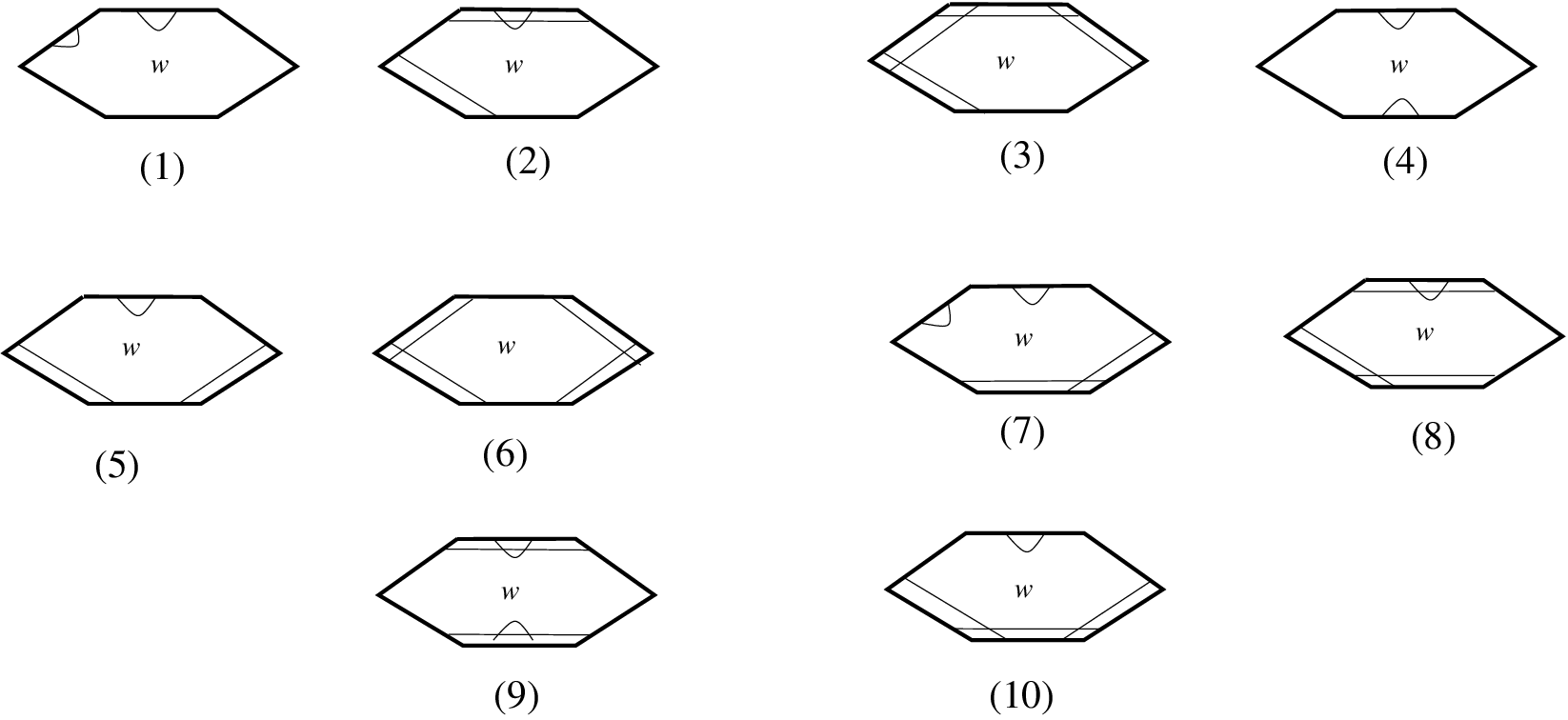, height=5cm}
\end{center}
\caption{{\bf The list of all the cases in the proof of
  Proposition~\ref{p:nicemarad}.} In (1), (2), (3), (7), (8) the
  oriented arcs intersect the top horizontal and the upper left interval,
  while in (4), (5), (6), (9), (10) the top and bottom
  horizontals. For out-pointing oriented arcs (i) and (iii) has the same
  effect.  In (1) the two oriented arcs are both (i) and point in, in (2)
  the vertical points in and is (i), the other points out, in (3) both
  point out. In (4) both point in and are (i), in (5) the top one
  points in and is (i), the other one points out, while in (6) both
  point out. (7), (8), (9) and (10) are the modifications of (1), (2),
  (4) and (5) by replacing the type (i) oriented arc with the type (iii)
  having the same direction.}
\label{f:eset}
\end{figure}

We will define the Heegaard Floer chain complexes (determining the
stable Heegaard Floer invariants) using combinatorial properties of
convenient diagrams. Since in Algorithm~\ref{algo:alg} there are a
number of steps which involve choices (recall that the algorithm
itself starts with the choice of an essential pair-of-pants diagram for
$Y$), it will be crucial for us to relate the results of various
choices. The relations will be discussed in the next section.

\begin{rem} {\whelm There are further possible choices for the dashed
curves to turn an essential pair-of-pants diagram into a nice one.
For example, Figure~\ref{f:megegy} shows an alternate configuration
instead of Figure~\ref{f:add}(i). Although such variants will be used in our
later arguments, in the definition of convenient diagrams we chose
the curve given by Figure~\ref{f:add}(i), since this choice led to the
least number of possibilities to examine.}
\end{rem}
\begin{figure}[ht]
\begin{center}
\epsfig{file=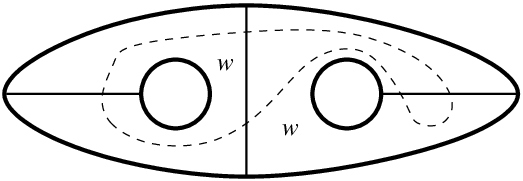, height=2cm}
\end{center}
\caption{{\bf An alternate curve for Figure~\ref{f:add}(i).}} 
\label{f:megegy}
\end{figure}

\section{Convenient diagrams and nice moves}
\label{sec:convdia}

The aim of the present section is to show that convenient diagrams of a 
fixed 3--manifold can be connected by nice moves. 
In order to state the main theorem of the section, we need a definition.
\begin{defn}
{\whelm Suppose that $\DD_1, \DD_2$ are given nice diagrams of a 3--manifold
$Y$. We say that $\DD_1$ and $\DD_2$ are \emph{nicely connected}
if there is a sequence $(\DD ^{(i)})_{i=1}^n$ of nice diagrams 
all presenting the same $3$--manifold $Y$
such that
\begin{itemize}
\item $\DD_1=\DD ^{(1)}$ and $\DD_2=\DD ^{(n)}$, and
\item consecutive elements $\DD ^{(i)}$ and $\DD ^{(i+1)}$ of the sequence
differ by a nice move.
\end{itemize}}
\end{defn}
It is a simple exercise to verify that being nicely connected is an
equivalence relation among nice diagrams representing a fixed
3--manifold $Y$. With the above terminology in place, in this section
we will show
\begin{thm} \label{thm:MainConv} 
Suppose that $Y$ is a given 3--manifold which contains no $S^1\times
S^2$--summand.  Suppose that $\DD_i$ $(i=1,2)$ are convenient
diagrams derived from essential pair-of-pants diagrams for $Y$.
Then $\DD_1$ and $\DD_2$ are nicely connected.
\end{thm}

\begin{rem}
{\whelm Notice that the diagrams in the path $(\DD ^{(i)})_{i=1}^n$ 
connecting the two given convenient diagrams $\DD _1$ and $\DD _2$ 
are all nice, but not necessarily convenient for $i\neq 1,n$.}
\end{rem}

\subsection{Convenient diagrams corresponding to a fixed pair-of-pants diagram}
\label{ssec:fix}

In this subsection, we wish to show that any two convenient diagrams
belonging to a fixed essential pair-of-pants diagram can be nicely connected.
We start by relating
the $\alphak$--, $\betak$-- and symmetric convenient Heegaard diagrams
corresponding to the same pair-of-pants diagram and the same choice of
oriented arcs.

\begin{prop}\label{p:symmlesz}
Suppose that $\DD_1$ is an $\alphak$--convenient Heegaard diagram.  Let
$\DD_2$ denote the symmetric convenient diagram corresponding to the
same pair-of-pants diagram and the same choice of oriented arcs fixed
in Steps 4 and 5 of Algorithm~\ref{algo:alg}. Then $\DD_1$ and $\DD_2$
are nicely connected.
\end{prop}
\begin{proof} Let us fix an octagon of the bigon--free pair-of-pants
  decomposition underlying the convenient diagrams.  We only need to work in
  the respective $\alphak$-- or $\betak$--pair-of-pants containing this fixed
  octagon. To visualize the octagon better (and to indicate that arcs
  correspond to potentially more than one parallel segments), now we use two
  parallel $\betak$-- (or $\alphak$--) curves from the bunch intersecting the
  pair-of-pants. In Figure~\ref{f:ii}(a) we show an $\alphak$--pair-of-pants
  (that is, the circles are all $\alphak$--curves, the newly chosen one being
  dashed, while the intervals denote the $\betak$--components in this
  pair-of-pants, and the dotted lines correspond to the new $\betak$--curve).
  Figure~\ref{f:ii}(b) shows a possible configuration in the
  $\betak$--pair-of-pants containing the same octagon (again, the new
  $\alphak$--curve is dashed while the new $\betak$--curve is dotted). Now the
  sequence of nice isotopies and nice handle slides on both the $\alphak$--
  and the $\betak$--curves, as indicated by the diagrams of
  Figure~\ref{f:ii} (showing the effect of the nice moves only in the
  $\alphak$--pair-of-pants), transforms the diagram into Figure~\ref{f:ii}(g).
  From here, a nice type-$b$ destabilization (for each octagon) provides a
  symmetric convenient diagram (depicted by Figure~\ref{f:ii}(h)).
\begin{figure}[ht]
\begin{center}
\epsfig{file=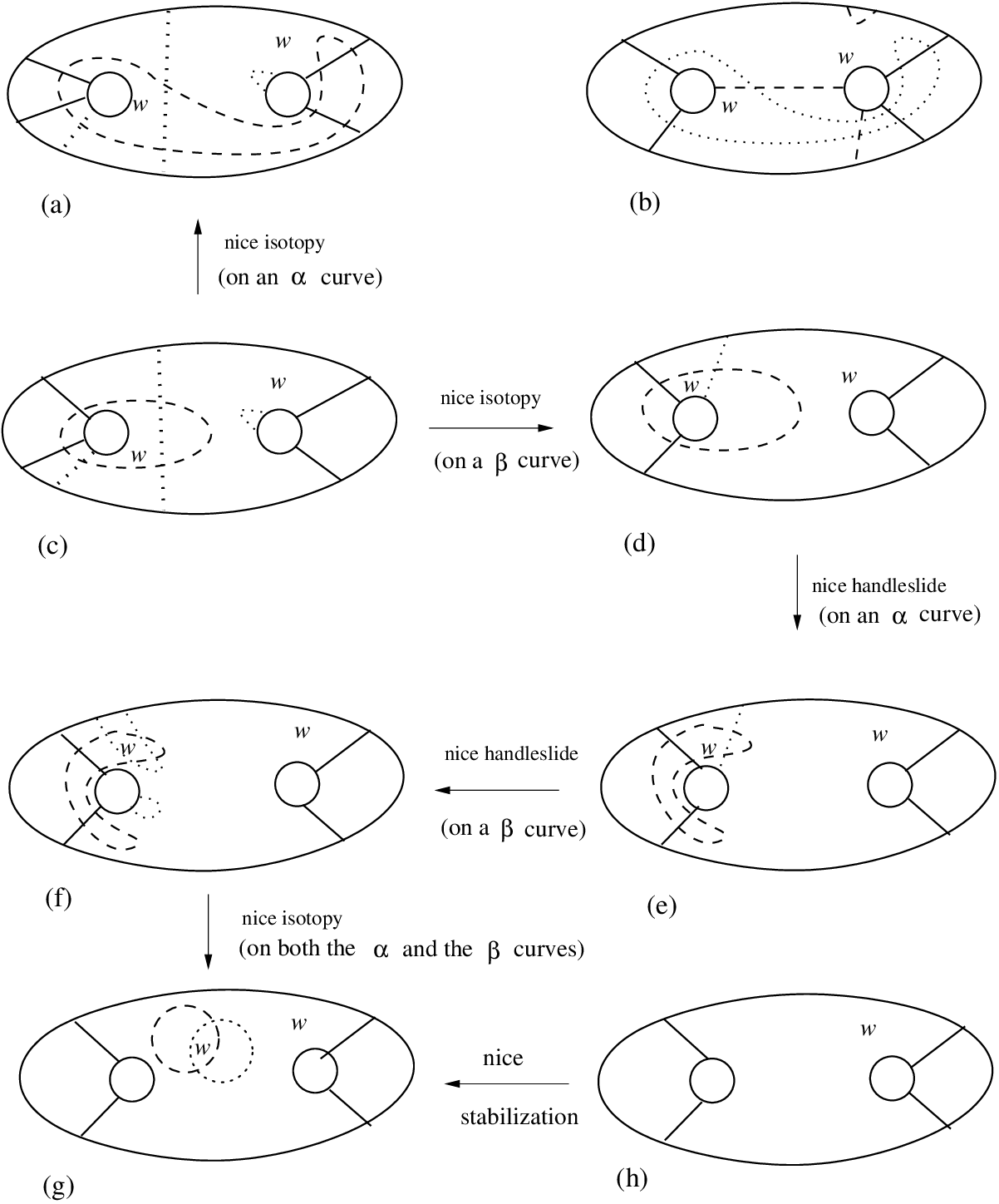, height=10cm}
\end{center}
\caption{{\bf Isotopies and handle slides for showing that
    $\alphak$--convenient and symmetric convenient diagrams are nicely
    connected.}  Diagram (b) shows the $\betak$--pair-of-pants in the
  starting Heegaard diagram, all other diagrams are depicting the
  $\alphak$--pair-of-pants. Nice moves are indicated between
  consecutive diagrams. The new $\alphak$--curve appear as dashed,
  while the new $\betak$--curve as dotted segment(s).}
\label{f:ii}
\end{figure}
\end{proof}

In view of the above result, when studying which diagrams can be
nicely connected, it is no longer necessary to specify if a diagram is
$\alphak$--convenient, $\betak$--convenient or symmetric. So we will
typically drop this quantifier from the notation, and refer simply to
{\em convenient} diagrams.

Next we will analyze the connection between convenient diagrams
corresponding to a fixed pair-of-pants decomposition of $Y$.
\begin{thm}\label{t:convconn}
Suppose that two convenient Heegaard diagrams $\DD_1$ and $\DD_2$
are derived from the same essential pair-of-pants 
diagram of a 3--manifold $Y$ which contains no $S^1\times
S^2$--summand. Then the convenient diagrams are nicely connected.
\end{thm}
\begin{proof}
According to Proposition~\ref{p:symmlesz}, we need to relate symmetric
convenient Heegaard diagrams only. According to
Algorithm~\ref{algo:alg}, the two symmetric diagrams differ by the
different choices of the oriented arcs connecting the basepoints
sharing the same ($\alphak$-- or $\betak$--) pair-of-pants
components. Since the choice of these arcs is independent from each
other, we only need to examine the case of changing one choice in one
single pair-of-pants.  The proof will rely on giving the sequence of
diagrams, differing by nice moves, connecting the two different
choices. Since we can work locally in a single pair-of-pants, these
diagrams will not be very complicated. To simplify matters even more,
we will follow the convention that bigons are omitted from the
diagrams. Once again, we always imagine that bigons are very thin and
almost reach the basepoint which is in the domain. Since nice moves
cannot cross basepoints, the addition of these bigons will still keep
niceness.

For the case of Figure~\ref{f:add}(i) we need to specify only the
direction of the oriented arc. As Figure~\ref{f:indep1} shows, the two
choices can be connected by a nice handle slide and a nice
isotopy. 
\begin{figure}[ht]
\begin{center}
\epsfig{file=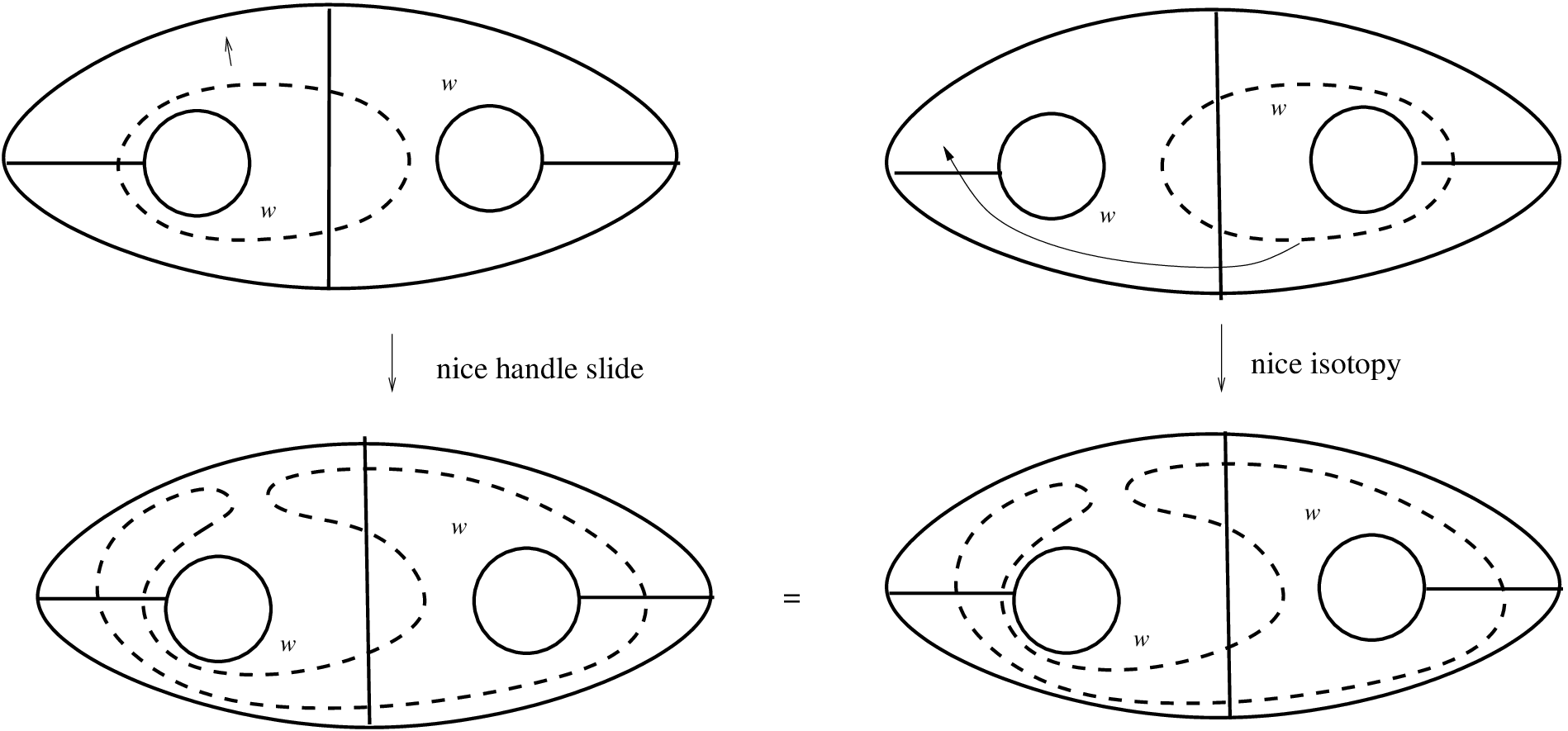, height=5cm}
\end{center}
\caption{{\bf Connecting different choices by nice moves for the
    configuration in Figure~\ref{f:add}(i).}  In this case the
  difference between the chosen oriented arcs $a_P$ and $a'_P$
  providing the upper diagrams is in their orientations.}
\label{f:indep1}
\end{figure}
In the case depicted by Figure~\ref{f:add}(iii) we need to consider
the change of the oriented arc and the change of its direction. We can
deal with the two cases separately; and as Figures~\ref{f:indep2} and
\ref{f:indep3} show, these changes can be achieved by nice isotopies
and nice handle slides.  
\begin{figure}[ht]
\begin{center}
\epsfig{file=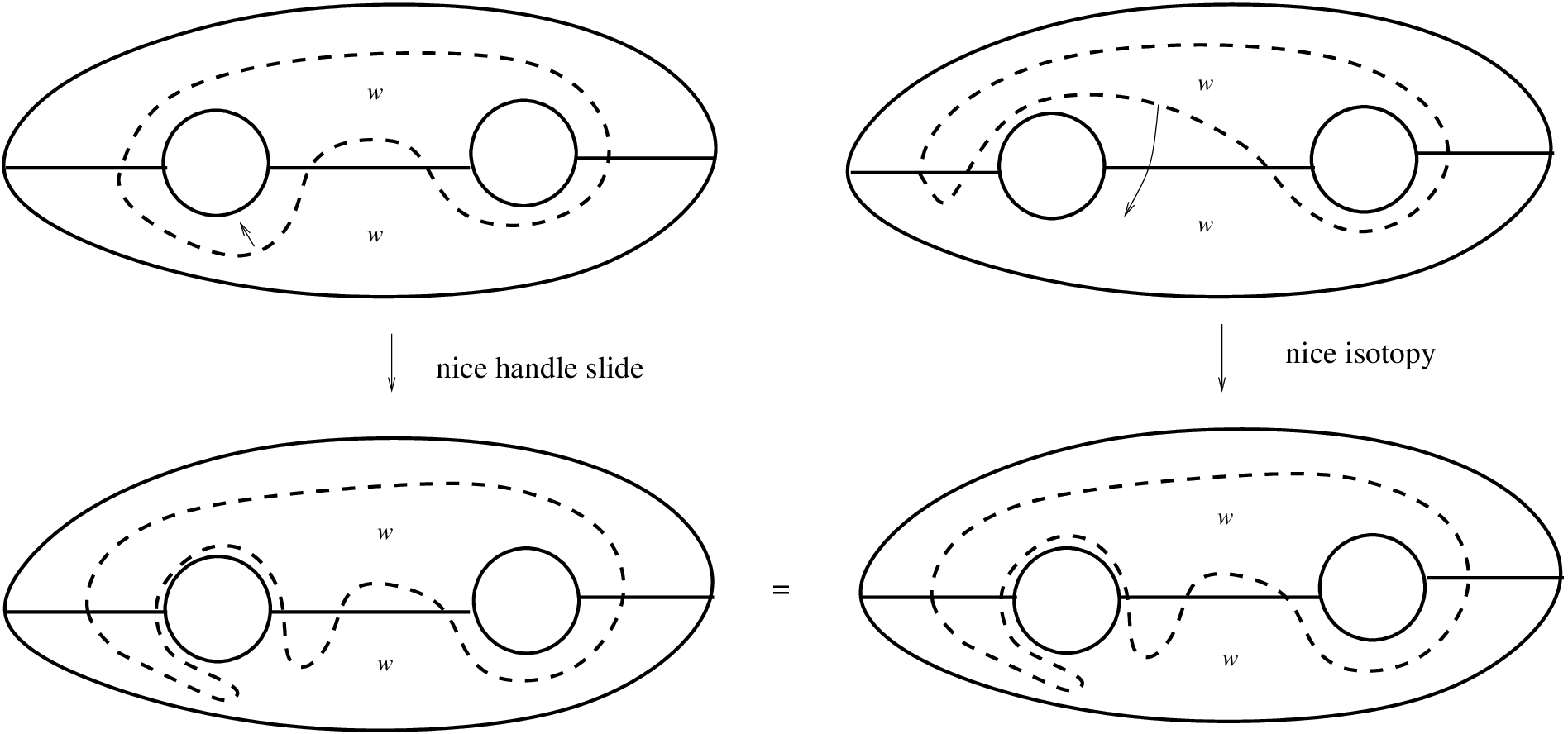, height=5cm}
\end{center}
\caption{{\bf Connecting different choices by nice moves for the
    configuration in Figure~\ref{f:add}(iii).}The upper diagrams correspond
to choosing the two arcs $a_P$ and $a_P'$ in such a way that they intersect
different horizontal arcs of Figure~\ref{f:poss}(iii).}
\label{f:indep2}
\end{figure}
\begin{figure}[ht]
\begin{center}
\epsfig{file=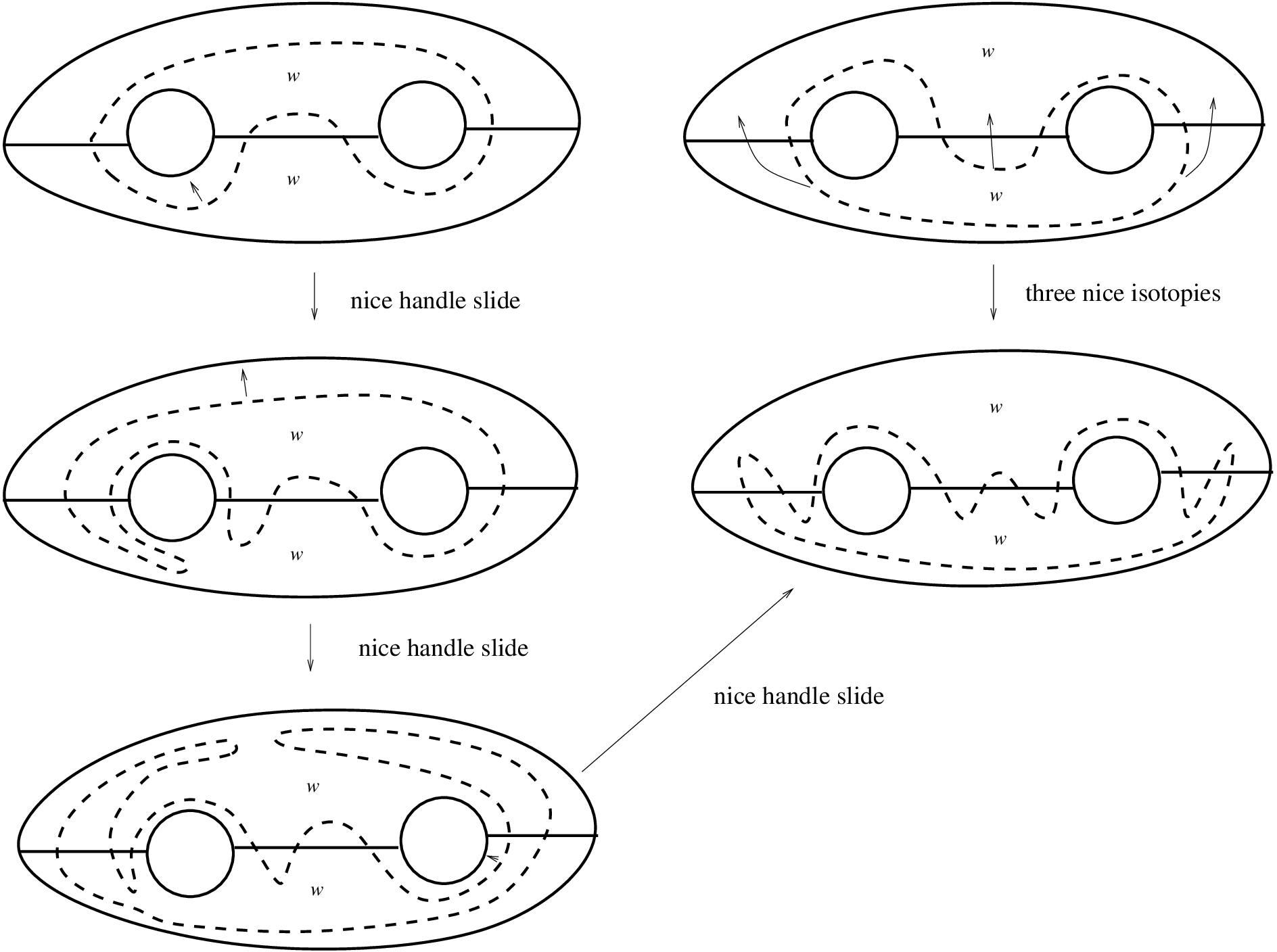, height=7cm}
\end{center}
\caption{{\bf Connecting different choices by nice moves for the
    configuration in Figure~\ref{f:add}(iii).} The two upper diagrams correspond to the choice
    of the same arc $a_P$ equipped with the two possible orientations.}
\label{f:indep3}
\end{figure}
\end{proof}
\begin{rem} {\whelm
Although it is not needed in the present situation, we will refer later
to the simple fact that the choices given by Figure~\ref{f:add}(i) and the one
shown by Figure~\ref{f:megegy} are nicely connected. 
} 
\end{rem}

\subsection{Convenient diagrams 
corresponding to a fixed Heegaard decomposition}

The next step in proving Theorem~\ref{thm:MainConv} is to relate
convenient Heegaard diagrams which are derived from the same Heegaard
decomposition but not necessarily from the same essential
pair-of-pants Heegaard diagram. This is the most demanding 
part of the proof of Theorem~\ref{thm:MainConv}, since now we need to work in
the four-punctured sphere as opposed to the three-punctured sphere 
(as in Subsection~\ref{ssec:fix}).

\begin{thm}\label{t:fixHD}
Suppose that ${\mathcal {U}}$ is a fixed Heegaard decomposition of the
3--manifold $Y$, which contains no $S^1\times S^2$--summand.  If $\DD
_i$ are convenient diagrams of $Y$ derived from the essential
pair-of-pants diagrams $(\Sigma , \alphak _i ,\betak _i)$ $(i=1,2)$
both corresponding to ${\mathcal {U}}$, then $\DD_1$ and $\DD_2$ are
nicely connected.
\end{thm}

According to Lemma~\ref{lem:conn}
(which rests on Theorem~\ref{t:luo}) two essential pair-of-pants
diagrams determining the same Heegaard decomposition can be connected
by a sequence of essential pair-of-pants diagrams, where the consecutive terms
differ by a flip of one of the curve systems.

Thus, we need to connect two
$\betak$--convenient diagrams which are derived from pair-of-pants
decompositions $(\Sigma ,\alphak , \betak )$ and $(\Sigma ,\alphak ',
\betak )$, so that $\alphak, \alphak '$ and $\betak$ are markings,
and $\alphak '$ is given by applying a flip to one of the curves in
$\alphak$. Let $S$ denote the 4--punctured sphere in which the flip
takes place, i.e. $S$ is the union of two pair-of-pants components of
$\Sigma - \alphak$. According to Theorem~\ref{t:convconn} we can
assume that away from $S$ (i.e. for all basepoint pairs outside of
$S$) and for all $\betak$--pair-of-pants we apply the same choices for
the two convenient diagrams. Hence all differences between the
convenient diagrams are localized in $S$.

First we would like to enumerate the possible configurations the
$\betak$--curves can have in $S$.  Let us first consider only those
$\alphak$-- and $\betak$--curves which were in the given pair-of-pants
decomposition. We will denote these sets of elements still by $\alphak$ and
$\betak$.  Let $\alphak_1$ denote $\alphak - \{ \alpha_0\}$, where $\alpha_0$
is the curve on which we will perform the flip. Recall that the curves in
$\alphak$ and $\betak$ provide a bigon--free Heegaard diagram, hence by
Proposition~\ref{p:convdomains} the domains in $\Sigma -\alphak_1 -\betak$ are
either rectangles, hexagons or octagons. Recall also that further
$\betak$--curves (and then $\alphak$--curves) are added to the diagram to turn
it into a convenient diagram.  By our previous discussion in
Section~\ref{sec:cdiagdef}, it follows that when forgetting about the bigons
in $S$, the additional $\betak$--curves will cut the octagons into hexagons.
(Once again, in our diagrams and considerations we will disregard the bigons.
Since those can be assume to be very thin and almost reach the basepoints, the
nice moves will remain nice even when adding these bigons back.)  Assume now
that we did add the new $\betak$--curves, but we did not add the new
$\alphak$--curves in $S$ yet. (Recall that we are considering here a $\betak$--convenient
diagram.) According to the above, we can assume that the
domains of $S-\betak $ are all hexagons. We continue to follow the
convention that parallel $\betak$--arcs in $S$ are denoted by a single
arc. The two pair-of-pants components contain four basepoints altogether,
hence there are four hexagons in $S$. This means that there are six arcs
partitioning $S$ into the four hexagons.

\begin{lem}
There are six possible configurations of six arcs to partition $S$
into four hexagons. These configurations are given by
Figure~\ref{f:hexposs} and are indexed by the four--tuples of degrees of the
four boundary circles of $S$. (The degree of a circle is the
number of arcs intersecting the circle.) 
\begin{figure}[ht]
\begin{center}
\epsfig{file=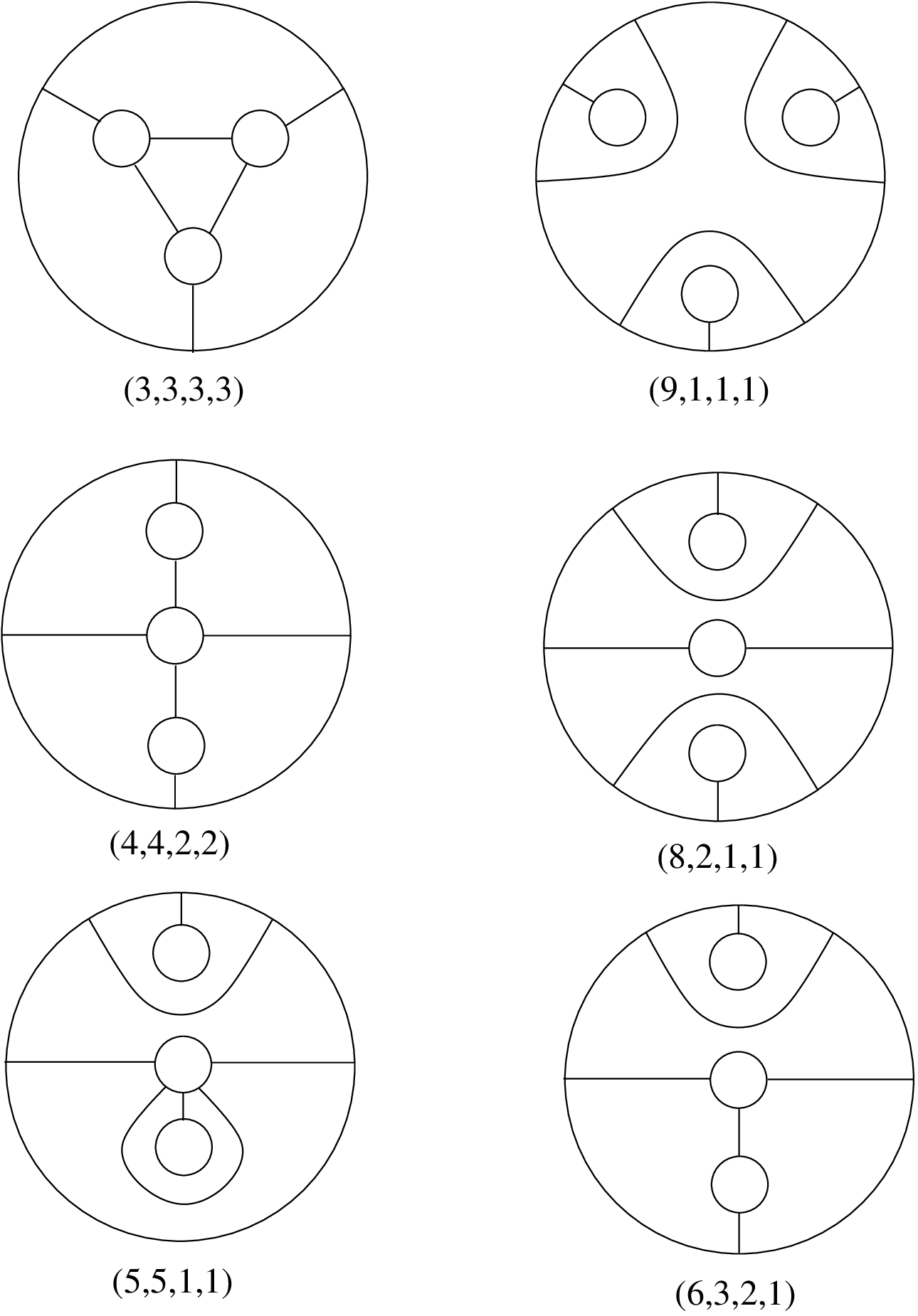, height=9cm}
\end{center}
\caption{{\bf Possible configurations of $\betak$--curves in the
    4--punctured sphere component of $\Sigma - \alphak_1$.} Boundary
  circles are all $\alphak$--curves, while the arcs are (parallel)
  $\betak$--curves in the 4--punctured sphere. As usual, we do not
  indicate the bigons and the basepoints.}
\label{f:hexposs}
\end{figure}
\end{lem}
\begin{proof}
By contracting the boundary circles of $S$ to points, the above
problem becomes equivalent to the enumeration of connected spherical
graphs on four vertices involving six edges, such that no
homotopically trivial and parallel edges are allowed. We can further
partition the problem according to the number of loops (i.e. edges
starting and arriving to the same vertex) the graph contains. Since we
view the graphs on $S^2$, a loop partitions the remaining three points
into a group of two and a single one. By connectedness the single one
must be connected to the base of the loop. If there is no loop in the
graph, then a simple combinatorial argument shows that the graph is a
square with a diagonal, and with a further edge, for which we have two
possibilities, corresponding to the graphs $(3,3,3,3)$ and
$(4,4,2,2)$ of Figure~\ref{f:graphs}, giving the corresponding 
configurations of Figure~\ref{f:hexposs}. 
\begin{figure}[ht]
\begin{center}
\epsfig{file=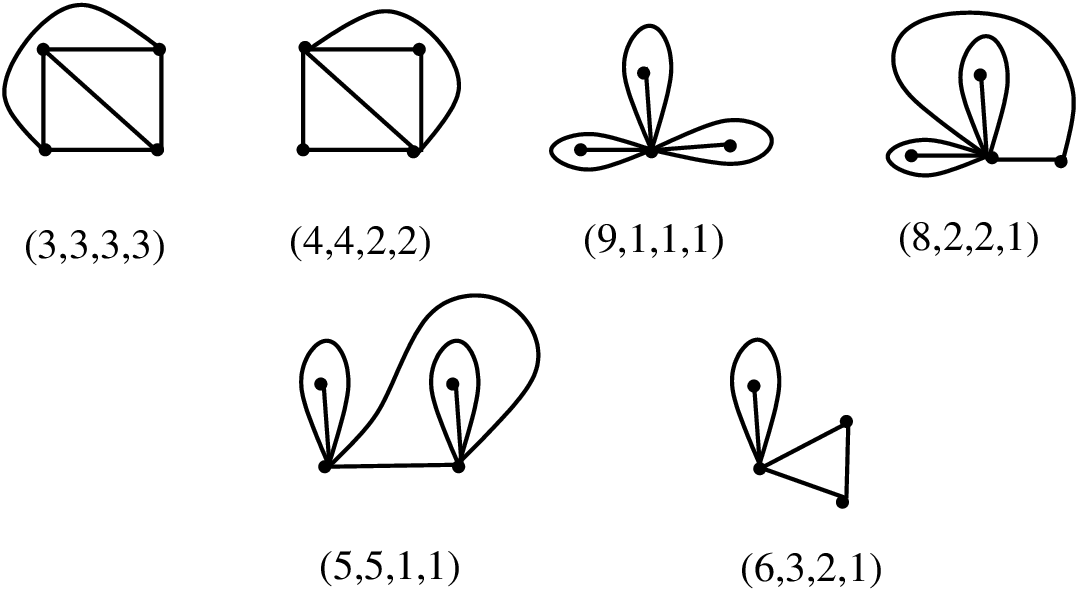, height=4cm}
\end{center}
\caption{{\bf The six connected spherical graphs.}}
\label{f:graphs}
\end{figure}
If the graph has one loop, then the only possibility in given by the
diagram with index $(6,3,2,1)$.  For two loops there are two
possibilities (according to whether the bases of the loops coincide or
differ); these are the graphs $(8,2,1,1)$ and $(5,5,1,1)$ of
Figure~\ref{f:graphs}, corresponding to the configurations $(8,2,1,1)$
and $(5,5,1,1)$ of Figure~\ref{f:hexposs}.  Finally there is one
possibility containing three loops, resulting in $(9,1,1,1)$ of
Figure~\ref{f:graphs}, giving rise to the configuration $(9,1,1,1)$ of
Figure~\ref{f:hexposs}.
\end{proof}

Now let us put $\alpha _0$ back into $S$.  Our next goal is to
normalize the curve $\alpha_0$ in $S$.  Notice that we could find a
diffeomorphic model of $S$ in which $\alpha _0$ is the standard curve
partitioning $S$ into two pairs-of-pants. With this model, however,
the configuration of the $\betak$--curves might be rather complicated.
We decided to work with a model of $S$ where the $\betak$--curves are
standard (as depicted by Figure~\ref{f:hexposs}) and in the following
we will normalize $\alpha _0$ by nice moves. In our subsequent
diagrams we will always choose a circle, which we will call ``outer''
and which we will draw as outermost in our planar pictures, and which
corresponds to the highest degree vertex of the spherical graph
encountered in the previous proof. (If this vertex is not unique, we
pick one of the highest degree vertices.)  The other three boundary
circles will be referred to as ``inner'' circles. Consider the
pair-of-pants from the two components of $S-\alpha_0$ which is
disjoint from the outer circle and denote it by $P$. Since we use a
model for $S$ that conveniently normalizes the $\betak$--arcs but not
necessarily $\alpha _0$, $P$ is not necessarily embedded in the
standard way into $S$ (as, for example the pairs-of-pants in
Figure~\ref{f:flip} embed into the 4--punctured sphere).  By an
appropriate homeomorphism $\phi$ on $P$, however, the $\betak$--curves
in $P$ can be normalized as before: since there are no octagons in
$S$, the result will look like one of the diagrams of
Figure~\ref{f:poss}(i) or (iii) (where $\alpha _0$ is the outer circle
of $P$). The two cases will be considered separately. We start with
the situation when the above pair-of-pants is of the shape of (i) (and
call this Case A), and address the other possibility (Case B)
afterwards. In the following we will show first that for any
$\alpha_0$ there is a sequence of nice isotopies and handle slides
which convert the curve system into one of finitely many cases which
we will call ``elementary curve configurations''.

Assume that we are in Case A. Consider the model of $P$
depicted by Figure~\ref{f:poss}(i) and connect the two boundary
components of $P$ different from $\alpha_0$ by a straight line
${\tilde {a}}_0$, which therefore avoids the $\betak$--segments
connecting different boundary components, and intersects the further
segments (intersecting the outer boundary $\alpha _0$ of $P$ twice)
transversely once each.  In the model $\alpha_0$ can be given by
considering the boundary of an $\epsilon$--neighbourhood (inside the
model for $P$) of the union of ${\tilde {a}}_0$ with the two boundary
circles it connects in $P$.  Let $a_0$ denote the image of ${\tilde
  {a}}_0$ in $S$ (when identifying $P\subset S$ with the model of
$P$ by the homeomorphism $\phi$). 
Consequently, $\alpha_0$ can be described by the arc $a_0$
connecting two boundary components of $S$: consider an
$\epsilon$--neighbourhood of the union of $a_0$ together with the two
boundary circles it connects.  We can also assume that the arc $a_0$
passes through the two basepoints $w_P^1$ and $w_P^2$ contained by the
pair-of-pants $P$, and we assume that these basepoints are near the
boundary components of $P$ the arc $a_0$ connects. Fix the dual curve
$a_0'$ (connecting the other two boundary circles of $S$ in the
complement of $a_0$, passing through the remaining two basepoints
$w_{S-P}^1$ and $w_{S-P}^2$) and distinguish one of the basepoints
from each pair outside and inside $P$, say $w_P^2$ and
$w_{S-P}^2$. The latter will be denoted by $w_d$.

Our immediate aim is to show that the curve system under consideration
can be transformed using nice moves into one of a finite collection of
curve systems (or ``elementary curve configurations'') described
below.  In order to state the precise result, first we need to consider
oriented arc systems in the diagrams of Figure~\ref{f:hexposs}.

\begin{defn} {\whelm Fix one of the diagrams of
    Figure~\ref{f:hexposs}, together with a distinguished basepoint
    $w_d$. An \emph{elementary situation} is a collection of three
    disjoint oriented arcs $\gamma_1$, $\gamma_2$, and $\gamma_3$
    in $S$ subject to the following
    constraints:
\begin{itemize}
\item each oriented arc $\gamma_i$ starts at one of the inner boundary
  circles, and there is only one arc starting at each inner circle,
\item immediately after starting at an inner circle, each $\gamma_i$ passes through the
  basepoint of the domain (i.e. the $\gamma_i$ crosses this basepoint before
  crossing any of the other $\beta$-circles), 
\item the intersection of $\gamma_i$ with the $\betak$ is minimal in
  the following sense: there are no bigons in $S$ consisting of an arc
  in $\gamma_i$ and an arc in one of the $\betak$, and in fact, there
  are no triangles consisting of an arc in $\gamma_i$, an arc in
  $\betak$ and an arc in $\alpha_i$,
\item each arc contains a unique basepoint, none of which is $w_d$,
  and finally
\item each arc enters the domain of $w_d$ exactly once and points into it.
\end{itemize}
}
\end{defn}

Before proceeding further, we give the list of all elementary
situations.
\begin{lem}\label{l:3333elem}
Consider the configuration of $S$ depicted by $(3,3,3,3)$ of
Figure~\ref{f:hexposs}, and fix $w_d$ in the lower left hexagon. Then there are four
elementary situations of this case,  given by Figure~\ref{f:elem}.
\begin{figure}[ht]
\begin{center}
\epsfig{file=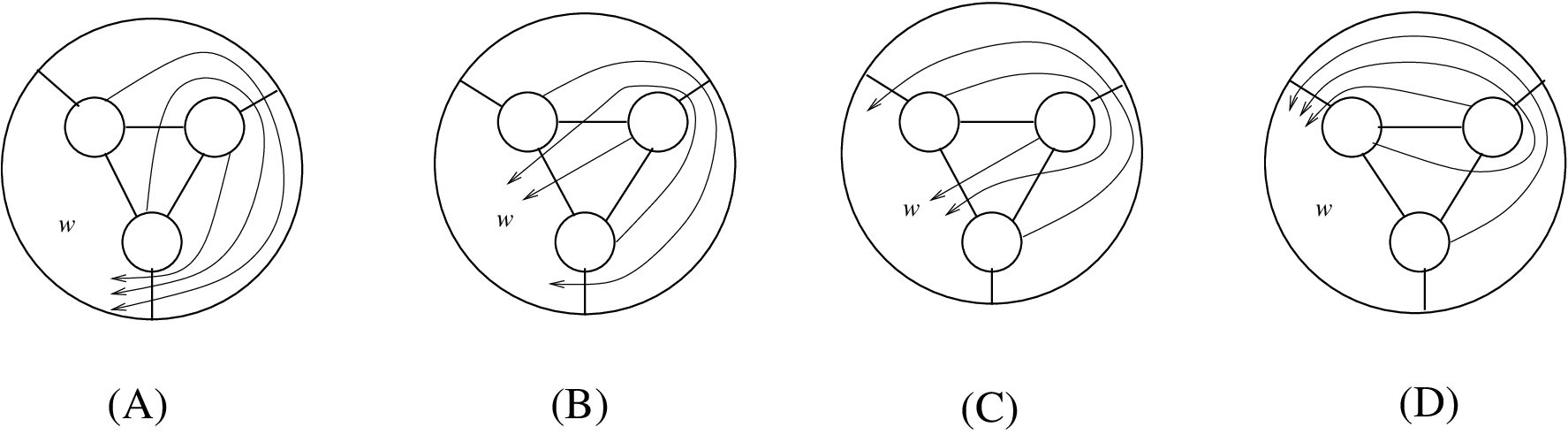, height=3cm}
\end{center}
\caption{{\bf Elementary situations for (3,3,3,3) of
    Figure~\ref{f:hexposs}, with $w_d=w$ being chosen in the lower
    left hexagon.}}
\label{f:elem}
\end{figure}
\end{lem}
\begin{proof}
Consider the arc starting at the circle which is disjoint from the
domain containing $w_d$. There are three choices for that arc (shown
by (A), by (B) and (C), and by (D) of Figure~\ref{f:elem}), since
after entering a domain (and passing through the basepoint there) the
arc should enter and therefore stop at the domain of $w_d$. A similar
simple case-by-case analysis for the remaining two arcs shows that
Figure~\ref{f:elem} lists all possibilities in this case.
\end{proof}
The further three possible choices of $w_d$ in the case of $(3,3,3,3)$
are all symmetric, hence (after possible rotations) the diagrams of
Figure~\ref{f:elem} provide a complete list of elementary situations
in the case of $(3,3,3,3)$. Before listing all elementary situations
for the remaining five possibilities of Figure~\ref{f:hexposs} we make an 
observation. Suppose that
$w_d$ is in a domain which has an inner circle on its boundary
which circle is not adjacent to any other domain. (For example, in
$(9,1,1,1)$ of Figure~\ref{f:hexposs} there are three such domains.) Then
there are no elementary situations with this choice of $w_d$, since $w_d$
could be the only basepoint for the arc starting at the inner circle,
but that is not allowed by our definition. This observation cuts down
the possible choices for the distinguished point $w_d$.

\begin{lem}
The elementary situations of the remaining five configurations of
Figure~\ref{f:hexposs} (up to symmetry) are shown by
Figure~\ref{f:tovabbi}.
\begin{figure}[ht]
\begin{center}
\epsfig{file=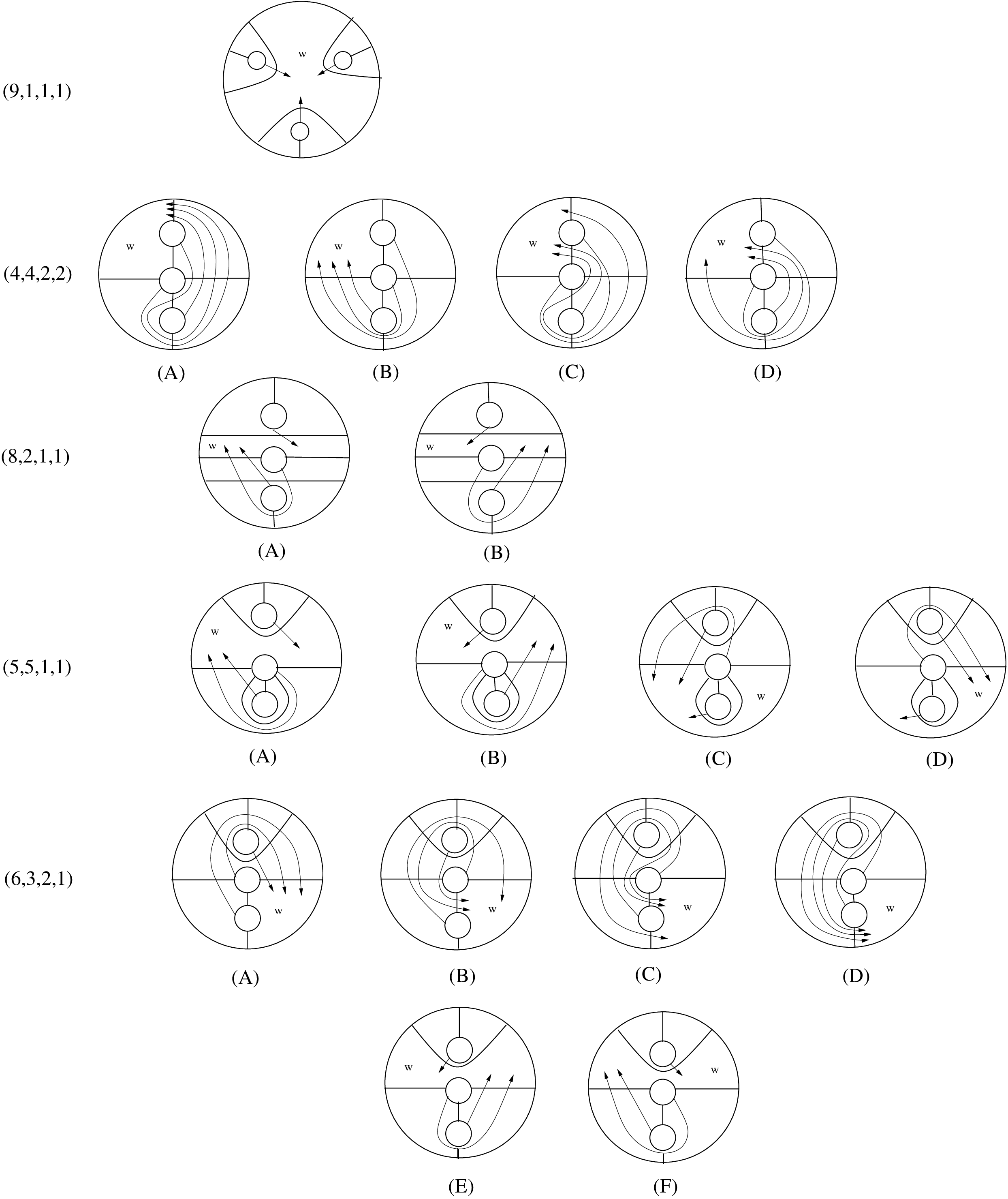, height=12cm}
\end{center}
\caption{{\bf Elementary situations for the further five possibilities of
  Figure~\ref{f:hexposs}.}}
\label{f:tovabbi}
\end{figure}
\end{lem}
\begin{proof}
In $(9,1,1,1)$ there is only one domain into which we can place $w_d$
without having an empty set of elementary situations.  For that choice
the elementary situation is unique.  For $(4,2,2,2)$
all four choices of domains for $w_d$ are possible and symmetric, for
$(8,2,1,1)$ there are two (symmetric) choices. (Figure~\ref{f:tovabbi} shows only
one of the symmetric choices.) For $(5,5,1,1)$ and
$(6,3,2,1)$ there are two possible choices, the further choices are
either symmetric, or do not provide any elementary situations. A
fairly straightforward argument, similar to the one given in the
proof of Lemma~\ref{l:3333elem} now shows that Figure~\ref{f:tovabbi}
provides all possible  elementary situations.
\end{proof}

Now we return to the discussion of curve systems on the
four--punctured sphere $S$. Notice first that an
elementary situation provides a curve system on $S$: take each
oriented arc, together with the boundary circle it starts from, and
consider the boundary of an $\epsilon$--neighbourhood (for
sufficiently small $\epsilon$) of it in $S$. The resulting curves,
regarded as $\alphak$--curves (together with the basepoints on which
the arcs passed through) provide a nice diagram on $S$ (which,
together with curves on $\Sigma - S$, gives a nice diagram for
$Y$). We will call these curve systems on $S$ \emph{elementary
  curve configurations}.

Let us consider the curve $\alpha_0$ in $S$, which (according to our
previous discussions) can be described by an arc $a_0$ connecting two
boundary components of $S$. Recall first that in the
$\betak$--convenient diagram there are further $\alphak$--curves: one
in $P$ (separating the two basepoints on the arc $a_0$) and one in
$S-P$ (separating the two basepoints on $a_0'$). We choose these
curves as follows.

Suppose first that $a_0$ enters and leaves the domain containing $w_d$
at least once. Then consider the subarc $a_1$ of $a_0$ which starts at
one of its endpoints, passes through one of the basepoints and stops
right before $a_0$ passes through the second basepoint, which we
choose to be the distinguished one. The boundary of a small
neighbourhood of the circle component from which $a_1$ starts and of
$a_1$ now provides $\alpha_1$. In case $a_0$ does not enter and leave
the domain of $w_d$ (for such a possibility see
Figure~\ref{f:elfel}(a)), we choose another curve $\alpha _1$: Instead
of applying Figure~\ref{f:add}(i), we rather apply the choice shown
by Figure~\ref{f:megegy}. In the above example the
appropriate choice is given by Figure~\ref{f:elfel}(b).
A similar choice applied in the
pair-of-pants $S-P$ gives $\alpha_1'$; now the subarc $a_1'$ will
avoid the distinguished basepoint $w_d$. Now we are ready to state the
result which nicely connects curve configurations in Case A to the
elementary curve configurations.
\begin{figure}[ht]
\begin{center}
\epsfig{file=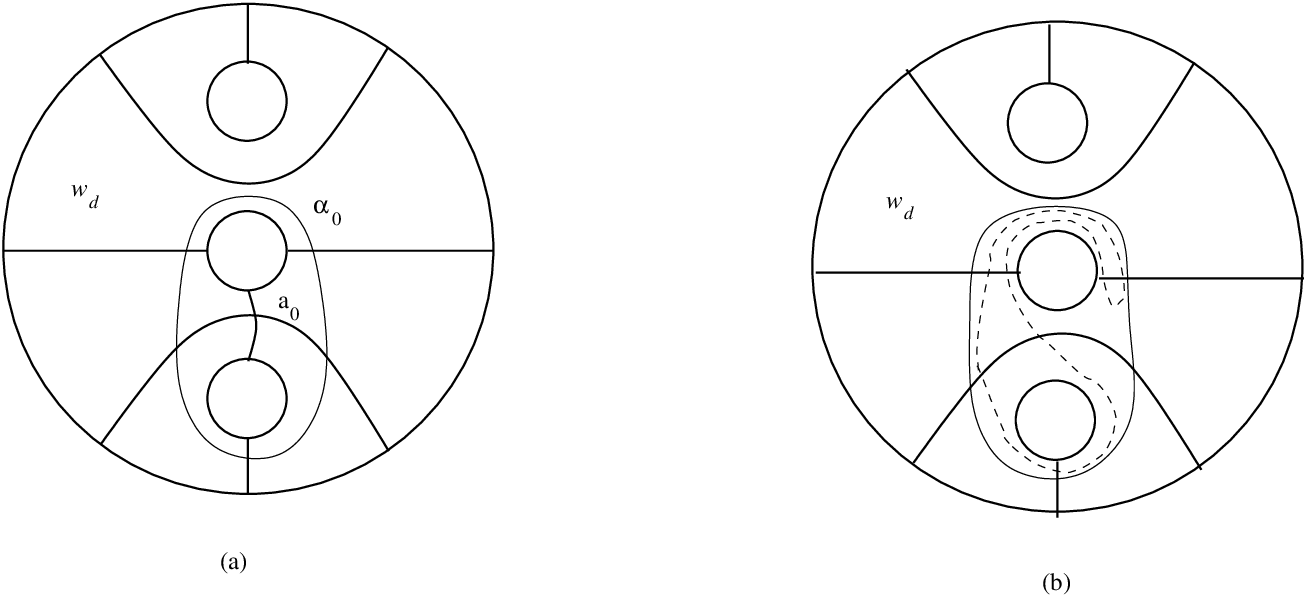, height=6cm}
\end{center}
\caption{{\bf A configuration when $a_0$ does not enter/leave 
the domain of $w_d$.} In (b) the choice of $\alpha _1$ is shown.}
\label{f:elfel}
\end{figure}

\begin{prop}\label{p:elemek}
Suppose that a $\betak$--convenient diagram $\DD_1$ in $S$ 
falling under Case A (with
$\alpha_0$ given, and $\alpha_1, \alpha_1'$ chosen as above) is
fixed. Then there is an elementary curve configuration $\DD_2$ such that 
$\DD_1$ and $\DD_2$ are nicely connected. 
\end{prop}
\begin{proof}
  First we will represent the three curves by three oriented arcs
  (which will resemble the presentation of elementary situations).  We
  start by applying a nice handle slide on $\alpha_0$ over $\alpha_1$
  performed at a segment of $\alpha_0$ neighbouring $w_d$. Notice that
  with the somewhat complicated choice of $\alpha _1$ given above,
  such a nice handle slide always exists: if the arc $a_0$ enters and
  leaves the domain of $w_d$ then the parallel portion of $\alpha _0$
  and $\alpha _1$ provide the required (nice) handle slide, while in
  the other possibility for $a_0$, our modified choice of $\alpha _1$
  makes sure of the existence of the handle slide. 
  In Figure~\ref{f:simp} we work out a particular example: (a) shows
  the two arcs $a_0$ and $a_0'$ (the neighbourhoods of which, together
  with their endcircles, provide $\alpha _0$ and $\alpha _0'$); the
  arc $a_0$ is solid, while $a_0'$ is dashed on Fiure~\ref{f:simp}(a).
  Figure~\ref{f:simp}(b) shows $\alpha _0$ and $\alpha _1$, and also
  indicates the point where we take the handle slide; in this figure
  $\alpha _0$ is dashed and $\alpha _1$ is solid.
\begin{figure}[ht]
\begin{center}
\epsfig{file=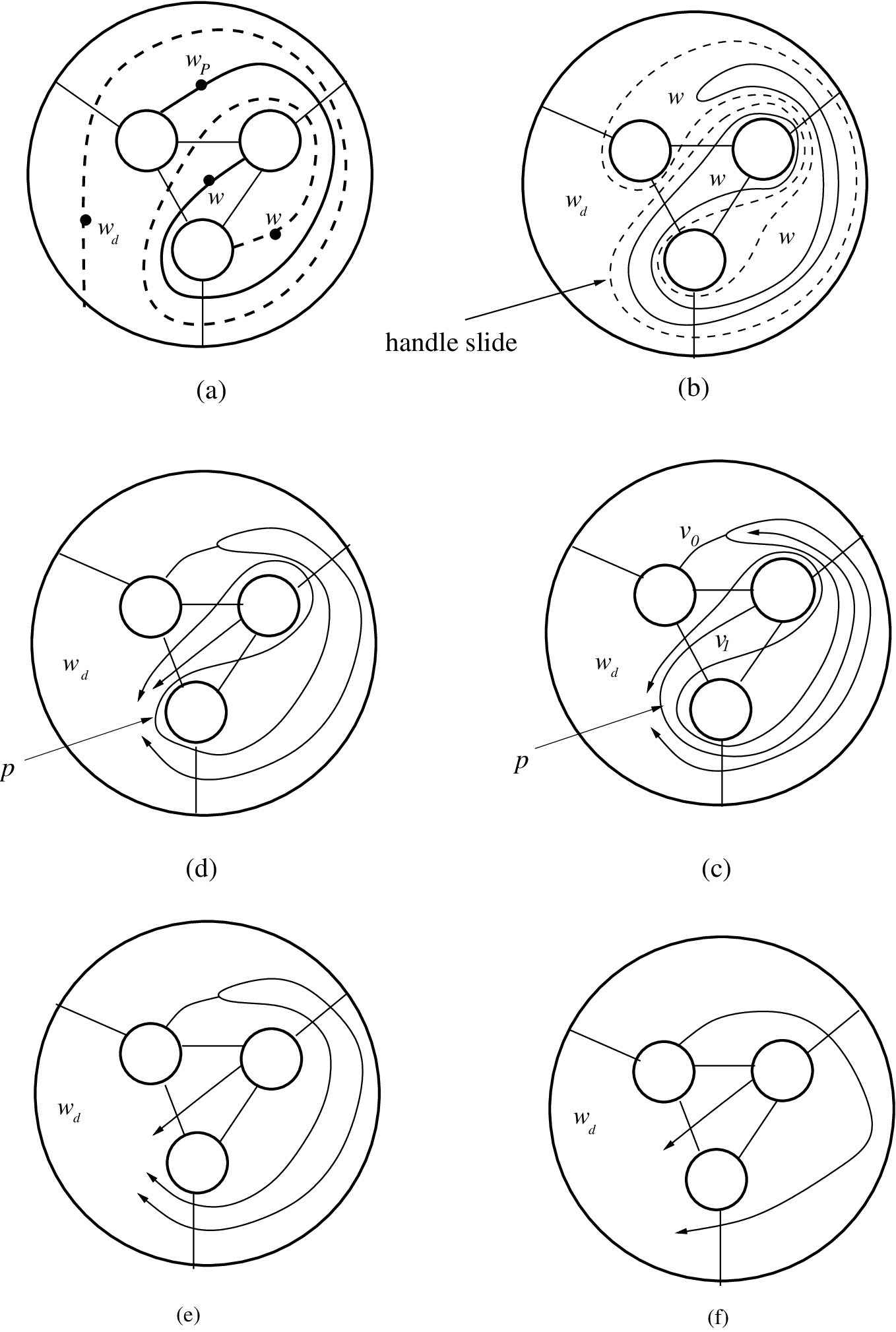, height=10cm}
\end{center}
\caption{{\bf Simplifying the curves in an example to an elementary
    situation.}  In (a) the curve $a_0$ (solid) and $a_0'$ (dashed)
  are shown, giving rise to $\alpha _0$ (dashed in (b)) and $\alpha
  _0'$ (solid in (b)). Diagrams (c) and (d) show the points where the
  simplifications (by nice isotopies) should be performed.  We pass
  from (e) to (f) by a nice isotopy again, and (f) is part of an
  elementary situation.}
\label{f:simp}
\end{figure}
Perform the handle slide along a curve $\delta$ where $\delta (0)$ is
the point of $\alpha_0$ in the domain $D_{w_d}$ of $w_d$ closest to
$w_d$. To simplify notation, indicate $\alpha_1$ with the subarc
defining it, with an arrow on its end which is not on a boundary
component, and denote this oriented arc by $v_1$.  The curve $a_0$,
after the handle slide has been performed, will be indicated by a
similar curve, this time however it starts at the other boundary
component (which was connected to the first by $a_0$), passes through
the other basepoint of $P$ and forks right before it reaches $v_1$.
We put an arrow to both ends of the fork; the result will be denoted
by $v_0$. The two curves in the chosen particular example are shown by
Figure~\ref{f:simp}(c).  A similar object is introduced for the last
curve $\alpha_0'$, which will be denoted by $v_0'$ (and which, for
simplicity, is not shown on Figure~\ref{f:simp}(c)). The result is
reminiscent to the three oriented arcs in the definition of an
elementary situation: we have three oriented 'arcs' (one of which
forks) starting at different inner circles, and passing through three
basepoints of $S$ distinct from $w_d$.  The arcs typically enter the
domain containing $w_d$ many times.  The curves $\alpha_1, \alpha_1'$
and $\alpha_0$ can be recovered from these arcs as the boundaries of
the small neighbourhoods of the arcs together with the boundary
circles the arcs start from.

Consider a point $p$ on one of the arcs which is in the domain
$D_{w_d}$ containing $w_d$, which point $p$ can be connected to $w_d$
in the complement of all the oriented arcs within $D_{w_d}$, and when
traversing on the arc containing $p$ to its end with an arrow, we
leave $D_{w_d}$ at least once; see Figure~\ref{f:simp}(c) for such a
point $p$. (If there is no such point on a certain arc, then the arc
at question enters $D_{w_d}$ and immediately stops, exactly as arcs in
an elementary situation do.) Now consider the same three arcs (one of
which still might fork), and modify the one containing $p$ by
terminating it at $p$. Consider the curve system corresponding to this
modified set of oriented arcs. (The result of $v_1$ of our example
under this operation is shown by Figure~\ref{f:simp}(d).)  The rest of
the arc (pointing from $p$ to the endpoint of the arc) then can be
regarded as a curve $\gamma$ defininig an isotopy from this newly
defined curve system back to the previous one. Since an arc can
terminate either in $D_{w_d}$ next to $w_d$, or in the bigon defined
by the fork, the isotopy defined by this $\gamma$ is a nice
isotopy. Repeat this procedure as long as appropriate $p$ can be found
(Figure~\ref{f:simp}(d) shows a further choice).  The two arrows of
the fork, together with an arc of the boundary of $D_{w_d}$, define a
bigon. If there are no other arcs in this bigon, then, as above, the
inverse of a nice isotopy can be used to eliminate the fork and
replace it with a single oriented arc. (This is exactly what happens
in Figure~\ref{f:simp}(e), and after applying this move, we get
Figure~\ref{f:simp}(f), which is an elementary situation --- at least
it provides two curves of an elementary situation, and the third can
be recovered easily from the above sequence of diagrams.)

By repeating the above procedure, we will get a collection of three
disjoint oriented arcs, starting on the three inner circles and entering
$D_{w_d}$ exactly once, hence we get an elementary situation.
Since all the isotopies performed above are nice isotopies (or their
inverses), the claim of the theorem follows at once.
\end{proof}

Before proceeding further with the proof of Theorem~\ref{p:symmlesz}, let us
describe our current position. We classified all possible background
configurations of the $\betak$-curves in the four-punctured sphere where the
flip on the $\alphak$-curves takes place (there are six types of these
backgrounds).  Then we divided the possible $\alphak$-curves into two classes
(Case A and Case B), and showed that there is a finite number of possibilities
for the $\alphak $-curves for each background (coming from the elementary situations) with
which all other Case A configurations are nicely connected. What is left is to
show that the Case B configurations and the elementary curve configurations
(corresponding to a fixed background) are also nicely connected.

We proceed next to the classification of Case B configurations. Notice
that (as for the Case A configurations) the curve $\alpha _0\subset S$
can be indicated by an arc $a_0$ connecting two inner boundary circles
of $S$.  (Then $\alpha _0$ is the boundary of the tubular
neighbourhood in $S$ of this arc, together with the two boundary circles it
connects.)  The corresponding pair-of-pants is of the shape given by
Figure~\ref{f:poss}(iii) exactly in case the arc $a_0$ defining
$\alpha _0$ can be isotoped to be disjoint from all the $\betak$-arcs
in $S$. This means that $a_0$ can be chosen to be parallel with one of
the $\betak$-arcs in $S$.  By fixing the outer circle, therefore for
$(3,3,3,3)$ there are three Case B configurations, while this number
is zero for $(9,1,1,1)$, $(8,2,1,1)$ and $(6,3,2,1)$, one for
$(5,5,1,1)$ and two for $(4,4,2,2)$.  Recall also that in Case A we
distinguished a basepoint (called $w_d$) which is in a domain
neighbouring the outer circle. We say that a Case B configuration is
\emph{compatible} with the choice of $w_d$ if the arc $a_0$ is not
parallel to a boundary arc of the domain of $w_d$.

Now we are ready to show that all elementary curve configurations of Case A corresponding to a
fixed background configuration of Figure~\ref{f:hexposs} and a fixed
distinguished point $w_d$, and all Case B configurations
compatible with $w_d$ (and also correspond to the same background) 
are equivalent under nice handle slides and
isotopies.  As before, we give the arguments in detail for the
case of $(3,3,3,3)$ as a background configuration, and then 
indicate the necessary modifications to be made for the other cases.

\begin{prop}\label{p:3333eq}
Fix a distinguished basepoint $w_d$ as before and consider all the
elementary situations for the background
$(3,3,3,3)$ with $w_d$ as the distinguished basepoint. Consider also the two
Case B configurations compatible with the chosen $w_d$. The convenient 
diagrams corresponding to these choices are nicely connected.
\end{prop} 
\begin{proof}
Consider the diagrams (1),  (2) and (3) of
Figure~\ref{f:choice}.
The first two diagrams show the two Case B configurations
(with the distinguished basepoint $w=w_d$), while the third diagram shows a 
Case A configuration which will be helpful in the proof.

Consider now the two placements of $\alpha_1$ (corresponding to the
position of $\alpha_0$ given by Figure~\ref{f:choice}(1)) as shown by
Figures~\ref{f:choice}(4) and (5). Since these two choices are two
cases of adding a new curve in a pair-of-pants listed (iii) in
Figure~\ref{f:poss}, Theorem~\ref{t:convconn} shows that the two
choices give rise to nicely connected diagrams.  On the other hand, by
adding the last $\alphak$-curve as given by the dashed curve in (4) and (5),
after a nice handle slide and nice isotopies we conclude that
the elementary situations (A) and (B) of Figure~\ref{f:elem} and the Case B configuration of
(1) are nicely connected. The same two placements of $\alpha _1$ in the
diagaram (2) (and the choice of  $\alpha _0'$ as shown in (C) or (D) of 
Figure~\ref{f:elem})
show that this Case B diagram is nicely connected with (C) and (D) of 
Figure~\ref{f:elem}. Putting the curve $\alpha _1$ into the diagram 
of (3) as given by Figures~\ref{f:choice}(6) and (7), a nice handle slide 
and nice isotopies turn this diagram into the elementary situations shown by 
(B) and (D) of Figure~\ref{f:elem}. In conclusion, we connected all 
the curve configurations corresponding to elementary situations
(and also Case B configurations) with the distinguished basepoint $w=w_d$
by nice moves, concluding the proof of the proposition.
\end{proof}
\begin{figure}[ht]
\begin{center}
\epsfig{file=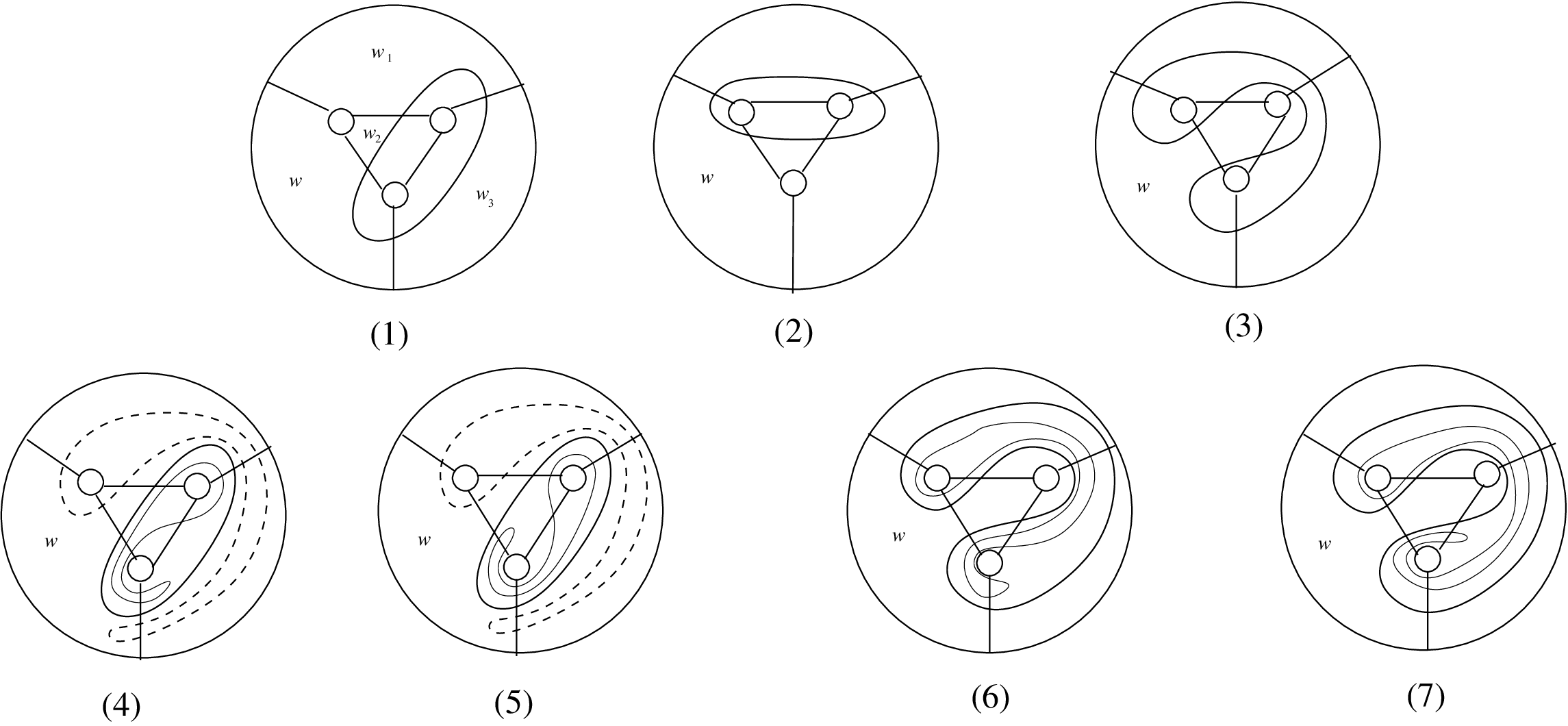, height=5cm}
\end{center}
\caption{{\bf Diagrams (4) and (5) connect (A) and (B) of
  Figure~\ref{f:elem}; here $w=w_d$.}  Similar choices in (2) and (3)
  connect (C) to (D) and (B) to (D) of Figure~\ref{f:elem}. Notice also that 
  (1) and (2) give all Case B configurations compatible with the choice
  $w=w_d$.}
\label{f:choice}
\end{figure}
\begin{prop}\label{p:vegso3333}
  Suppose that the $\betak$-curves are positioned in the
  four-punctured sphere $S$ as shown by $(3,3,3,3)$. Then all
  configurations are nicely connected.
\end{prop}
\begin{proof}
Since each Case B configuration is compatible with  two choices of the distinguished 
basepoint, we can use these configurations to connect 
configurations with different fixed distinguished basepoints.
The same argument applies if we change the choice of the outer circle,
concluding the argument.
\end{proof}
The same strategy applies for the 
further five 
remaining backgrounds listed in Figure~\ref{f:hexposs}:
\begin{thm}\label{t:elemgen}
Consider a background configuration of Figure~\ref{f:hexposs}.
Then the Case A elementary curve configurations and the Case B 
elementary curve configurations (if any)
corresponding  to the chosen  background are nicely connected.
\end{thm}
\begin{proof}
The idea of the proof is exactly the same as the proof of
Propositions~\ref{p:3333eq} and \ref{p:vegso3333}, therefore we only provide the 
$\alpha _0$-circles 
which should be used in the same spirit as used in the proofs of the above
propositions. 
Indeed, the circles can be defined by the dashed arcs of 
Figure~\ref{f:modif}. In each case one needs to make a careful (but rather straightforward)
choice of  the last $\alphak$-curve in the diagram; this last choice will not be given 
explicitly here. 
Notice that for $(9,1,1,1)$ there is one possible
nontrivial place for $w_d$ and in this case there is only one
elementary situation (and no Case B configuration), hence we do not need to do anything further. For the remaining cases the diagrams of Figure~\ref{f:modif}
provide the appropriate dashed arcs (as usual, the curves are the boundaries
of the neighbourhoods of the unions of the arcs and the two circles
they connect).
\end{proof}
\begin{figure}[ht]
\begin{center}
\epsfig{file=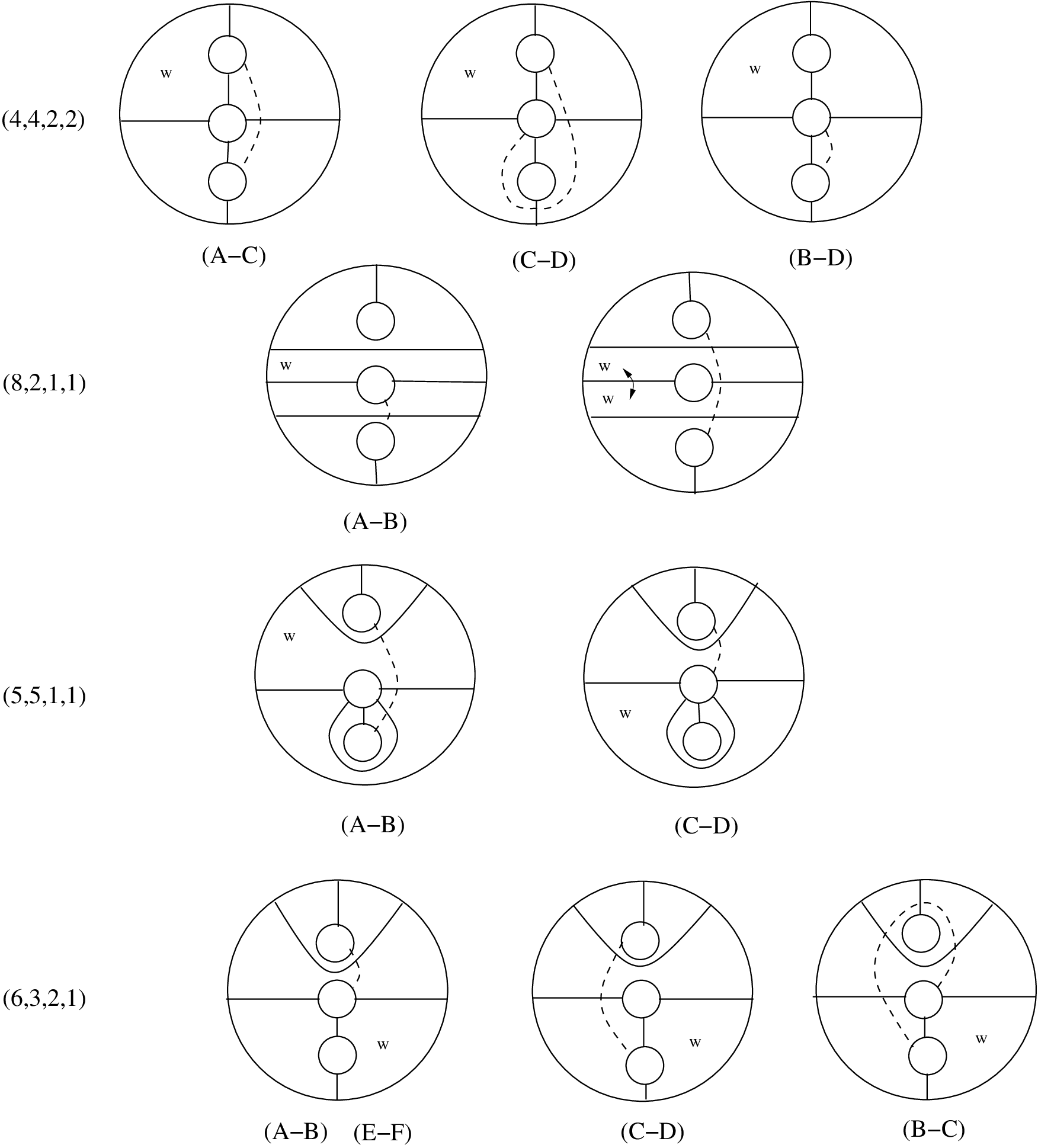, height=10cm}
\end{center}
\caption{{\bf Diagrams instructing how to connect elementary curve
    configurations.} The circle $\alpha _0$ is given by the boundary
  of the neighbourhood of the dashed arc together with the two circles
  it connects. The further curves should be added as it is shown in
  Figure~\ref{f:choice}; hence there are more than one possibilities
  for them. Applying straightforward handle slides and isotopies, the
  individual diagrams can be used to connect different elementary
  curve configurations and Case B curve configurations.}
\label{f:modif}
\end{figure}

\begin{proof}[Proof of Theorem~\ref{t:fixHD}]
Suppose that the convenient diagram $\DD_i$ is derived from the
essential pair-of-pants diagram $(\Sigma , \alphak _i, \betak _i)$
($i=1,2$). According to the assumption of the theorem, the two
essential pair-of-pants diagrams represent the same Heegaard
decomposition, therefore by Lemma~\ref{lem:conn} they are connected by
a sequence of flips. Therefore it is enough to check the theorem in
the case when the markings $\alphak_1$ and $\alphak _2$ differ by a
flip and $\betak_1=\betak _2$. Suppose that the flip takes place in
the four-punctured sphere $S\subset \Sigma$.  The $\betak$--curves
provide one of the configurations of Figure~\ref{f:hexposs}. According
to Proposition~\ref{p:elemek} then both the curve systems $\alphak _1$ and
$\alphak _2$ (before and after the flip) are nicely connected to
either an elementary curve configuration or a Case B curve configuration. Applying
Theorem~\ref{t:elemgen} we conclude that the original Heegaard
diagrams are nicely connected, finishing the proof of the theorem.
\end{proof}

\begin{rem}
{\whelm Notice that (since we normalized the shape of $\alpha _0$ in
  $S$), the same proof applies for $g$-flip equivalent configurations,
  hence we can use Theorem~\ref{thm:luohely} of the Appendix instead
  of Theorem~\ref{t:luo}.  }
\end{rem}

\subsection{Convenient diagrams and stabilization}
\label{ssec:stab}

Next we consider the relation between convenient diagrams and stabilizations.
Suppose that $(\Sigma , \alphak , \betak )$ is a bigon--free essential pair-of-pants
Heegaard diagram for the 3-manifold $Y$ which contains no $S^1\times S^2$-summand. 
 Choose a crossing $x$ of an $\alphak$-- and a
$\betak$--curve (called $\alpha_1$ and $\beta_1$) which is on the boundary
of a domain $D$ which is either a hexagon or an octagon. Let $(\Sigma ' ,
\alphak ' , \betak ')$ denote the pair-of-pants Heegaard diagram we get by the
stabilization procedure described in Lemma~\ref{l:stabil}. In the following,
${\DD}$ will denote a symmetric convenient diagram derived from $(\Sigma , \alphak ,
\betak )$, while ${\DD}'$ will be a symmetric convenient diagram we get from $(\Sigma '
, \alphak ', \betak ')$ by applying the following choices: 
\begin{itemize}
\item For any pair-of-pants
which is away from the stabilization we apply the same choices as for ${\DD}$.
\item For part of  of the diagram we get by stabilizing the annuli $A_{\alpha}$ and $A_{\beta}$, there are no further choices to make. This is because
these regions contain only
 rectangles and octagons, and our goal is to construct a
 symmetric convenient diagrams.
\end{itemize}
\begin{figure}[ht]
\begin{center}
\epsfig{file=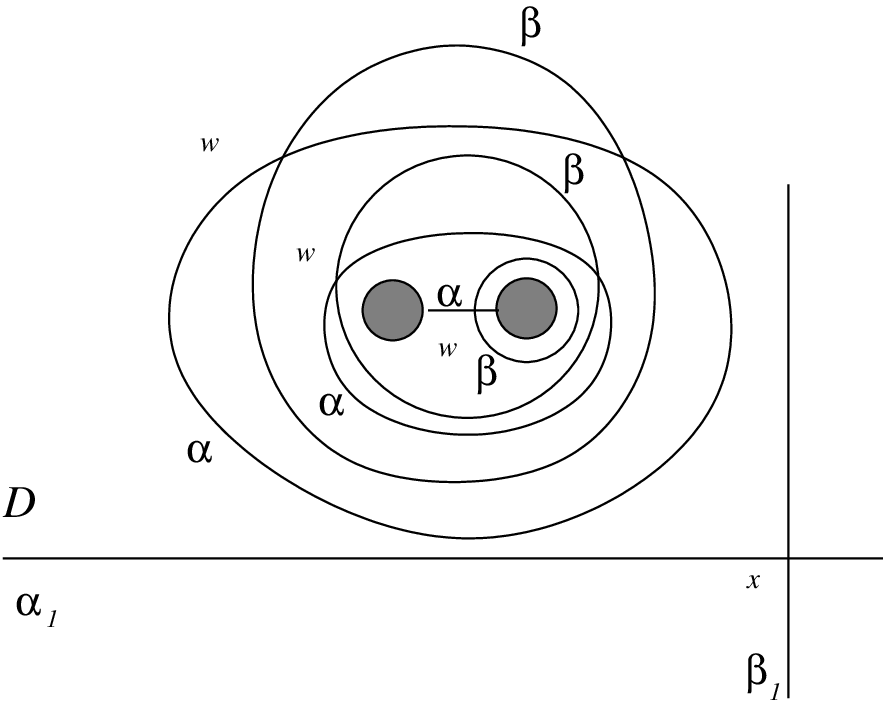, height=6cm}
\end{center}
\caption{{\bf Stabilizations resuting the diagram $\DD _1$.} The two
  full circles denote feet of a 1--handle. Basepoints are denoted by
  $w$.}
\label{f:stabb}
\end{figure}
\begin{thm}\label{thm:nicestab}
The convenient diagrams ${\DD}$ and ${\DD}'$ are nicely connected.
\end{thm}
\begin{proof}
Let us define the diagram $\DD_1$ by taking two nice type-$b$ and 
a nice type-$g$ stabilization in the elementary domain $D$ of $\DD$
containing a basepoint, where the pair-of-pants stabilization took
place, as it is instructed by  Figure~\ref{f:stabb}. We would like to show that 
$\DD '$ and $\DD _1$ can be connected by nice handle slides and nice isotopies.

Indeed,  slide the $\alpha _1''$ (and similarly $\beta _1''$) of 
Figure~\ref{f:stab}  in $\DD '$ over $\alpha _1$ (and $\beta _1$, resp.)
by a nice handle slide, and apply nice isotopies until the resulting curves
become part of the domain $D$. Repeat the same procedure now for the curves
$\alpha _1'$ and $\beta _1'$. The resulting diagram is shown in
Figure~\ref{f:mego}. Now it is easy to find a sequence of nice
handle slides and nice isotopies connecting the diagram of Figure~\ref{f:mego}
and of Figure~\ref{f:stabb}, concluding the proof of the theorem.
\begin{figure}[ht]
\begin{center}
\epsfig{file=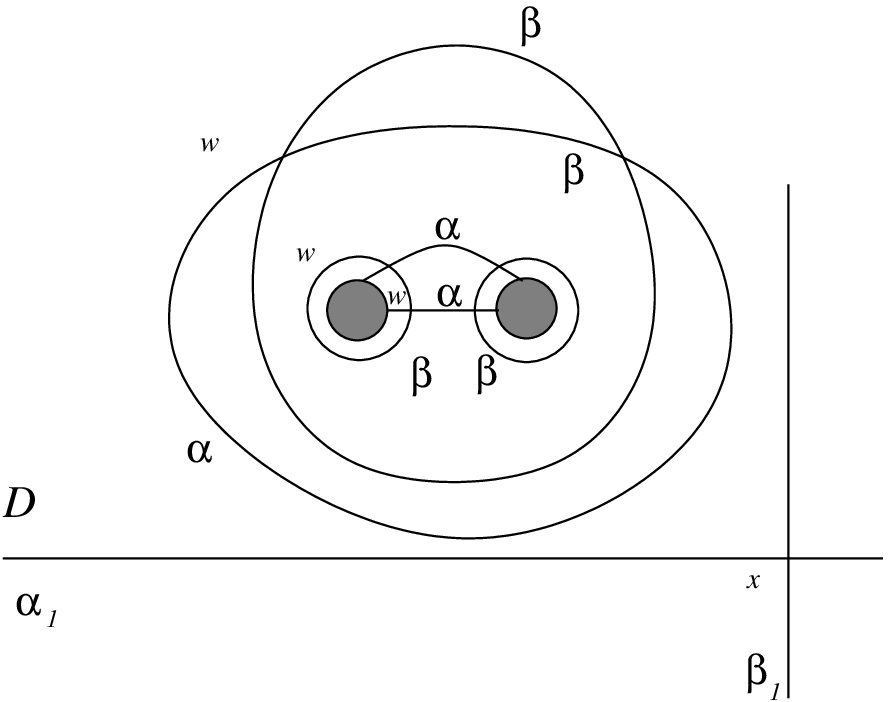, height=6cm}
\end{center}
\caption{{\bf The diagram after the four nice handle slides and
    appropriate nice isotopies.} Further nice handle slides and nice
  isotopies transform the diagram into Figure~\ref{f:stabb}.}
\label{f:mego}
\end{figure}
\end{proof}

\begin{proof}[Proof of Theorem~\ref{thm:MainConv}]
Suppose that $(\Sigma _i , \alphak _i, \betak _i )$ are essential
pair-of-pants diagrams (corresponding to Heegaard decompositions
${\mathcal {U}}_i$) giving rise to convenient Heegaard diagrams $\DD
_i$ ($i=1,2$). According to the Reidemesiter--Singer Theorem \cite{R,S} (see also \cite{saul}), the
Heegaard decompositions ${\mathcal {U}}_1$ and ${\mathcal {U}}_2$
admit isotopic stabilizations.
Let $(\Sigma , \alphak ^i , \betak ^i )$ denote the
essential pair-of-pants diagram compatibe with the common Heegaard decomposition 
${\mathcal {U}}$ we get
by stabilizing the essential pair-of-pants diagram $(\Sigma _i ,
\alphak _i , \betak _i)$.  Choose a convenient diagram $\DD ^i$
derived from $(\Sigma , \alphak ^i, \betak ^i)$. According to
Theorem~\ref{thm:nicestab}, the convenient diagrams $\DD_i$ and $\DD
^i$ are nicely connected for $i=1,2$. On the other hand, $\DD ^1$ and
$\DD ^2$ are now convenient diagrams corresponding to the same
Heegaard decomposition, hence by Theorem~\ref{t:fixHD} these diagrams
are nicely connected.  Since being nicely connected is transitive, the
above argument shows that the convenient diagrams $\DD_1$ and $\DD_2$
are nicely connected, concluding the proof.
\end{proof}

\section{The chain complex associated to a nice diagram}
\label{sec:hfhom}

In this section we define the chain complex $(\CFaa (\DD ), \partialaa
_{\DD } )$ on which the definition of the stable Heegaard Floer
invariant will rely.  The definition of this chain complex is modeled
on the definition of the Heegaard Floer homology groups $\CFa$ of
\cite{OSzF1, OSzF2}, cf. also \cite{SW} and
Section~\ref{subsec:Appendix} of the present paper. In the next two
sections we will deal with nice diagrams, and put our results
concerning convenient diagrams temporarily aside.

Suppose that $\DD = (\Sigma =\Sigma _g, \alphak =\{ \alpha_i \}
_{i=1}^k, \betak =\{ \beta _j \} _{j=1}^k, \w = \{ w_1, \ldots ,
w_{k-g+1}\} )$ is a nice multi-pointed Heegaard diagram for $Y$.  An
unordered $k$--tuple of points $ \x=\{ x_1, \ldots , x_k \} \subset
\Sigma$ will be called a \emph{generator} if the intersection of $\x$
with any $\alphak$-- or $\betak$--curve is exactly one point. In other
words, $\x$ contains a unique coordinate from each $\alpha_i$ and
from each $\beta _j$.  Let $\Gen$ denote the set of these
generators, and let
\[
\CFaa (\DD ) =\oplus _{\x \in \Gen}\Field
\]
be the $\Field$--vector space generated by the elements of ${\mathcal
  {S}}$.  We will typically not distinguish an element of ${\mathcal
  {S}}$ from its corresponding basis vector in $\CFaa (\DD)$.

\begin{defn}
\label{def:EmptyPolygon}
(Cf. \cite[Definition~3.2]{SW})\label{d:sw} {\whelm Fix two generators
  $\x$ and $\y \in \Gen$. We say that a \emph{2n--gon from $\x$ to
    $\y$} is a formal linear combination $D=\sum n_i D_i$ of the
  elementary domains $D_i$ of $\DD=(\Sigma , \alphak , \betak )$,
  satisfying the following conditions:
\begin{itemize}
\item
$x_i=y_i$ with $n$ exceptions;
\item
all multiplicities $n_i$ in $D$ are either 0 or 1, and
at every coordinate $x_i\in \x$ (and similarly for $y_i \in \y$)
either all four domains meeting at $x_i$ have multiplicity 0 (in which case
$x_i=y_i$) or exactly one domain has multiplicity 1 and all three others have
multiplicity 0 (when $x_i \neq y_i$);
\item
the support $s(D)$ of $D$, which is the union of the closures
${\overline {D}}_i$ of the elementary domains which have $n_i=1$ in
the formal linear combination $D=\sum _i n_i D_i$ is a subspace of
$\Sigma$ which is homeomorphic to the closed disk, with $2n$ vertices
on its boundary;
\item the $n$ coordinates (say $x_1, \ldots , x_n$ and $y_1, \ldots ,
  y_n$) where $x_i$ differs from $y_i$ (which we call the
  \emph{moving} coordinates) are on the boundary of $s(D)$ in an
  alternating fashion, in such a manner that, when using the boundary
  orientation of $s(D)$ (which is oriented by $\Sigma$) the
  $\alphak$--arcs point from $x_i$ to $y_j$ while the $\betak$--arcs
  from $ y_i$ to $x_j$. In short, $\partial (\partial D \cap \alphak
  )=\y -\x$ and $\partial (\partial D \cap \betak )=\x -\y$.
\end{itemize}
The $2n$--gon is \emph{empty} if the interior of $s( D)$ is disjoint
from the basepoints $\w$ and the two given points $\x$ and $\y$. As
before, for $n=1$ the $2n$--gon is called a \emph{bigon}, while for
$n=2$ it is a \emph{rectangle}. } 
\end{defn}

Notice that an empty bigon contains
exactly one elementary bigon and some number of elementary rectangles,
while an emtpy rectangle is the union of some number of elementary
rectangles.

Suppose that $\x,\y \in \CFaa(\DD )$ are two generators.  Define the
$\pmod{2}$ number $m_{\x \y}\in \Field$ to be the cardinality
$\pmod{2}$ of the set $\mxy$ defined as follows.  We declare $\mxy$ to
be empty if $\x$ and $\y$ are either equal or differ in at least three
coordinates.  If $\x$ and $\y$ differ at exactly one coordinate, then
we define $\mxy$ as the set of empty bigons from $\x$ to $\y$, while
if $\x$ and $\y$ differ in exactly two coordinates, then $\mxy$ is the
set of empty rectangles from $\x$ to $\y$.  It is easy to see that
either $\mxy$ is empty or it contains one or two elements.  The two
elements of $\mxy$ can be distinguished by the part of the $\alphak$--
(or $\betak$--) curves containing the moving coordinates that are in
the boundary of the domain.  (If $\x$ and $\y$ differ in exactly one
coordinate and $\mxy$ contains two elements, then there are isotopic
$\alphak$-- and $\betak$--curves.)

Now define the \emph{boundary operator}
\[
\partialaa _{\DD}\colon \CFaa (\DD ) \to \CFaa (\DD )
\]
by the formula
\[
\partialaa _{\DD} (\x) =\sum _{\y \in \Gen}m _{\x \y}\cdot \y
\]
on the generators, and extend the map linearly to $\CFaa (\DD)$.

For future reference, it will be convenient to have an alternative
characterization of $\partialaa _{\DD}$. To this end, it will help
to generalize Definition~\ref{def:EmptyPolygon} as follows:

\begin{defn}
  \label{def:Domain}
  {\whelm Suppose that $\x, \y \in \Gen$ are two generators in the
    Heegaard diagram $(\Sigma , \alphak , \betak )$.  A {\em domain
      connecting $\x$ to $\y$} (or, when $\x$ and $\y$ are implicitly
    understood, simply a {\em domain}) is a formal linear combination
    $D=\sum_{i} n_i\cdot D_i$ of the elementary domains, which in turn
    can be thought of as a $2$-chain in $\Sigma$, satisfying the
    following constraints.  Divide the boundary $\partial D$ of the
    2-chain $D$ as $a+b$, where $a$ is supported in $\alphas$ and $b$
    is supported in $\betas$.  Then, thinking of $\x$ and $\y$ as
    0-chains, we require that $\partial a=\y-\x$ (and hence $\partial
    b=\x-\y$). The set of domains from $\x$ to $\y$ will be denoted by
    $\pi _2 (\x , \y)$.}
\end{defn}

Less formally, for each $i=1,\dots,g$, the portion of $\partial D$ in
$\alpha_i$ determines a path from the $\alpha_i$--coordinate of $\x$
to the $\alpha_i$--coordinate of $\y$, and the portion of $\partial D$
on $\beta _i$ determines a path from the $\beta_i$--coordinate of $\y$
to the $\beta_i$--coordinate of $\x$.

\begin{defn}
  A domain $D=\sum n_i\cdot D_i$ is \emph{nonnegative} (written $D\geq
  0$) if all $n_i\geq 0$.  Given an elementary domain $D_i$, the
  coefficient $n_i$ is called the {\em multiplicity of $D_i$ in
    $D$}. Equivalently, given a point $z\in \Sigma-\alphak -\betak $
  the {\em local multiplicity of $D$ at $z$}, denoted $n_{z}(D)$, is
  the multiplicity of the elementary domain $D_i$ containing $z$ in
  $D$.  For $\w = \{ w_1, \ldots , w_k\} $ we define $n_{\w }(D)=\sum
  _i n _{w_i}(D)$.
\end{defn}

It is often fruitful to think of domains from the following elementary
point of view.  A domain $D$ connecting $\x$ to $\y$ is a linear
combination of elementary domains whose local multiplicities satisfy a
system of linear equations, one for each intersection point $p$ of
$\alpha_i$ with $\beta_j$. To describe these relations, we need a
little more notation. At each intersection point $p$ of $\alpha_i$
and $\beta_j$ there are four (not necessarily distinct) elementary
domains, which we label clockwise as $A_p$, $B_p$, $C_p$, and $D_p$,
so that $A_p$ and $B_p$ are above $\alpha_i$ and $B_p$ and $C_p$ are
to the right of $\beta_j$, cf. Figure~\ref{f:kereszt}. Let $a_p$,
$b_p$, $c_p$, and $d_p$ denote the multiplicities of $A_p$, $B_p$,
$C_p$, and $D_p$ in $D$.
\begin{figure}[ht]
\begin{center}
\epsfig{file=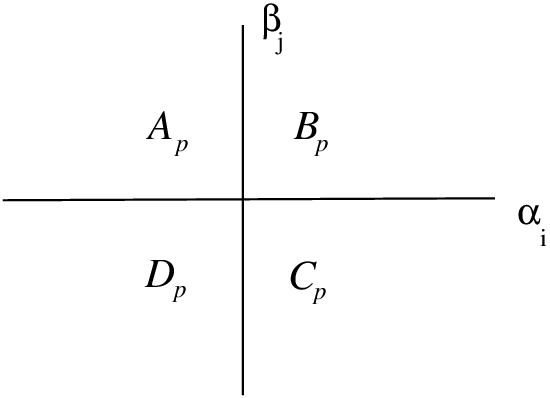, height=3cm}
\end{center}
\caption{{\bf The quadrants $A_p,B_p,C_p$ and $D_p$ at a crossing.}}
\label{f:kereszt}
\end{figure}
For a generator $\x \in {\mathcal {S}}$ and an intersection point
$q$ define
  \[\delta(q,\x)=\left\{
  \begin{array}{ll}
  +1 & {\text{if $q\in \x$}} \\ 0 & {\text{otherwise.}}
  \end{array}
  \right.
  \]

\begin{lem}
The formal linear combination $D=\sum D_i$ is in $\pi _2 (\x , \y )$
(ie. is a domain from $\x$ to $\y$) if, for each $p\in \alpha_i\cap
\beta_j$, we have that
  \begin{equation}
    \label{eq:DomainFromXtoY}
    a_p+c_p= b_p+d_p- \delta(p,\x)+\delta(p,\y).
  \end{equation}
\end{lem}
\begin{proof}
  Consider the quadrants around each intersection point $p$ as
  illustrated in Figure~\ref{f:kereszt}. The right horizontal arc
  (between $B_p$ and $C_p$, oriented out of $p$) appears in $\partial
  D$ with multiplicity $b_p-c_p$, while the left horizontal arc
  (between $A_p$ and $D_p$, oriented into $p$) appears in $\partial D$
  with multiplicity $a_p-d_p$. Thus, the point $p$ appears in
  $\partial (\partial D\cap \alpha_i)$ with multiplicity
  $a_p+c_p-b_p-d_p$; and in a domain from $\x$ to $\y$, each
  coordinate appears with multiplicity
  $\delta(p,\y)-\delta(p,\x)$. Equation~\eqref{eq:DomainFromXtoY}
  then follows.
\end{proof}

It is straightforward to see that if $D=\sum n_i D_i\in \pi _2 (\x ,
\y )$ then $-D\in \pi _2 (\y , \x )$, and for the sum $D+D'$ with
$D\in \pi _2 (\x,\y )$ and $D'\in \pi _2 (\y , \z )$ we have
$D+D'\in \pi _2 (\x , \z )$.

Suppose that $(\Sigma , \alphak , \betak , \w)$ is a nice
multi-pointed Heegaard diagram, and assume that the elementary domain
$D_i$ is a $2n$--gon. Define $e(D_i)$ by the formula $1-\frac{n}{2}$,
and extend this definition linearly to all domains with $n_{w_i}=0$
($i=1, \ldots , k-g+1$).  The resulting quantity $e(D)$ is the
\emph{Euler measure} of $D$. 
\begin{rem}
{\whelm
The Euler measure has a natural
interpretation in terms of the Gauss-Bonet theorem as follows. Endow
$\Sigma$ with a metric for which all $\alpha_i$ and $\beta_j$ are
geodesics, meeting at right angles.  The Euler measure of an
elementary domain is the integral of the curvature of this
metric. Notice that this alternate definition applies for elementary
domains which are not $2n$--gons.}
\end{rem}

If $D\in \pi _2 (\x,\y )$, then for each $\x$-- (and $\y$--)
coordinate $x_i$ (and $y_j$) consider the average of the
multiplicities of the four domains meeting at $x_i$ (and $y_j$). The
sum of the resulting numbers $p_{x_i}(D)$ and $p_{y_j}(D)$ will be
denoted by $p(D)$ and is called the \emph{point measure} of $D$. We
define the \emph{Maslov index} $\mu (D)$ to be the sum
\begin{equation}
  \label{eq:MaslovIndex}
  \mu(D)=e(D)+p(D).
\end{equation}
\begin{rem} {\whelm The term ``Maslov index'' is used here since,
    according to a theorem of Lipshitz~\cite{Lipshitz}, the quantity
    defined in Equation~\eqref{eq:MaslovIndex} computes the expected
    dimension of the moduli space of curves associated to the domain
    $D$.}
\end{rem}

It will be useful to have another construction; before introducing it,
we pause for a definition:
\begin{defn}
  An {\em elementary $\alphak$-arc} $a$ is a subarc of
  $\alpha_i\subset \Sigma$ which connects two intersection points
  $x_1=\alpha_i \cap \beta _j$ and $x_2=\alpha_i \cap \beta _k$ such
  that int$(a)$ contains no further intersection points, ie.
  int$(a)\cap \betak = \emptyset$. A similar definition gives the
  notion of elementary $\betak$--arcs. Let ${\mathcal {A}}$ denote the
  set of all elementary arcs ($\alphak$-- or $\betak$--) of the
  diagram.  It follows from the definition that an elementary arc $a$
  is in the boundary of two (not necessarily distinct) elementary
  domains $D_{l}$ and $D_{r}$.
\end{defn}

Let $\x,\y \in \Gen$ be two generators and consider $D\in \pi _2(\x,\y
)$ with $D\geq 0$. A topological space $S$, together with a tiling on
it, and a map $f\colon S \to \Sigma $ can be built from $D$ in the
following way.  If an elementary domain $D_i$ appears in $D$ with
multiplicity $n_i> 0$ then take $n_i$ copies of $D_i$ and denote them
by $D_i ^{(1)}, D_i ^{(2)}, \ldots , D_i ^{(n_i)}$. Suppose now that
$a\subset \alpha_t$ is an $\alphak$--elementary arc in the boundary of
the elementary domains $D_i$ and $D_j$, and assume without loss of
generality that $n_i \leq n_j$.  Then glue $D_i ^{(1)}$ to $D_j
^{(1)}, D_i ^{(2)}$ to $D_j ^{(2)}, \ldots , D_i ^{(n_i)}$ to $D_j
^{(n_i)}$ along the part of their boundary corresponding to the arc
$a$. If $b\subset \beta _l$ is a $\betak$--elementary arc on the
boundary of $D_i$ and $D_j$ (and once again $n_i \leq n_j$), then glue
$D_i ^{(n_i)}$ to $D_j ^{(n_j)}, D_i ^{(n_i-1)}$ to $D_j ^{(n_j-1)},
\ldots , D_i ^{(1)}$ to $D_j ^{(n_j-n_i+1)}$ along the part of their
boundary corresponding to the arc $b$. The existence of both the
tiling and the contiuous map $f$ obviously follow from the
construction. (Note that this construction is similar to the
construction of the surface in~\cite{Lipshitz}: the only difference is
the manner in which we handle the corner points.)

\begin{prop}
  \label{p:ConstructSurface}
  Suppose that for the domain $D\in \pi _2 (\x,\y)$ we have $D\geq 0$,
  $n_{\w}(D)=0$ and suppose that at each coordinate $x_i\in\x$ and
  $y_j\in\y$, we have that $p_{x_i}(D)$ and $p_{y_j}(D)$ are
  strictly less than $1$.  Then the topological space $S$ defined
  above is a surface with boundary and with corner points
  corresponding to the points $z\in\x\cup\y$ with $p_{z}(D)\equiv
  \frac{1}{4}\pmod{\frac{1}{2}}$.
\end{prop}
\begin{proof}
  The above construction provides a smooth manifold-with-boundary over
  each point $t\in D$ which is not one of the coordinates of $\x$ or
  $\y$.  At coordinates of $\x$ or $\y$, there are only a few ways the
  local multiplicities can distribute over the four adjoining regions.
  Indeed, up to cyclic orderings (reading clockwise around $t$) we can
  have one of the following distributions: $(1,0,0,0)$, $(1,1,0,0)$,
  $(1,2,0,0)$, $(2,1,0,0)$, and $(1,1,1,0)$. Following the above
  construction, we see that in all but the second case, the surface
  $S$ has a corner over $t$.
\end{proof} 

\begin{prop}\label{p:mashogy}
Suppose that  $\DD$ is a nice multi-pointed Heegaard
diagram.  Suppose furthermore that for $D\in \pi _2 (\x,\y )$ we
have $D\geq 0$, $n_{\w }(D)=0$ and $\mu
(D)=1$. Then either 

\noindent (a) $e(D)=\frac{1}{2}$ and the point measures
$p_{x_i}(D)=p_{y_i}(D)$ vanish with a single exception $i=j$, for which
both point measures $p_{x_j}(D)=p_{y_j}(D)$  are equal to $\frac{1}{4}$, or

\noindent (b) $e(D)=0$ and the point measures vanish with two exceptional
indices $i,j$ for which
$p_{x_i}(D)=p_{x_j}(D)=p_{y_i}(D)=p_{y_j}(D)=\frac{1}{4}$.
\end{prop}
\begin{proof}
Notice that since $D\geq 0$, by definition $p(D)$ is a positive
multiple of $\frac{1}{4}$ and (since the Heegaard diagram is nice) the
Euler measure $e(D)$ is a nonnegative multiple of
$\frac{1}{2}$. Therefore the condition $\mu (D)=e(D)+p(D)=1$ implies
that either

\noindent (a) $e(D)=\frac{1}{2}$ and $p(D)=\frac{1}{2}$ (implying
$(p_{x_1}(D),p_{y_1}(D))=(\frac{1}{4}, \frac{1}{4})$) or

\noindent (b) $e(D)=0$ and $p(D)=1$. 
\noindent In this latter case we have three possibilities for the point
measures:
\begin{enumerate}
\item $(p_{x_1}(D), p_{x_2}(D), p_{y_1}(D), p_{y_2}(D))=(\frac{1}{4},
  \frac{1}{4}, \frac{1}{4}, \frac{1}{4})$ or
\item $(p_{x_1}(D),p_{y_1}(D))=(\frac{1}{2}, \frac{1}{2})$ or
\item $(p_{x_1}(D),p_{y_1}(D))=(\frac{3}{4}, \frac{1}{4})$.
\end{enumerate}
Case (a) is exactly the first and (1) of Case (b) is the second
possibility given by the proposition.  We claim that (2) and (3) of
Case (b) cannot exist. When the point measure
$(p_{x_1}(D),p_{y_1}(D))$ is equal to $(\frac{1}{2}, \frac{1}{2})$, we
have that the two points $x_1$ and $y_1$ are equal and the entire
$\alphak$-- (or $\betak$--) circle containing it is in the boundary
$\partial D$. Since the domain is in one of its side, by
Lemma~\ref{l:mindket} we conclude that $n_{\w} (D)\neq 0$, a
contradiction.

Finally, we need to exclude the possibility for the point measures to
be equal to $(p_{x_1}(D),p_{y_1}(D))=(\frac{3}{4}, \frac{1}{4})$. This
can in principle happen in one of two ways: either all local
multiplicities around the corner point with multiplicity $\frac{3}{4}$
are bounded above by one, or not. In the latter case, there is some
curve, say $\alpha_i$ (or $\beta_j$, but that is handled in exactly
the same manner) with the property that the local multipicities of $D$
at the corner point are strictly greater on one side of $\alpha_i$
than they are on the other. From this, it follows globally that the
local multiplicities of $D$ are strictly greater on one side of
$\alpha_i$ than they are on the other. Since $D$ is a nonnegative
domain, it follows that $D$ contains all the elementary domains on one
side of $\alpha_i$. In view of Lemma~\ref{l:mindket}, this violates
the condition that $n_{\ws}(D)=0$. We return now to the case where all
local multiplicities are $\leq 1$. In this case,
Proposition~\ref{p:ConstructSurface} constructs a surface $S$ mapping
to $D$. It is easy to see that $S$ has a single boundary component,
and hence its Euler characteristic is congruent to $1\pmod{2}$. Note
that $S$ is an unbranched cover of a subsurface of $\Sigma$, and hence its Euler
measure is calculated as the Euler measure of $D$.  On the other hand,
by the Gauss-Bonet theorem, the Euler measure of $S$ coincides with
its Euler characteristic (since the correction terms coming from the
two corners cancel). But the Euler measure $e(D)$ is zero,
contradicting $\chi (S)\equiv 1$ (mod 2).
\end{proof}

We now give the following result essentially
from~\cite[Theorem~3.3]{SW}, where it is shown that the Maslov index
one pseudo-holomorphic disks in a nice diagram are (empty) embedded
bigons and rectangles.  The proof of this result consists of two
parts. In the first, it is shown that the properties of the index
formula ensure that the holomorphic curves that need to be counted
correspond to bigons and rectangles mapping into $\Sigma$. The second,
combinatorial part, shows that such bigons and rectangles are in fact
embedded.

The version we need here is slightly different. It states that the
index one (as defined by Equation~\eqref{eq:MaslovIndex}) nonnegative
domains are embedded bigons and rectangles. Again, the argument can be
thought of as consisting of two parts. In the first part, it is shown
that a nonnegative, index one domain corresponds to an immersed bigon or
rectangle (this is, effectively, Proposition~\ref{p:mashogy} above).
Once this is done, the proof that the corresponding domain is in fact
an embedded bigon or rectangle proceeds exactly as in~\cite{SW}.

\begin{prop}\label{p:atfog}
The space $\mxy$ 
of empty rectangles and bigons connecting $\x$ and $\y$ can be described by
\[
\mxy= \{ D\in \pi _2(\x, \y) \mid D\geq 0,\  n_{\w }(D)=0,\  \mu (D)=1 \}. 
\]
\end{prop}
\begin{proof}
For $D\in \mxy $ we have, by definition, that $D\in \pi _2 (\x,\y )$
and $D\geq 0$ (since all coefficients are either 0 or 1). Since $D$ is
empty, we have also $n_{\w}(D)=0$ and that all coordinates which do
not move have vanishing point measure. In addition, the moving
coordinates have point measure $\frac{1}{4}$. Now if $D$ is a bigon,
then it contains a unique elementary bigon, hence its Euler measure is
$\frac{1}{2}$, and since it has two moving coordinates $x_i, y_i$, we
conclude that $p(D)=\frac{1}{2}$, implying $\mu (D)=1$. If $D$ is a
rectangle, then $e(D)=0$ and since there are four moving coordinates,
we get $p(D)=1$, showing again that $\mu (D)=1$. Therefore $D\in \mxy$
satisfies the three required properties.

Assume conversely that $D\in \pi _2(\x, \y)$ satisfies $D\geq 0$,
$n_{\w}(D)=0$ and $\mu (D)=1$. Notice first that these properties
imply that $p(D)\leq 1$, hence, in particular, $D$ is empty,
i.e. does not contain any coordinate $x_i$ (or $y_i$) in its
interior. Consider
now the surface-with-boundary $S$ with the map $f\colon S\to \Sigma$
representing $D$ and the tiling given on $S$, as constructed in
Proposition~\ref{p:ConstructSurface}.  In view of
Proposition~\ref{p:mashogy}, $S$ is a disk with either two or four
corner points, each of which has $90^\circ$ angle. Now the same line
of reasoning as the one given for \cite[Theorem~3.3]{SW} shows that
$f$ is an embedding and $D$ is a bigon or rectangle, hence $D\in
\mxy$, concluding the proof.
\end{proof}
With the above identity, the boundary operator $\partialaa _{\DD}$ can
be rewritten on $\x \in \Gen$ as
\[
\partialaa _{\DD } \x = \sum _{\y \in \Gen}\sum _{\{ D\in \pi
  _2(\x, \y)\big| D\geq 0, n_{\w }(D)=0, \mu (D)=1\} } \y.
\]

We now turn back to the study of the pair $(\CFaa (\DD ), \partialaa
_{\DD})$.

\begin{thm}\label{thm:chaincomplex}
The pair $(\CFaa (\DD ) , \partialaa _{\DD}) $ is a chain complex,
that is, $\partialaa _{\DD}^2=0$.
\end{thm}
\begin{proof}
We need to show that for any
pair of generators $\x, \z$ the matrix element
\[
\langle \partialaa _{\DD}^2 \x , \z \rangle
\]
is zero (mod 2). Notice that the above matrix element is simply the
cardinality of the set
\[
{\mathfrak {N}}_{\x \z}= \bigcup _{\y \in \Gen} \mxy \times 
{\mathfrak{M}}_{\y ,\z}.
\]
The proof that ${\mathfrak {N}}= {\mathfrak {N}}_{\x \z}$ contains an
even number of elements will be partitioned into three subcases.
Define
\[
{\mathfrak {N}}(b)=\{ (D_1, D_2)\in {\mathfrak {N}}\mid {\mbox {both }}
D_i {\mbox { are bigons}}\}.
\]
In a similar vein, define ${\mathfrak {N}}(r)$ as the set of pairs
$(D_1, D_2)\in {\mathfrak {N}}$ when both $D_i$ are rectangles, and
finally define the set of mixed pairs ${\mathfrak {N}}(m)$ consisting
of those $(D_1, D_2)$ of ${\mathfrak {N}}$ in which one of the domains
is a bigon and the other one is a rectangle. Obviously
\[
{\mathfrak {N}}={\mathfrak {N}}(b)\cup {\mathfrak {N}}(r) \cup {\mathfrak {N}}(m)
\]
is a disjoint union, and if all the above subsets have even
cardinality, the evenness of $\vert {\mathfrak {N}}\vert $ follows at
once.

\noindent {\bf Case 1: Examination of ${\mathfrak {N}}(b)$.}  The set
${\mathfrak {N}}(b)$ will be further partitioned as follows: Suppose
that $(D_1, D_2)\in {\mathfrak {N}}(b)$. Let $i$ (and $j$) denote the
moving coordinate of $D_1$ (of $D_2$ resp.). Let ${\mathfrak
  {N}}(b)_1$ denote the set of pairs $(D_1, D_2)\in {\mathfrak
  {N}}(b)$ with $i=j$, and ${\mathfrak {N}}(b)_2$ the set of those
pairs where $i\neq j$.  

Suppose that the pair of bigons $(D_1, D_2)\in \mxy \times {\mathfrak
  {M}}_{\y ,\z}$ for some $\y \in \Gen$ is in ${\mathfrak {N}}(b)_2$.
Since the moving coordinate of $D_1$ is $i$, we get that $x_j=y_j$,
therefore the bigon $D_2 \in {\mathfrak{M}}_{\y , \z}$ can be regarded
as a bigon $D_1'=D_2\in {\mathfrak {M}}_{\x ,\y '}$, where the
coordinates of $\y'$ are given as $y_k'=x_k(=z_k)$ for all $k\neq
i,j$, $y_i'=x_i$ and $y_j'=z_j$. With this choice of $\y'$ it is easy
to see that $D_2'=D_1$ can be regarded as an element of
${\mathfrak{M}}_{\y ', \z}$, since $y_i=z_i$,
cf. Figure~\ref{f:part}(a). (The diagram also indicates that although
the moving coordinates of $D_1$ and $D_2$ are disjoint, the embedded
bigons themselves might intersect, requiring no alteration of the
above argument.)  Since $(D_1, D_2)$ and $(D_1', D_2')$ clearly
determine each other, we found a pairing on ${\mathfrak {N}}(b)_2$,
showing that the cardinality of this set is even.

\begin{figure}[ht]
\begin{center}
\epsfig{file=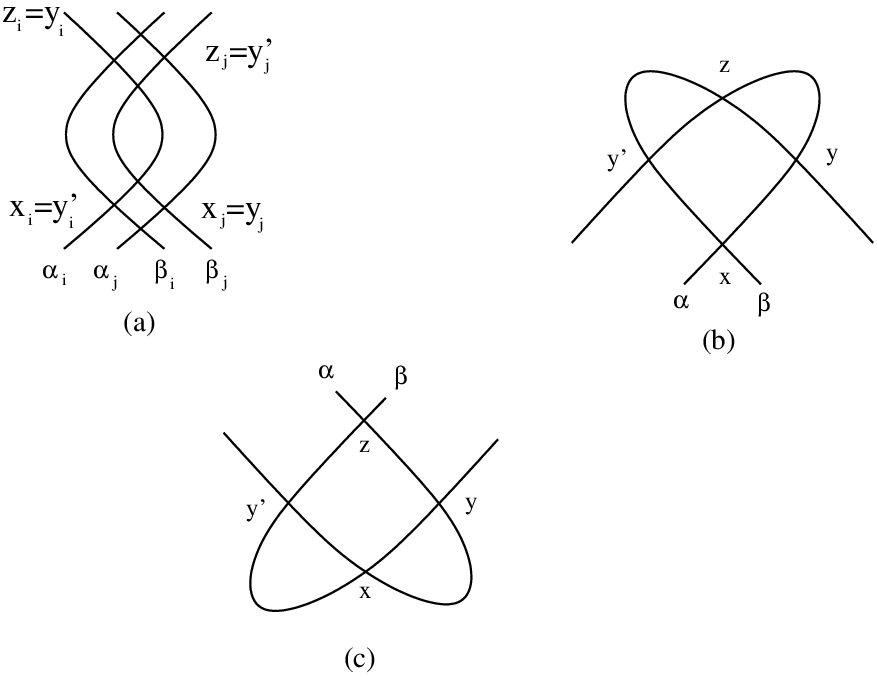, height=9cm}
\end{center}
\caption{{\bf Geometric possibilities when examining the matrix
    element $\langle \partialaa _{\DD}^2 \x , \z \rangle$.} The
  diagrams here correspond to Case 1 in the proof of
  Theorem~\ref{thm:chaincomplex}. In (a) the case of two moving coordinates
is illustrated, while in (b) and (c) we show the two possible
scenarios for one moving coordinate. The difference of these last 
diagrams is in the point measures of the starting ($x$) and final ($z$)
coordinates.}
\label{f:part}
\end{figure}

Consider now an element $(D_1, D_2)$ of ${\mathfrak {N}}(b)_1$.
Suppose that $D_1\in \mxy$ while $D_2\in {\mathfrak{M}}_{\y , \z}$.
Let $\alpha _i, \beta _i$ denote the curves containing the moving
coordinates $x_i, y_i, z_i$.  It follows from the orientation
convention that the elementary domain having multiplicity 1 in $D_2$
and starting at $\y$ is neighbouring the elementary domain at $\y$
which has multiplicity 1 in $D_1$. These elementary domains therefore
share either an elementary $\alphak$-- or a $\betak$--arc. The two
cases being symmetric, we assume the former. This means that the
domain $D_2$ starts back on the same elementary $\alphak$--arc
$\subset \alpha _i$ on which $D_1$ arrived to $y_i\in \y$. Now there
are two cases to consider. The segment either reaches first the
coordinate $x_i$ of $\x$ or $z_i$ of $\z$ on $\alpha _i$. In the first
case $p_{x_i}(D_1\cup D_2)=\frac{3}{4}$ and $p_{z_i}(D_1\cup
D_2)=\frac{1}{4}$, while in the second case $p_{x_i}(D_1\cup
D_2)=\frac{1}{4}$ and $p_{z_i}(D_1\cup D_2)=\frac{3}{4}$. Suppose that
we reach $z_i$ first --- the other case can be handles by obvious
modifications. This means that the $\betak$--curve $\beta_i$ enters
the bigon $D_1$ at $z_i$. Since at $x_i$ a portion of $\beta _i$ is
out of $D_1$, at some point $\beta_i$ must leave $D_1$. It can leave
the bigon between $z_i$ and $y_i$ (entering another bigon, which it
must also leave at some point), or between $z_i$ and $x_i$.  Since
$\beta_i$ will return to $x_i$, there exists an intersection point
$y_i'$ between $z_i$ and $x_i$ at which $\beta _i$ first leaves the
bigon. This argument then produces another intersection point $\y '$
with the coordinate on $\alpha_i$ and $\beta_i$ being $y_i'$, and puts
the situation in the form depicted in Figures~\ref{f:part}(b) and (c)
(depending whether the point measure of $D_1\cup D_2$ is $\frac{3}{4}$
at $z$ or at $x$).  So the pair $(D_1,D_2)\in \mxy \times
{\mathfrak{M}}_{\y ,\z}$ determines another pair $(D_1', D_2')\in
{\mathfrak{M}}_{\x ,\y'} \times {\mathfrak{M}}_{\y' , \z}$ (such that
the supports $s(D_1\cup D_2)$ and $s(D_1'\cup D_2')$ are equal),
defining a pairing on ${\mathfrak {N}}(b)_1$.  Since $y_i$ and $y_i'$
determine each other, we get that the cardinality of ${\mathfrak
  {N}}(b)_1$ is even.  This step concludes the proof that the
cardinality of ${\mathfrak {N}}(b)$ is even.

\noindent {\bf Case 2: Examination of ${\mathfrak {N}}(r)$.}  As
before, the set under examination can be partitioned further according
to the number of moving coordinates. This number is at least two
(since a single rectangle involves two moving coordinates), and for
the same reason it is at most four.  The case of four moving
coordinates means that the element $(D_1, D_2)$ involves two disjoint
rectangles (again, in the sense that although the supports might
intersect, the moving coordinates are on distinct curves,
cf. Figure~\ref{f:partcase2}(a)), and the evenness of the set of these
pairs follows from the same principle for ${\mathfrak {N}}(b)_2$.

Suppose that there are three moving coordinates. Suppose that the
corner point $y_i$ of $D_1$ is also a corner point of $D_2$. (Since
there are three moving coordinates, the two rectangles must share a
corner.) As before, the elementary domain in $D_2$ starting at $y_i$
shares a side with $D_1$; suppose it is an $\alphak $--arc. Moving
towards the $\x$--coordinate $x_i$ on that circle, we reach either
$x_i$ or the $\z$--coordinate $z_i$ first. As before, this means that
we found a point ($x_i$ or $z_i$) with the property that the point
meaure of $D_1\cup D_2$ at that point is $\frac{3}{4}$. This fact
provides an arc which cuts $D_1\cup D_2$ into two other rectangles and
provides the new coordinates for $\y '$. Since the triples $(\x, \y,
\z)$ and $(\x, \y ', \z)$ determine each other, and $\y \neq \y'$, the
evenness of the cardinality of the set at hand follows at once. See 
Figure~\ref{f:partcase2}(b).

\begin{figure}[ht]
\begin{center}
\epsfig{file=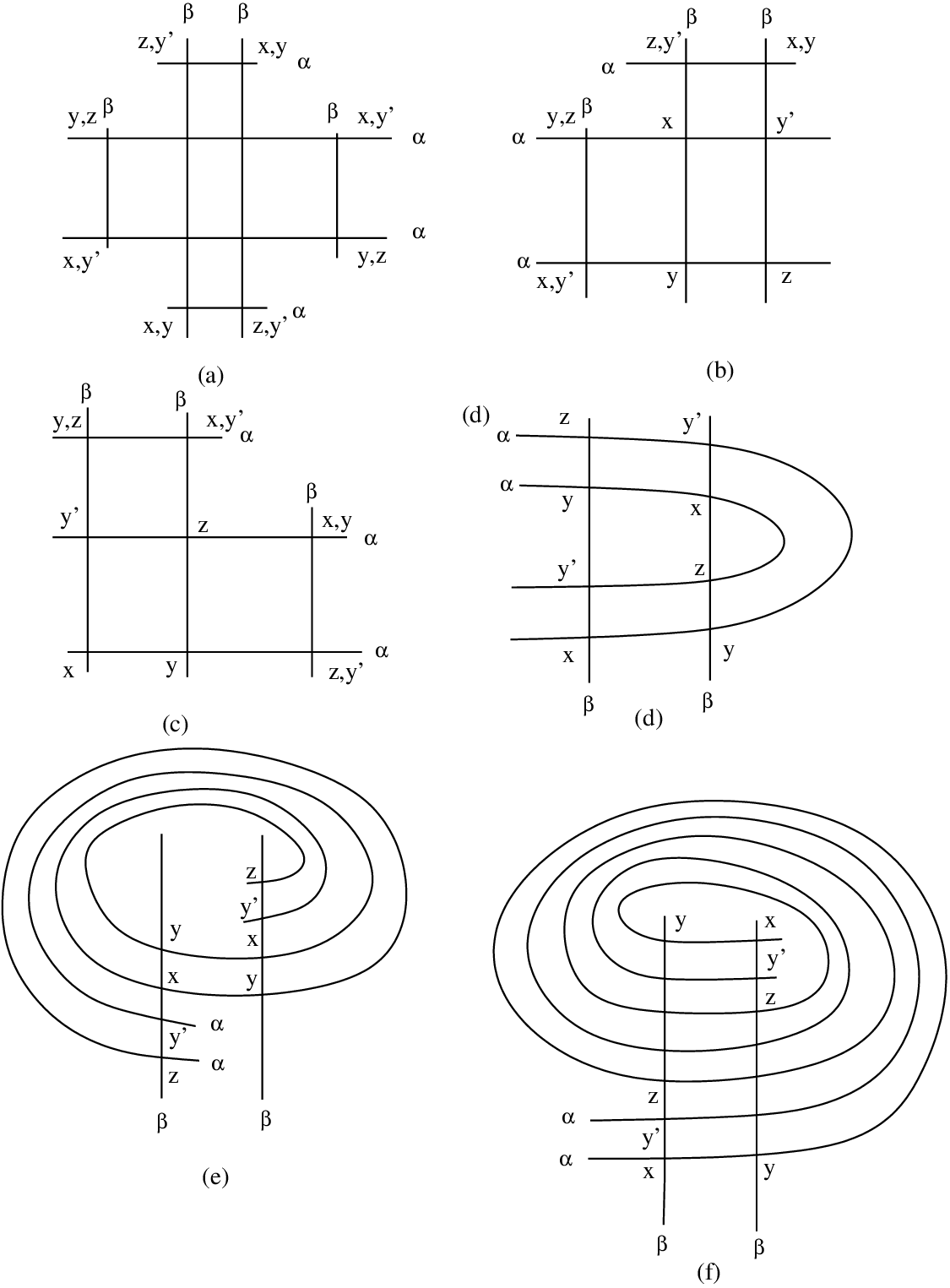, height=10cm}
\end{center}
\caption{{\bf Geometric possibilities when examining the matrix
    element $\langle \partialaa _{\DD}^2 \x , \z \rangle$.} The
  diagrams here correspond to Case 2 in the proof of
  Theorem~\ref{thm:chaincomplex}.} 
\label{f:partcase2}
\end{figure}

Finally we deal with the case of a pair $(D_1, D_2)$ of rectangles
with exactly two moving coordinates, which are on the curves $\alpha
_i, \alpha _j$ and $\beta _i , \beta _j$. Suppose that $D_1$ is a rectangle
from $\x$ to $\y$ and $D_2$ from $\y $ to $\z$ (with coordinates
$x_i, y_i, z_i$ on $\alpha _i$ and $x_j , y_j, z_j $ on $\alpha _j$).
There are various possibilities for $D_2$ to start at $\y$: for each of the
two coordinates $y_i, y_j$ it can start by sharing either an $\alphak$- or a 
$\betak$-edge with $D_1$, and on that side the $\z$-coordinate can be
reached before or after the corresponding $\x$-coordinate. We will deal with 
these different cases separately.

Let us start with the case when $D_2$ shares the $\alphak$-edge at
$y_j$ and the $\betak$-edge at $y_i$ with $D_1$. If on both arcs the
$\x$-coordinate comes before the $\z$-coordinate, then at $x_j$ there
is a quadrant where the multiplicity of $D_2$ is at least two, a
contradiction. Similarly, if both $\z$-coordinates come before the
$\x$-coordinates, we find a corner $z_i$ of $D_2$ with point measure
strictly greater than $\frac{1}{4}$, a contradiction again. Therefore
we can assume that on $\beta _j$ the domain $D_2$ reaches the
$\z$-coordinate first and then the $\x$-coordinate, while on $\alpha
_i$ we encounter $x_i$ first and then $z_i$. To avoid the
contradiction above, it can happen only if the point measurse of $z_j$
and $x_i$ are $\frac{3}{4}$; see Figure~\ref{f:partcase2}(d) for an
example. Since we are dealing with empty rectangles, the further
intersection points $y_i'$ and $y_j'$ become readily visible as the
point where $\alpha _i$ and $\beta _j$ leave $D_1$ and $D_2$
respectively.  This argument then provides the pair $(D_1', D_2')$
verifying the existence of the pairing on ${\mathfrak {N}}(r)_2$.
Notice that in this case one $\x$- and one $\z$-coordinate comes with
point measure $\frac{3}{4}$ in $D_1\cup D_2$.

Next assume that $D_2$ shares the $\alphak$-arcs at both $y_i$ and
$y_j$ with $D_1$. Since in a rectangle opposite sides support the same
number of elementary rectangles, when traveling from the $\y$- to the
$\z$-coordinate on either sides of $D_1$ we first reach either the
$\x$- or the $\z$-coordinate on both $\alphak$-curves. Suppose that we
reach the $\x$-coordinate first. The usual argument (using the tiling
of $D_2$ by elementary rectangles and the fact that in this case the
two $\x$-coordinates will have point measure $\frac{3}{4}$) provides
the appropriate point $\y '$ as the intersection points where the
$\alphak$-curves leave $D_2$.  This argument verifies the result in
this specific case, cf. Figure~\ref{f:partcase2}(e).  Similarly, if
the $\z$-coordinate is reached first, then we get a configuration
where the two $\z$-coordinates have point measure $\frac{3}{4}$ in
$D_1\cup D_2$. Once again, the $\y '$-coordinates can be easily
identified as the intersections of the $\betak$-curves with the
$\alphak$-curves when these latter leave $D_1\cup D_2$. Notice that in these
cases, depending on the position of $z_i, z_j$ with respect to the $\x$- and
the $\y$-coordinates, the domains $D_1',D_2'$ might ``spiral around'',
as it is illustrated by Figure~\ref{f:partcase2}(f).

\noindent {\bf Case 3: Examination of ${\mathfrak {N}}(m)$.}  Once again, we
subdivide our study according to the number of moving coordinates. By the fact
that we have a pair $(D_1, D_2)$ of a rectangle and a bigon, this number is
either two or three.  When it is three, the usual argument dealing with
disjoint domains (in the sense of having different moving coordinates, cf.
Figure~\ref{f:partuj}(a) for an example of intersecting interiors) proves
evenness for that subcase.  Assuming two moving coordinates, consider the case
when $D_1$ is a rectangle and $D_2$ is a bigon. (The other case is symmetric,
requiring only obvious modifications of the argument.)  Suppose that $y_i$ is
the corner of $D_2$ which moves to $z_i$. Start moving again towards $z_i$,
and distingush two cases whether we reach $z_i$ or $x_i$ first. In either case
we get a portion of an $\alphak$-- or $\betak$--curve which enters the
rectangle $D_1$ (or the bigon $D_2$ in the other case), which must eventually
leave it, producing a new intersection point $y_i'$. Notice that we get two
combinatorially different cases depending on how the arc leaves the bigon;
prototypes of the two cases are depicted by Figures~\ref{f:partuj}(b) and (c).
In the first case the domain $D_1'$ connecting $\x$ and $\y '$ is a bigon,
while from $\y '$ to $\z$ the domain $D_2'$ is a rectangle (recall that $D_1$
from $\x$ to $\y$ was a rectangle, while $D_2$ from $\y$ to $\z$ was a bigon).
In the second case the domain $D_1'$ connecting $\x$ and $\y'$ is still a
rectangle, while $D_2'$ from $\y '$ to $\z$ is a bigon. Neverthless, the same
argument as before shows that there are an even number of pairs in ${\mathfrak
  {N}}(m)$.  

Putting all three cases together, it follows that $\vert {\mathfrak
  {N}} _{\x , \z }\vert$ is even, concluding the proof of $\partialaa
_{\DD}^2 =0$.
\end{proof}
\begin{figure}[ht]
\begin{center}
\epsfig{file=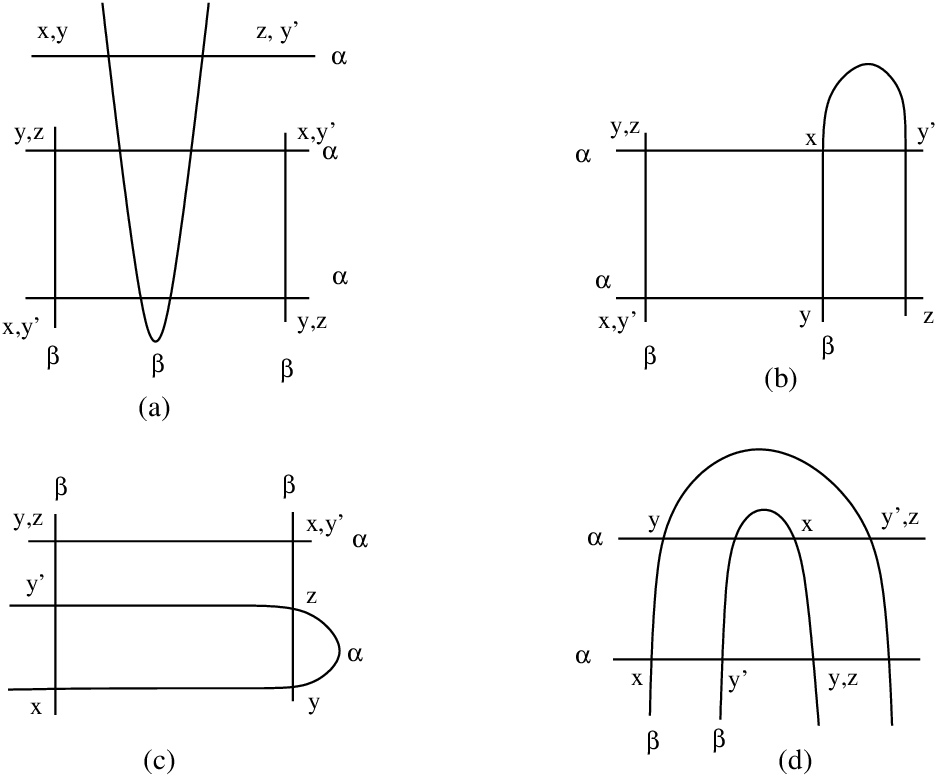, height=9cm}
\end{center}
\caption{{\bf Geometric possibilities when examining the matrix
    element $\langle \partialaa _{\DD}^2 \x , \z \rangle$.} Diagrams
  describe possibilities corresponding to Case 3 in the proof of
  Theorem~\ref{thm:chaincomplex}.}
\label{f:partuj}
\end{figure}

With the above result at hand, we have
\begin{defn}
{\whelm Suppose that $\DD$ is a nice multi-pointed Heegaard diagram. The
  \emph{combinatorial Heegaard Floer group} of $\DD$ is the homology
  group $\HFaa (\DD )=H_* (\CFaa (\DD ) , \partialaa _{\DD })$ of the
  chain complex $(\CFaa (\DD ), \partialaa _{\DD})$ defined above.}
\end{defn}

Recall that according to Definition~\ref{def:equiv} two pairs $(V_i,
b_i)$ of $\Field$--vector spaces (with $\Field = \bfz /2\bfz$) and
positive integers are equivalent if (assuming $b_1\geq b_2$) we have
that, as vector spaces, $V_1\cong V_2 \otimes (\Field \oplus \Field
)^{(b_1-b_2)}$. The equivalence class of $(V_i, b_i)$ is usually
denoted by $[V_i, b_i]$.
\begin{defn}
{\whelm Suppose that $\DD=(\Sigma , \alphak , \betak , \w)$ is a nice
  multi-pointed Heegaard diagram.  The \emph{stable Heegaard Floer
    homology of $\DD$} is defined as the equivalence class of the pair
\[
[\HFaa (\DD ) , b(\DD )],
\]
where $\HFaa (\DD )$ is the Floer homology group of $\DD$ defined as
above, and $b(\DD )$ is the cardinality of the basepoint set $\w$.  We
will denote the stable Heegaard Floer group of $\DD$ by $\HFast (\DD
)$.}
\end{defn}

\section{Nice moves and chain complexes}
\label{sec:cpx}
Although a nice move can change the chain complex derived from the
Heegaard diagram, as it will be shown in this section, the homology
of the chain complex defined in the previous section remains unchanged
under nice isotopy and handle slide and changes in a controlled manner
under nice stabilization.  Suppose therefore that $\DD_1$ is a nice
diagram and $\DD_2$ is given by the application of a nice move on
$\DD_1$. The main result of this section is summarized by the
following

\begin{thm}\label{thm:nicemoves}
If the nice move applied to $\DD_1$ to get $\DD_2$ is a nice isotopy,
a nice handle slide or a nice type-$g$ stabilization then
$( \CFaa (\DD_2),
\partialaa _{\DD_2})$ has homology isomorphic to that of 
$( \CFaa (\DD_1), \partialaa _{\DD_1})$, i.e.
\[
\HFaa (\DD_2)\cong \HFaa (\DD_1).
\]
If $\DD_2$ is given by a nice type-$b$ stabilization on $\DD_1$ then
\[
\HFaa (\DD_2 ) \cong \HFaa (\DD_1)\otimes (\Field \oplus \Field ).
\]
\end{thm}

\begin{cor}\label{c:nochange}
Suppose that the nice diagrams $\DD_1$ and $\DD_2$ are nicely
connected. Then the stable Heegaard Floer homologies $\HFast (\DD_1)$
and $\HFast (\DD_2)$ are equal.
\end{cor}
\begin{proof}
Applying an induction on the length of the chain of nice diagrams
connecting $\DD_1$ and $\DD_2$, it is enough to verify the statement
only in the case when $\DD_1$ and $\DD_2$ differ by a single nice
move.  If the nice move is a nice isotopy, a nice handle slide or a
nice type-$g$ stabilization, then (according to
Theorem~\ref{thm:nicemoves}) the Heegaard Floer homologies $\HFaa
(\DD_1)$ and $\HFaa (\DD_2)$ are isomorphic.  Since in these steps
the number of basepoints remains unchanged, we readily get that
\[
\HFast (\DD_1) = \HFast (\DD_2).
\]
If the nice move connecting $\DD_1$ and $\DD_2$ is a nice type-$b$ 
stabilization, then (once again, by Theorem~\ref{thm:nicemoves}) 
we have that 
$\HFaa (\DD_2 ) \cong \HFaa (\DD_1)\otimes (\Field \oplus \Field )$,
while (by the definition of a nice type-$b$ stabilization) we also get
that $b(\DD_2)= b (\DD_1)+1$.
According to the definition of the stable Heegaard Floer invariants, therefore
we conclude that 
\[
\HFast (\DD_1)=\HFast (\DD_2)
\]
in this case as well,
concluding the proof of the corollary.
\end{proof}

In proving Theorem~\ref{thm:nicemoves} we consider first the cases
where $\DD_2$ is given by a nice isotopy or a nice handle slide on
$\DD_1$, which will be followed by the (much simpler) cases of
stabilizations.  The strategy in the first two cases (isotopy and
handle slide) will be the following.  We will specify a subcomplex
$(K, \partialaa _K)$ of $(\CFaa (\DD_2), \partialaa _{\DD_2})$,
providing a quotient complex $(Q, \partialaa _Q)$. A relatively simple
argument will show that $H_*(K, \partial _K)=0$, and that $Q$ (as a
vector space) is isomorphic to $\CFaa (\DD_1)$. Based on the special
circumstances then we will show that we can pick a vector space
isomorphism between $\CFaa (\DD_1)$ and $Q$ which is, in fact, an
isomorphism of chain complexes.  Theorem~\ref{thm:nicemoves} then
follows quickly. In fact, the vector space $\CFaa$ comes with a
natural basis (given by the set of the generators), and since $K$ is
also defined using a subset of the generators, the quotient complex
$(Q, \partialaa _Q )$ also comes with a natural basis. It will be
therefore useful to describe $\partialaa _Q$ explicitly in this
special basis. We start our discussion with some formal aspects of the
situation.

\subsection{Formal aspects}
\label{ssec:formasp}
Suppose that a chain complex $(B, \partial _B)$ is given, and $B$ has
a preferred basis ${\mathcal {B}}$.  Assume that for $\x, \y \in
{\mathcal {B}}$ the matrix element $\langle \partial _B \x,\y
\rangle$ (defining the boundary map $\partial _B$) is given by the
(mod 2 defined) number $n_{\x\y}$.  Suppose furthermore that
${\mathcal {B}}$ can be given as a disjoint union ${\mathcal
  {B}}_1\cup {\mathcal {K}}\cup {\mathcal {L}}$, and there is a fixed
bijection $J\colon {\mathcal {K}}\to {\mathcal {L}}$. In addition,
assume that on the vector space spanned by basis vectors corresponding
to the elements of ${\mathcal {L}}$ there is a $\bfq$--filtration. (In
the following, vectors corresponding to elements of ${\mathcal {B}}$
will be denoted by boldface letters, while their linear combinations
by usual italics.)  Suppose that for a basis element $\kaa$
corresponding to an element in ${\mathcal {K}}$ we have that
\begin{align}
\partial _B \kaa  = & J(\kaa )\label{eq:filt} \\
& +
  {\mbox {higher filtration level in }} {\mathcal {L}} \nonumber \\
& + 
  {\mbox {further terms
      with coordiates only in }} {\mathcal {B}}_1\cup {\mathcal {K}}.
\end{align}
Consider now the subcomplex $K$ generated by the vectors corresponding
to the elements of ${\mathcal {K}}$, together with their $\partial
_B$--images, and let $Q=B/K$. From the property given by
Equation~\eqref{eq:filt} it follows that the quotient complex
$Q$ is generated by the vectors
$\{ \x +K\mid \x \in {\mathcal {B}}_1\}$, and is equipped with a
differential given by
\[
\partial _Q (\x +K)=\partial _B \x +K,
\]
for $\x \in {\mathcal {B}}_1$. Suppose now that $\partial _B \x = b_1 + k + l$,
where $b_1, k, l$ are vectors in the subspaces of $B$ spanned by basis
elements corresponding to elements of ${\mathcal {B}}_1, {\mathcal
  {K}}$ and ${\mathcal {L}}$, respectively. By the existence of the
filtration there are  further vectors $k', k''$ (also in the subspace
spanned by ${\mathcal {K}}$)
such that
\[
\partial _B (\x +k')=b_1'+k'',
\]
i.e., the coordinates in ${\mathcal {L}}$ can be eliminated.
Therefore $\partial _B \x = b_1 ' +k''+\partial _B k'$, which is
of the shape $b_1'+K$. This identity means that for $\x \in {\mathcal
  {B}}_1$we have $\partial _Q(\x +K)=b_1' +K$ (with $b_1'$ having
coordinates from ${\mathcal {B}}_1$ only).

Our next goal is to determine the matrix $(\langle \partial _Q (\x
+K), \y +K \rangle )_{\x, \y \in {\mathcal {B}}_1}$ defining the
boundary map $\partial _Q$. Recall that the boundary operator
$\partial _B $ is given by the matrix $N=(n _{\x ,\y})_{\x,\y \in
  {\mathcal {B}}}$ in the basis ${\mathcal {B}}= {\mathcal {B}}_1 \cup
{\mathcal {K}}\cup {\mathcal {L}}$; the matrix $N$ can be viewed as a
$3\times 3$ block matrix. The blocks in the block matrix will be
indexed by the sets of basis vectors they connect, for example the
upper right block is $N_{{\mathcal {B}}_1, {\mathcal {L}}}$. We also
assume that the basis vectors in ${\mathcal {L}}$ are ordered
according to their filtration.  Notice that by \eqref{eq:filt} the
block $N_{{\mathcal {K}}, {\mathcal {L}}}$ is lower triangular, with
only 1's in the diagonal, hence can be written as $I+T$, where $T$ is a
strictly lower triangular matrix. Note also that $I+T$ is obviously
invertible: $(I+T)^{-1}=\sum _{k=0}^{\infty} T^k$ where the sum is
finite, since $T$ is strictly lower triangular.

In order to determine the matrix of the boundary map $\partial _Q$ in the
basis $\{ \x +K\mid \x \in {\mathcal {B}}_1\}$, we need to examine the 
linear transformation $\partial _B$ in another
basis (since the basis vectors in ${\mathcal {K}}\cup {\mathcal {L}}$ do not 
generate a subcomplex). Let us therefore denote the set of vectors
\[
\{ \partial _B \kaa \mid \kaa \in {\mathcal {K}}\}
\]
by $\partial _B {\mathcal {K}}$.  If we take the matrix of $\partial _B$ 
in the basis
${\mathcal {B}}_1\cup {\mathcal {K}}\cup \partial _B{\mathcal {K}}$,
then the matrix of $\partial _Q$ (in the basis 
$\{ \x +K \mid \x \in {\mathcal {B}}_1\}$) is simply the upper left block of this matrix
(written in a block form corresponding to the blocks of basis vectors suggested by
the notation). In order to determine this upper left block, let us denote the
matrix of the base change
\[
{\mathcal {B}}_1\cup {\mathcal {K}}\cup \partial _B {\mathcal {K}} \to
{\mathcal {B}}_1\cup {\mathcal {K}}\cup {\mathcal {L}}
\]
by $G$. (In the following linear algebra considerations we always regard
vectors written in given bases as row vectors and the application of a linear transformation
will correspond to matrix multiplication from the right with the matrix of the linear transformation
in the given basis.)

\begin{lem}
The matrix $G$ is of the form
\[
\begin{pmatrix}
I & 0 & 0\\
0 & I & 0\\
U & V & Z
\end{pmatrix}
\]
where 
$U=N_{{\mathcal {K}}, {\mathcal {B}}_1}$, $V=N_{{\mathcal {K}}, {\mathcal {K}}}$
and $Z=N_{{\mathcal {K}}, {\mathcal {L}}}$.
\end{lem}
\begin{proof}
Notice that the basis vectors given by the elements of ${\mathcal {B}}_1$ and 
${\mathcal {K}}$ are in both bases, hence they map into themselves. In matrix terms,
this means that the  rows corresponding to these basis elements contain a single
1 each, and 0's otherwise. Finally, if we take an element $\partial _B \kaa$ and apply 
$G$ to it, we get its expansion in the basis ${\mathcal {B}}_1\cup {\mathcal {K}}\cup {\mathcal {L}}$,
where the coordinates exactly give the matrices 
$N_{{\mathcal {K}}, {\mathcal {B}}_1}, N_{{\mathcal {K}}, {\mathcal {K}}}$
and $N_{{\mathcal {K}}, {\mathcal {L}}}$. This observation then concludes our argument.
\end{proof}
Now it is easy to verify that 
$G^{-1}$ is  equal to
\[
\begin{pmatrix}
I & 0 & 0\\
0 & I & 0\\
-Z^{-1}U & -Z^{-1}V & Z^{-1}
\end{pmatrix}
\]
Consequently, the matrix $M$ of $\partial _B$ in the new basis
${\mathcal {B}}_1\cup {\mathcal {K}}\cup \partial _B{\mathcal {K}}$ is
equal to $GNG^{-1}$, hence its upper left block (representing
$\partial _Q$ in the basis $\{ \x +K \mid \x \in {\mathcal {B}}_1\}$)
is equal to 
$N_{{\mathcal {B}}_1, {\mathcal {B}}_1}-N_{{\mathcal {B}}_1, {\mathcal
    {L}}}\cdot Z^{-1}U$. Since $Z=N_{{\mathcal {K}}, {\mathcal {L}}}=I+T$ is invertible and
$Z^{-1}=(I+T)^{-1}=\sum _{k=0}^{\infty} T^k$, as a conclusion we get
\begin{lem}\label{l:qleir}
The matrix of $\partial _Q$ in the basis $\{ \x +K \mid \x \in
{\mathcal {B}}_1\}$ is equal to
\[
N_{{\mathcal {B}}_1, {\mathcal {B}}_1}-
\sum _{k=0}^{\infty} N_{{\mathcal {B}}_1, {\mathcal {L}}}\cdot T^k\cdot
N_{{\mathcal {K}}, {\mathcal {B}}_1}.
\]
\qed
\end{lem}

By its definition, $K$ is a subcomplex of $(B, \partial _B)$ (with the
boundary map inherited from $\partial _B$), and by 
Property~\eqref{eq:filt} it easily follows that $H_*(K)=0$.
Now the short exact sequence 
\[
0\to K \to B \to Q \to 0
\]
of chain complexes induces an exact triangle on the homologies, which (by
the vanishing of $H_*(K)$) provides an isomorphism between $H_*(B,
\partial _B)$ and $H_*(Q, \partial _Q)$.  

 In fact, the two chain
complexes $(B, \partial _B)$ and $(Q, \partial _Q)$ are chain homotopy
equivalent. Indeed, define the map $F\colon B \to Q$ by sending a
basis element $\x \in {\mathcal {B}}_1$ to $\x +K\in Q$ and all
elements of $K$ into 0. Let $G\colon Q \to B$ on $\x +K$ for
$\x \in {\mathcal {B}}_1$ be defined by $G(\x +K)=\x +k'$ where $k'$
is in the span of the basis vectors corresponding to elements of
${\mathcal {K}}$ and has the property that $\partial _B (\x +k')$ has
no coordinates in ${\mathcal {L}}$. Notice that such an element $k'$
always exists by Property~\eqref{eq:filt} of the filtration, and the map is
well-defined because of the uniqueness of $k'$: if $k'$ and $k''$ both 
satisfy these conditions, then $\partial _B (k'+k'')$ has no coordinates
in ${\mathcal {L}}$, contradicting Property~\eqref{eq:filt} of the filtration
unless $k'+k''=0$, i.e. $k'=k''$ (mod 2).
\begin{lem}
The maps $F$ and $G$ provide chain homotopies between the chain
complexes $(B, \partial _B)$ and $(Q, \partial _Q )$.
\end{lem}
\begin{proof}
Both $F$ and $G$ are chain maps,
and it is easy to see that $F\circ G$ is the identity on $Q$.  We claim that the map
$G\circ F$ is chain homotopic to the identity id$_B$. Indeed, consider the map
$H\colon B\to B$ defined as 0 on the vectors spanned by ${\mathcal
  {B}}_1\cup {\mathcal {K}}$, and sending $\partial _B\kaa $ to $\kaa$
for every element $\partial _B\kaa $ of $\partial _B{\mathcal {K}}$. 
Then 
\[
G\circ F + id_B=H\circ \partial _B + \partial _B \circ H
\]
holds for every basis element: for an element $\kaa \in {\mathcal {K}}$
we have $F(\kaa )= H(\kaa )=0$, hence both sides of the above equality
take the value $\kaa$ on $\kaa$. Similar simple argument works for an element of the form 
$\partial _B \kaa $ (with $\kaa \in {\mathcal {K}}$), while we need Property~\eqref{eq:filt}
to verify the equality for $\x \in {\mathcal {B}}_1$.
(Remember that we are working over the field $\Field = \Z /2\Z$ of two
elements, so the signs are not relevant to  our discussions.)
\end{proof}

In conclusion, suppose that the chain complex $(B, \partial B)$ comes
with the partition ${\mathcal {B}}_1\cup {\mathcal {K}}\cup {\mathcal
  {L}}$ of its basis ${\mathcal {B}}$, together with the map
$J\colon {\mathcal {K}}\to {\mathcal {L}}$ and the filtration
on the subspace spanned by ${\mathcal {L}}$ satisfying Property~\eqref{eq:filt} 
then 

\begin{prop}\label{p:homeq}
The factor complex $(Q, \partial _Q)$ is chain homotopy equivalent 
to $(B, \partial _B)$. \qed
\end{prop}

\subsection{Invariance  under nice isotopies}

In the following three subsections we will compare domains and
elementary domains of two nice (multi-pointed) Heegaard diagrams: $\DD$ will
denote the diagram before, while $\DD'$ the diagram after the nice
move. To keep the arguments transparent, we will adopt the convention
that elementary domains in $\DD $ (in $\DD'$) will be typically
denoted by $D_i$ (and $D_i'$, resp.), while domains in $\DD $ (and in
$\DD '$) will be typically denoted by ${\mathcal {D}}_i$ or by $\Delta
$ (and ${\mathcal {D}}_i'$, $\Delta '$ for $\DD '$).

We start with examining nice isotopies.  Assume that we isotope an
$\alphak$--curve $\alpha_1$ by a nice isotopy; let $\alpha_1'$ denote
the result of the isotopy and let $\DD$ and $\DD '$ denote the
Heegaard diagrams before and after the isotopy, respectively.  Recall
that the isotopy is determined by a nice arc $\gamma$ (along which the
finger move is performed).  Near every intersection point $x_i$ of
$\gamma$ with a $\betak$--curve $\beta _j$ there are two new
intersection points between $\alpha_1'$ and $\beta _j$; the two points
will be denoted by $f_i$ and $e_i$.  We follow the convention that
when using the induced orientation on the bigon (created by the finger
move) with corners $f_i$ and $e_i$, we traverse from $f_i$ to $e_i$ on
$\alpha_1'$.

To apply the result of the previous subsection, let $B=\CFaa (\DD ')$,
with generators $\Gen '$ and notice that the elements of $\Gen $
(i.e., generators originated from the Heegaard diagram $\DD$) can be
viewed naturally as elements of $\Gen '$. This set $\Gen $ will play
the role of ${\mathcal {B}}_1$, while ${\mathcal {K}}$ (and ${\mathcal
  {L}}$) will be the set of basis vectors having an $f_i$ ($e_i$,
resp.) as a coordinate.  The map $J(f_i\y)=e_i\y$ determines a
bijection $J\colon{\mathcal{K}}\longrightarrow{\mathcal{L}}$ which,
(as we shall see in Lemma~\ref{l:specif}) satisfies the requirements
of Subsection~\ref{ssec:formasp}.  A filtration on the vector space
generated by the elements of ${\mathcal {L}}$ is given by the linear
ordering of the points along the arc of $\alpha_1'$ containing all
these points.  The property required by Equation~\eqref{eq:filt} is
given (in fact, in a much stronger form) by the following result.

\begin{lem}\label{l:specif}
Let $f_i\y \in {\mathcal {K}}$ and $e_j\y '\in {\mathcal {L}}$ denote
elements of $\Gen '$ which contain $f_i$, resp. $e_j$ as a
coordinate. (As always, the same symbols also denote the corresponding
basis vectors of $\CFaa (\DD ')$.) Then the set
${\mathfrak{M}}_{f_i\y, e_j\y'}$ is nonempty if and only if $i=j$ and
$\y =\y '$. In this case ${\mathfrak{M}}_{f_i\y , e_i\y}$ consists of
a single bigon.
\end{lem}
\begin{proof}
  Consider any ${\mathcal D} '\in{\mathfrak{M}}_{f_i\y,e_j\y'}$; by
  our orientation convention, the intersection $\partial{\mathcal D}
  '\cap\alpha_1'$ is an embedded arc from $f_i$ to $e_j$.  First we
  wish to identify this arc. In fact, there are two paths on
  $\alpha_1'$ connecting $f_i$ and $e_j$: one passes along the part of
  the curve $\alpha_1$ not affected by the finger move, while the
  other one is contained by the part created by the finger move.  By
  Lemma~\ref{l:mindket}, there is a basepoint next to the first path
  (on its either side), hence that cannot be used when considering
  empty bigons or rectangles connecting $f_i$ with some
  $e_j$. Therefore, $\partial{\mathcal D} '\cap\alpha_1'$ must be the
  second path.  The orientation convention shows that ${\mathcal D}'$
  contains the new elementary bigon $B'$ constructed during the
  isotopy with multiplicity $1$. This shows that $f_i \y$ and $e_j \y
  '$ differ only at one coordinate, so $\y = \y '$, and moreover that
  $f_i$ and $e_j$ are on the same $\betak$--circle.  Traversing along
  that $\betak$--circle (starting at $f_i$) the first intersection
  with $\alpha _1'$ is by definition $e_i$. The fact that ${\mathcal
    D}'$ has two convex corners now implies that ${\mathcal D}'$ in
  fact coincides with the bigon from $f_i\y$ to $e_i\y$.
\end{proof}

We define the subcomplex $K\subseteq \CFaa (\DD ')$ as in
Subsection~\ref{ssec:formasp}, i.e., it is generated by the basis
vectors corresponding to the elements of ${\mathcal {K}}$ together
with their $\partialaa _{\DD '}$--images, and then we take the
quotient complex $(Q, \partial _Q)$.  Consider now the map $F\colon
\CFaa (\DD)\to Q$ defined by
\[
\x \mapsto \x +K.
\] 
Since by Lemma~\ref{l:specif} any vector in $K$ has a component which
contains one of $f_i $ or $e_j$ as a coordinate, $F$ is clearly
injective. Now, an element of $\Gen '$ is either in $\Gen $ or
contains an $f_i$-- or an $e_i$--coordinate, hence $\dim \CFaa (\DD) +
\dim K = \dim \CFaa (\DD ')$. Thus, the map $F$ is a vector space
isomorphism.  Next we want to show that the map $F$ is a chain map,
that is, $F(\partialaa _{\DD}(\x ))=\partialaa _Q (F(\x ))$.  In order
to achieve this, we need a geometric description of the boundary map
$\partialaa _Q$. We start with a definition.
\begin{defn}
  \label{def:chain} {\whelm Fix $\x, \y \in \Gen \subset \Gen '$ and
    call a sequence $C=({\mathcal D}_1', {\mathcal D}_2', \ldots ,
    {\mathcal D}_{n}')$ of domains in $\DD '$
a \emph{chain of length $n$ connecting $\x$
      and $\y$} if for $i=1, \ldots , n-1 $ we have $\kaa _i=f_i\kaa
    _i' \in {\mathcal {K}}$, $\laa _i=J(\kaa _i) = e_i\kaa _i'\in
    {\mathcal {L}}$ and 
\[
{\mathcal D}'_1\in {\mathfrak{M}}^{\DD '}_{\x
      \laa_1}, {\mathcal D}'_2\in {\mathfrak{M}}^{\DD '}_{\kaa_1 \laa
      _2}, \ldots , {\mathcal D}'_{n-1}\in {\mathfrak {M}}^{\DD
      '}_{\kaa _{n-2} \laa _{n-1}}, {\mathcal D}'_{n}\in
    {\mathfrak{M}}^{\DD '}_{\kaa _{n-1}, \y} .
\]
 The definition allows $n=1$,
    when the chain consists of a single element ${\mathcal D}'\in
    {\mathfrak{M}}_{\x,\y}$.  A domain ${\mathcal D}_C'$ can be
    associated to a chain $C$ by adding the domains ${\mathcal D}_i'$
    appearing in $C$  and subtracting the bigons in
    ${\mathfrak {M}}^{\DD '}_{\kaa _i , J (\kaa _i)}$ for $\kaa _i$ appearing
    in the chain.}
\end{defn}
The interpretation of the product in Lemma~\ref{l:qleir} then easily
implies
\begin{prop}\label{p:faktor1}
For $\x,\y \in \CFaa (\DD )$ the matrix element $\langle
\partialaa _Q(\x +K), \y +K\rangle$ is equal to the (mod 2) number of
chains connecting $\x$ and $\y$. 
\end{prop}
\begin{proof}
The number of chains is equal to the cardinality of the set
\[
\mxy ^{\DD '}\cup \bigcup _{(\laa_1, \ldots , \laa _n)}{\mathfrak
  {M}}^{\DD '}_{\x, \laa_1}\times {\mathfrak{M}}^{\DD '}_{\kaa_1,\laa _2}\times \cdots
\times {\mathfrak{M}}^{\DD '}_{\kaa _n,\y}.
\]
By the definition of the matrix elements, this number is (mod 2) equal
to the $(\x,\y)$--element of the matrix $N_{{\mathcal {B}}_1,
  {\mathcal {B}}_1}- \sum _{k=0}^{\infty} N_{{\mathcal {B}}_1,
  {\mathcal {L}}}\cdot T^k\cdot N_{{\mathcal {K}}, {\mathcal {B}}_1}$,
which, by Lemma~\ref{l:qleir} is equal to the matrix element of the
boundary operator $\partialaa _Q$. This identity verifies the
statement.
\end{proof}

\begin{rem}
{\whelm Notice that according to Lemma~\ref{l:specif}, the space
  ${\mathcal {M}}^{\DD '}_{\kaa_1, \laa _2}$ for $\laa _2=J (\kaa _2)$
  (and $\kaa _1\neq \kaa _2$) is necessarily empty, which implies that
  in fact any chain is of length one or two.}
\end{rem}

Our final aim in this subsection is to relate, for any
$\x,\y\in\Gen $, the set $\mxy ^{\DD}$ with the set of chains in
$\DD '$ connecting $\x$ and $\y$.  As we already explained, the
intersection points $\x, \y$ can be regarded as elements of $\Gen '$;
the set of domains connecting them therefore will be denoted by 
$\pi_2^{\DD }(\x,\y)$ and $\pi_2^{\DD '}(\x,\y)$, respectively,
indicating the diagram we are working in.  

Given $\x,\y\in{\mathcal{S}}$, we define a map
\[
\Phi \colon \pi_2 ^{\DD '}(\x,\y ) \to \pi_2 ^{\DD}(\x,\y )
\]
as follows. Recall that the nice isotopy is defined by a nice arc
$\gamma$; let us consider an $\epsilon$--neighbourhood of $\gamma$ and
suppose that the finger move took place in this neighbourhood.  Let
$\{D_i\}_{i=1}^m$ denote the set of elementary domains for $\DD $.
For each $i=1,\dots, m$, choose a point $p_i$ in $D_i$, which is not
in the $\epsilon$-neighborhood of $\gamma$.  Consider a domain $\Delta
'\in\pi_2^{\DD '}(\x,\y )$ and define
\[
\Phi(\Delta ')=\sum_{i=1}^m n_{p_i}(\Delta ')\cdot D_i.
\]
According to the next lemma, the map $\Phi$ is well-defined, i.e.  is
independent of the choice of $p_i$.
\begin{lem}
  \label{l:PhiWellDefined}
  Let $p$ and $q$ be two points in the Heegaard surface
  which can be connected by a path $\eta$
  which crosses none of the $\alphak$-- or $\betak$--circles for
  $\DD $; i.e. $p$ and $q$ lie in the same elementary domain of  $\DD $
  (but the path $\eta$ might cross $\gamma$, and hence $p$ and $q$
  might be in different elementary domains of $\DD '$). Then, given
  $\x,\y\in{\mathcal{S}}$, and $\Delta '\in\pi_2^{\DD '}(\x,\y)$,
  we have that
  $n_{p}(\Delta ')=n_{q}(\Delta ')$.
\end{lem}

\begin{proof}
  The path $\eta$ crosses $\alpha_1'$ some even number of times
  (with intersection number zero), in the support of the
  isotopy. However, by the hypothesis on $\x$ and $\y$, we know that
  $\Delta '$ has no corners on this portion of $\alpha_1'$, hence the result
  follows readily.
\end{proof}

It is easy to see that $\Phi(\Delta ')$ does indeed give an element of
$\pi_2^{\DD}(\x,\y)$: thinking of the condition that $\Delta
'\in\pi_2^{\DD '}(\x,\y)$ as a system of linear equations
parameterized by intersection points between the circles of $\DD '$
(Equation~\eqref{eq:DomainFromXtoY}), those equations are easily seen
to hold at the intersection points between the circles of $\DD$ as
well.

\begin{lem}\label{l:bijmas}
The map $\Phi$ is a bijection between the two sets of connecting domains.
In addition, $\mu (\Phi(\Delta '))=\mu (\Delta ')$ for all $\Delta '\in 
\pi_2 ^{\DD '}(\x,\y)$.
\end{lem}
\begin{proof}
  Recall that the nice arc $\gamma$, thought of as a curve in $\DD$,
  starts out at $\gamma (0)\in \alpha _1$, and we denoted the
  elementary domain having $\gamma (0)$ on its boundary (but disjoint
  from $\gamma (t)$ for small $t$) by $D_1$. Then $\gamma$ proceeds
  through a domain $D_2$, crosses some further collection of domains
  $\{ D_i\}_{i=3}^{f-1}$, and then terminates in the interior of an
  elementray domain which we  label $D_f$.  Note that the domains
  $D_i$ for $i=1,\dots,f$ are not necessarily distinct. The elementary
  domains $D_1$ and $D_f$ are replaced by domains $D_1'$ and $D_f'$ in
  $\DD '$.  The new diagram $\DD '$ contains also a sequence of new
  elementary domains, which consist of a sequence of rectangles
  $\{R_i '\}_{i=1}^{n}$ (in a neighborhood of $\gamma$), terminating in
  a bigon $B'$ (which contains $\gamma(1)$ in its interior). Other
  elementary domains in $\DD '$ correspond to connected components of
  $D_i\setminus \gamma$, where here $D_i$ is an elementary
  domain for $\DD$. 

Note that the local multiplicities of a domain in $\pi^{\DD
  '}_2(\x,\y)$ (with no corners among the $\{e_i,f_i\}$) at these new
domains $\{R_i '\}_{i=1}^n$ and $B'$ are uniquely determined by their
local multiplicites at all the other outside regions. This follows
easily from a local analysis (using Equation~\eqref{eq:DomainFromXtoY}
at the new intersection points, along with the hypothesis that none of
these intersection points is a corner); see
Figure~\ref{f:bijmas}. This gives a map $\Psi\co
\pi^{\DD}_2(\x,\y)\longrightarrow\pi_2 ^{\DD '}(\x,\y)$ which is an
inverse to $\Phi$, proving that $\Phi$ is a bijection.

We now check that $\mu(\Phi(\Delta '))=\mu(\Delta ')$ for any
$\x,\y\in{\mathcal S}$ and $\Delta '\in \pi _2 ^{\DD '}(\x , \y )$.  Since
$\x$ and $\y$ are both in $\Gen $, the point measures at $\x$ (or $\y$) of
$\Delta '$ and $\Phi(\Delta ')$ are equal. Therefore to check that $\mu (\Phi
(\Delta '))=\mu (\Delta ')$ we only need to deal with the Euler measures. To
this end, we compare domains in $\DD$ and $\DD '$.  Recall that the arc
(thought of as supported in $\DD$) specifying the isotopy started at the
elementary domain $D_1$, crossed its first domain labelled $D_2$, and
terminated in $D_f$. Let us assume for notational simplicity that the three
elementary domains $D_1$, $D_2$, and $D_f$ are distinct. Then, in $\DD '$, the
corresponding domains $D_1'$, $D_2'$, and $D_f'$ acquire two additional
corners. (Note that $D_2'$ is not an elementary domain, as $\gamma$
disconnects $D_2$; but all its elementary components appear with the same
local multiplicity in $\Delta '$.)  Hence, for $i=1$, $2$, or $f$, we have
that
\[
e(D_i')=e(D_i)-\frac{1}{2}.
\]
Moreover, the diagram $\DD '$ contains also a new bigon $B'$, with
$e(B ')=\frac{1}{2}$. By analyzing corners, we see that if $n_i$
denotes the local multiplicities of $D_i$, and if $b$ denotes the
local multiplicity of $B'$ in $\Phi(\Delta ')$, then
$b-n_f=n_1-n_2$. All other elementary domains in $\DD '$ either are
the new rectangles, which have Euler measure zero, or they are
components of the complement $D_i\setminus \gamma$ in
$\DD$. In $\Phi(\Delta ')$ each such elementary domain appears with
the same local multiplicity as $D _i'$ had in $\Delta '$; moreover the
sum of the Euler measures of these componets add up to the Euler
measure of $\Delta '$.  Putting these observations together, we
conclude that $e(\Phi(\Delta ' ))=e(\Delta ')$ (see
Figure~\ref{f:bijmas} for an illustration).  Note that we have assumed
that $D_1$, $D_2$, and $D_f$ are all distinct.  The above discussion
can be readily adapted to the case where this does not hold (e.g. if
$D_1=D_f\neq D_2$, then $D_1'=D_f'$ acquires four extra corner points,
and $e(D _1')=e(D_1)-1$).
\end{proof}
\begin{figure}[ht]
\begin{center}
\input{bijmas.pstex_t}
\end{center}
\caption{{\bf Bijectivity of domains.}}
\label{f:bijmas}
\end{figure}
Our next proposition will make use of the following result of Sarkar.
Note that Sarkar's proof is combinatorial, derived using properties of
the Euler measure and the point measure (i.e. taking the definition of
Maslov index as in Equation~\eqref{eq:MaslovIndex}).

\begin{thm}(Sarkar, \cite[Theorems~3.2~and~4.1]{sarka})\label{thm:sark}
Suppose that $\DD =(\Sigma , \alphak , \betak , \w )$ is a nice
diagram, and define the Maslov index of a domain using the
combinatorial formula of Equation~\eqref{eq:MaslovIndex}.  Then the
Maslov index $\mu$ is additive, that is, if ${\mathcal {D}}_1\in \pi_2
(\x, \y )$ and ${\mathcal {D}}_2 \in \pi_2 (\y ,\z )$ then for 
${\mathcal {D}}_1 +
{\mathcal {D}}_2 \in \pi_2 (\x ,\z )$ we have
\[
\mu ({\mathcal D}_1 + {\mathcal D}_2 ) =\mu ({\mathcal D}_1) + 
\mu ({\mathcal D}_2).
\]
\qed
\end{thm}
Lemma~\ref{l:bijmas} has then the following refinement:
\begin{prop}
  \label{prop:IdentifyChains}
  Given $\x,\y\in\Gen$, there is a (canonical) identification
  between ${\mathfrak{M}}^{\DD }_{\x,\y}$ and the set of chains in $\DD '$
  connecting $\x$ to $\y$ (in the sense of Definition~\ref{def:chain}).
\end{prop}

\begin{proof}
  Recall that a chain $C$ connecting $\x, \y \in {\mathcal {S}}\subset
  \Gen '$ naturally defines a domain ${\mathcal D}'_C\in \pi _2 ^{\DD
    '}(\x, \y)$ by taking the domains of the chain with multiplicty
  one, and the bigons of ${\mathfrak{M}}^{\DD '}_{\kaa_i,\laa_i}$ with
  multiplicity $-1$. According to Theorem~\ref{thm:sark} we have that
  $\mu ({\mathcal D}'_C)=1$, hence by Lemma~\ref{l:bijmas} the domain
  $\Phi ({\mathcal D}'_C)$ also satisfies $\mu (\Phi ({\mathcal
    D}'_C))=1$. The construction of $\Phi$ and the fact that
  ${\mathcal D}'_C$ is derived from a chain implies that $\Phi
  ({\mathcal D}'_C)\geq 0$, and since $\gamma$ avoids all the
  basepoints we also get that $n_{\w} (\Phi ({\mathcal
    D}'_C))=0$. Therefore, by Proposition~\ref{p:atfog}
    we conclude that $\Phi ({\mathcal D}'_C)\in \mxy
  ^{\DD}$.

  Conversely, we start with a domain $\Delta\in \mxy ^{\DD
  }$. According to Lemma~\ref{l:bijmas}, there is a corresponding
  domain $\Delta '\in\pi_2^{\DD '}(\x,\y)$ with $\Phi(\Delta
  ')=\Delta$.  We must argue that this domain $\Delta '$ is in fact
  the domain associated to a chain $C$ (in the sense of
  Definition~\ref{def:chain}); and indeed this chain is uniquely
  determined by its underlying domain $\Delta '={\mathcal D}' _C$.  We
  continue with the notation from Lemma~\ref{l:bijmas}. The nice arc
  $\gamma$ starts at the elementary domain $D_{1}$ (for $\DD$),
  immediately crosses $D_2$, and terminates in $D_f$; $n_i$ denotes
  the local multiplicity of $\Delta '\in\pi_2^{\DD '}(\x,\y)$ at
  $D_i'$, while the local multiplicity of $\Delta '$ at $B'$ is $b$.

  {\bf{Case 1: Assume that $n_1=n_f=0$.}} 
    In this case the length of the chain is determined by its domain: if
  ${\mathcal D}'$ is the domain associated to a chain connecting $\x$ to $\y$,
  and $b$ denotes the local multiplicity of this domain at the new bigon
  $B'$, then the length of the chain (as in Definition~\ref{def:chain})
  is given by $1-b$. Since $\x,\y\in\Gen $, we have that
  $n_1-n_2=b-n_f$. Thus, since $n_2$ is at most $1$, and
  $n_1=n_f=0$, we conclude that $b=0$ or $b=-1$.
  
  Now, suppose that $\Delta \in{\mathfrak{M}}^{\DD}(\x,\y)$, and the
  corresponding $\Delta '$ (with $\Phi (\Delta ') = \Delta$) has
  $b=0$. Then, this implies that $n_2=0$, and in fact that any chain
  representing $\Delta '$ has length one.
  
  Next suppose that $\Delta '\in\pi_2^{\DD '}(\x,\y)$ has $b=-1$.  Then
  we claim that there is a unique chain whose domain coincides with
  $\Delta '$.  This follows from a case-by-case analysis, considering
  the various possibilities for the starting domain $\Delta \in
  \pi_2^{\DD}(\x,\y)$, as we will explain below.

  We start with the case where $\Delta $ is a rectangle. Any rectangle
  (such as $\Delta$) is tiled by elementary domains, each of which is
  a rectangle. Equivalently, $\Delta$ contains a grid of parallel
  $\alphak$-- and parallel $\betak$--arcs. The nice arc $\gamma$
  enters the $\alpha_1$-labelled boundary arc of $\Delta$, where it
  cannot cross any of the other parallel $\alphak$--circles, $\gamma$
  possibly crosses some of the $\betak$--arcs, and then it exits on one
  of the two $\betak$--arcs on the boundary. There are two subcases,
  according to which direction $\gamma$ turns.  More precisely, let
  $\gamma_0$ be the connected subarc of $\gamma\cap \Delta$ which
  contains the initial point of $\gamma$.  Then, $\gamma_0$ separates
  $\Delta$ into two components, one of which contains three corners of
  $\Delta$, and the other one contains only one. The component
  containing one of the corners of $\Delta$ might contain a coordinate
  of $\x$ or a coordinate of $\y$. We assume the latter case (the
  former case follows similarly). Moving $\alpha_1$ along $\gamma$ to
  get $\alpha'_1$, we find two intersection points $e_i$ and $f_i$
  which are nearest to the terminal point of $\gamma_0$. In the case
  we are considering, there is a bigon (supported inside $\Delta$)
  connecting some coordinate of $\y$, which we label $y_1$, to $f_i$.
  A diagram describing this possibility is shown by Figure~\ref{f:iso}.
  In fact, writing $\y=y_1\kaa '$, we can consider the generator $\kaa
  =f_i \kaa '$. We claim that:
  \begin{itemize}
  \item There is a chain whose underlying domain is
    $\Delta'$, gotten by a rectangle connecting
    $\x$ to $e_i \kaa '$, supported inside $\Delta$,
    followed by the bigon
    ${\mathcal B}'$ from $f_i\kaa '$ to $\y$. We call this the {\em canonical
      chain for $\Delta '$}.
  \item 
    The aforementioned canonical chain is the only chain
    connecting $\x$ to $\y$, and whose support is $\Delta '$.
  \end{itemize}

  The first claim is straightforward.  Suppose we have any chain whose
  domain coincides with $\Delta '$.  We have seen that this chain has
  length two; i.e. we can write ${\mathcal D}_1'
  \in{\mathfrak{M}}^{\DD '}_{\x,e_j{\mathbf t}}$ and ${\mathcal
    D}_2'\in{\mathfrak{M}}^{\DD '}_{f_j{\mathbf t},\y}$.  Our goal is
  to show that this chain coincides with the canonical one. To this
  end, consider the bigon ${\mathcal B}'$ (in the canonical chain)
  connecting $f_i$ to $y_1$. This in turn must contain an elementary
  bigon $B_0'$. Clearly, one of ${\mathcal D}_1'$ or ${\mathcal D}_2'$,
  call it ${\mathcal D}_i'$, must contain $B_0'$ as well; and hence
  ${\mathcal D}_i'$ must be a bigon. We argue that ${\mathcal D}_i'$
  must be supported inside ${\mathcal B}'$. To see this, note that
  ${\mathcal B}'$ has a point on its $\alphak$--boundary and another
  point on its $\betak$--boundary, which have push-offs lying outside
  the support of ${\mathcal D}_i'$ (since they are both supported
  outside $\Delta$, but not in the finger move region).  Moreover,
  since each ${\mathcal D}_i'$ contains corner points of $\Delta$,
  this actually forces ${\mathcal D}_i'$ and ${\mathcal B}'$ to have the
  same support. This also forces $i=2$, since the terminal points
  coincide, giving ${\mathcal D}_2'={\mathcal B}'$ as domains connecting
  intersection points.  It follows easily now that our chain coincides
  with the canonical chain.
  
  \begin{figure}[ht]
    \begin{center}
      \input{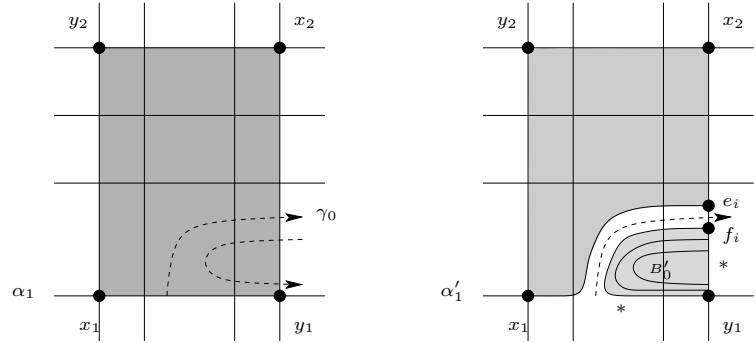}
    \end{center}
    \caption{{\bf A rectangle ${\Delta}$ turning into a chain.} At the left,
      we have a rectangle from $\x$ to $\y$, which is cut across by
      some collection of $\gamma$-arcs,  labelled by
      oriented, dashed arcs. (Here, we have chosen to illustrate the
      case of two such arcs.) The initial one is labelled $\gamma_0$.
      After the finger move is performed, we arrive at the new domain
      pictured on the right. Note the small elementary domain $B_0'$,
      which is a bigon; this is contained in the larger (shaded) bigon
      called ${\mathcal B}'$ in the text. The two regions near
      $\Delta '$ where the local multiplicity of ${\mathcal D}'_i$ (from
      the text) is guaranteed to be zero are indicated by stars.}
    \label{f:iso}
  \end{figure}
  
  A similar analysis holds in the case where $\gamma$ turns the other
  direction (except in this case the bigon ${\mathcal B}'$ connects $\x$
  to $e_i\kaa'$). Indeed, a similar analysis can be done in the case
  where $\Delta$ is a bigon, instead of a rectangle. This concludes
  the proposition, provided $n_1=n_f=0$.
  
  When the nice arc starts or terminates in a bigon, the above discussion
  requires some modifications. In this case, we no longer know that
  $n_1=n_f=0$; and indeed this means that the length of a chain is no longer
  necessarily given by $1-b$. However, we still know that $0\leq n_1+n_f\leq
  1$: $n_1=n_f=1$ would mean that both $D_1$ and $D_f$ are bigons, contained
  by $\Delta$, which can contain at most one elementary bigon. The cases
  where $n_1+n_f=0$ was treated before; so it remains to consider cases where
  $n_1=0$ and $n_f=1$ or $n_1=1$ and $n_f=0$.

  {\bf {Case 2: Suppose $n_1=0$ and $n_f=1$.}}  We consider now $\Delta\in
  \pi_2 ^{\DD }(\x,\y)$. The fact that $n_f=1$ ensures that the region $D_f$
  is contained in $\Delta$. Thus, $D_f$ cannot contain a basepoint, and
  hence it must be a bigon.  Now, considering Euler measures, we can conclude
  that $\Delta$ is a bigon as well, and the nice arc $\gamma$ starts on the
  boundary of $\Delta$. Since $\Delta$ is a bigon, we can write $\x=x_1
  \kaa'$ and $\y=y_1 \kaa'$.

  We have the following subcases (cf. also Figure~\ref{f:iso2}):
  \begin{list}
    {(2-\alph{bean})}{\usecounter{bean}\setlength{\rightmargin}{\leftmargin}}
  \item 
    \label{item:FingerLeaves}
    The nice arc $\gamma$ terminates outside of $\Delta$.
  \item 
    \label{item:FingerStays}
    The nice arc $\gamma$ is supported entirely inside $\Delta$,
    terminating in its elementary bigon.
  \item 
    \label{item:FingerLeavesButReturns}
    The nice arc $\gamma$ crosses $\Delta$, and eventually
    reenters it, terminating in its elementary bigon.
  \end{list}

  Consider first Case~(2-a) (depicted by the upper diagrams of Figure~\ref{f:iso2}). 
  Clearly, $\Delta '$ in this case has
  negative local multiplicity somewhere, and hence we can conclude
  that $\Delta '$ represents a chain of length $2$. We argue that that
  chain is uniquely determined. To this end, let $\gamma_0$ denote the
  connected component of $\Delta \cap \gamma$ containing the initial
  point of $\gamma$. The arc $\gamma_0$ disconnects $\Delta$. When we
  thicken up $\gamma_0$, we see that the endpoint of $\gamma_0$ gives
  rise to two intersection points $e_i$ and $f_i$ of $\alpha_1'$ with
  $\beta_1$.  Indeed, inside the tiling of $\Delta $, we can find a
  new elementary bigon $B_0'$. This elementary bigon $B_0'$ is
  contained in a unique bigon ${\mathcal B}'$ which contains one of
  the corners of $\Delta$: either the initial corner $x_1$ or the
  terminal one $y_1$. Assume it is the terminal corner.  Then,
  ${\mathcal B}'$ is a bigon connecting $f_i \kaa'$ to $y_1 \kaa'$.
  We claim:
  \begin{itemize}
    \item There is a chain whose underlying domain is $\Delta '$, originating from
      the bigon which connects $\x=x_1\kaa'$ to $e_i\kaa'$, supported inside $\Delta$,
      followed by the above bigon ${\mathcal B}'$ from $f_i\kaa'$ to $\y$. We
      call this the {\em canonical chain for $\Delta '$}.
    \item The canonical chain is the only chain connecting $\x$ to
      $\y$ and whose support is $\Delta '$.
  \end{itemize}
  
  The first claim is straightforward.  For the second, consider any
  chain ${\mathcal D}_1'$ and ${\mathcal D}_2'$ with the stated
  support.  Note that $B_0'$ is contained in one of ${\mathcal D}_1'$
  or ${\mathcal D}_2'$; denote the one it is contained in ${\mathcal
    D}_i'$. A geometric argument as before (using the properties that
  ${\mathcal D}_i'$ contains $B_0'$, and it has one of $x_1$ or $y_1$
  as a corner) shows that $i=2$ and indeed ${\mathcal D}_2'=B'$.  It
  is easy to conclude that the chain coincides with the canonical
  chain.

  Consider next Case~(2-b) (see the lower diagrams of Figure~\ref{f:iso2}). 
  In this case, it is straightforward to see
  that $\Delta '$ is an embedded bigon. As such, we cannot find any
  decomposition of it as a length $2$ chain; i.e. it corresponds to a
  length $1$ chain. 
  
  Finally, Case~(2-c) follows the same way as Case (2-a).
  \begin{figure}[ht]
    \begin{center}
      \input{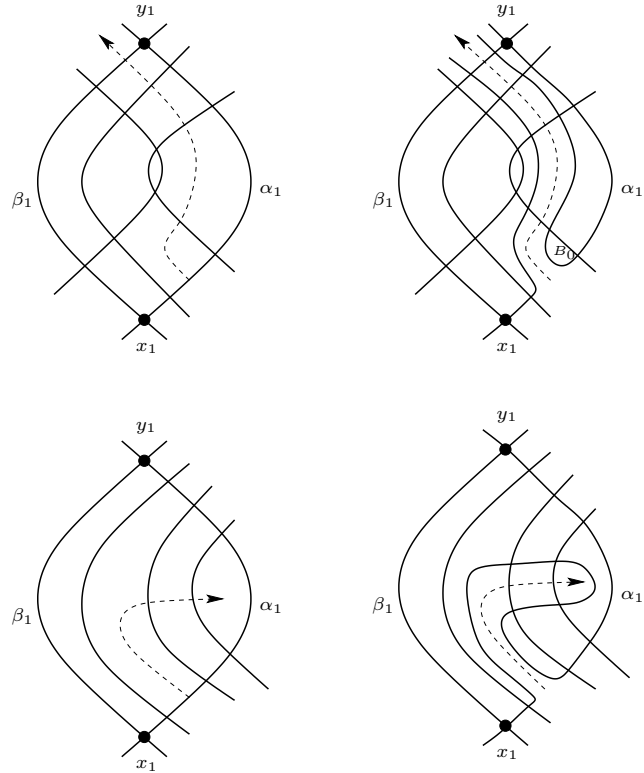}
    \end{center}
    \caption{{\bf Cases (2-a) and (2-b) of the proof of
        Proposition~\ref{prop:IdentifyChains}.}}
    \label{f:iso2}
  \end{figure}

  {\bf{Case 3: Suppose that $n_1=1$ and $n_f=0$.}}
  We can assume that $n_2=0$ (for otherwise $\Delta$ and $\Delta '$
  agree: both are bigons). 

  We have that $\Delta$ is a bigon, as it contains the elementary
  domain $D_1$ (which in turn must be a bigon).  But in this case,
  $\Delta '$ is a nonnegative domain with $\Mas(\Delta ')=1$, which
  contains the elementary bigon $B_0'$. Evidently, this forces $\Delta '$
  to be a bigon, as well.  Thus, $\Delta '$ is the domain of a length
  $1$ chain. Since it is an embedded bigon, it cannot be realized as
  the domain of a length $2$ chain.
\end{proof}

\begin{lem}\label{l:isochain}
The map $F\colon (\CFaa (\DD), \partialaa _{\DD})\to (Q ,
\partialaa _Q)$ is an isomorphism of chain complexes.
\end{lem}
\begin{proof}
Recall that $F$ is a vector space isomorphism, therefore we only need
to verify that the matrix elements $\langle \partialaa _{\DD}\x ,
\y\rangle$ and $\langle \partialaa _Q (\x +K) , \y +K \rangle$ are
equal. By Proposition~\ref{p:faktor1} the latter number has been
identified as the number of chains connecting $\x$ and $\y$ in $\DD
'$. Proposition~\ref{prop:IdentifyChains} then allows us to conclude
the proof.
\end{proof}

Now we are ready to show the isomorphism of the groups of
$\HFaa (\DD)$ and $\HFaa  (\DD ')$:

\begin{prop}\label{p:niceisso}
The homology of $(Q, \partialaa _Q)$ is isomorphic to both
\begin{enumerate}
\item $H_* (\CFaa (\DD) , \partialaa _{\DD})$ and to
\item $H_* (\CFaa (\DD ') , \partialaa _{\DD '})$.
\end{enumerate}
Consequently, if the nice diagrams $\DD$ and $\DD '$ differ by a
nice isotopy then $\HFaa (\DD )\cong \HFaa (\DD ')$.
\end{prop}
 \begin{proof}
   According to Lemma~\ref{l:isochain} the map $F$ provides an
   isomorphism between the chain complexes $(\CFaa (\DD ), \partialaa
   _{\DD})$ and $(Q, \partialaa _Q)$, and hence induces an
   isomorphism between their homologies. This verifies (1).

To prove (2) consider the exact triangle of homologies given by
the short exact sequence
\begin{equation}\label{e:hes}
0\to K \to \CFaa (\DD ' )\to Q \to 0
\end{equation}
of chain complexes. By Lemma~\ref{l:specif} the map $\partialaa _{\DD
  '}$ is injective on the basis vectors corresponding to the elements
of ${\mathcal {K}}$, and since it obviously surjects as a map from the
subspace spanned by these vectors to their $\partialaa _{\DD
  '}$--image $\partialaa _{\DD '} {\mathcal {K}}$, we get that $H_*
(K)=0$.  Exactness of the triangle associated to the short exact
sequence of \eqref{e:hes} now verifies (2).
\end{proof}

\begin{rem}
{\whelm
According to the adaptation of Proposition~\ref{p:homeq}, the 
chain complexes $(\CFaa (\DD ), \partialaa _{\DD})$ and 
$(\CFaa (\DD '), \partialaa _{\DD '})$ are, in fact, 
chain homotopy equivalent complexes.
}
\end{rem}

\subsection{Invariance  under nice handle slides}
Next we will consider the case of a nice handle slide.  The proof of
the invariance of the homology groups in this case will be formally
very similar to the case of nice isotopies.  

Let $\DD=(\Sigma,\alphak,\betak,\w)$ be a nice diagram, equipped
with an embedded arc $\delta$ connecting $\alpha_1$ to $\alpha_2$ in
an elementary rectangle $R$, and let $\DD '=(\Sigma,\alphak',\betak,\w)$
denote the diagram resulting from the nice handle slide of $\alpha_1$
over $\alpha_2$ along $\delta$. In particular, let $\alpha_1'$ denote
the curve replacing $\alpha_1$ in the new diagram. Recall that
$\alpha_1$, $\alpha_1'$ and $\alpha_2$ bound a pair-of-pants in the
Heegaard surface, which contains the handle slide arc $\delta$.

Orient $\alpha_2$ as the boundary of this pair-of-pants (which in turn
inherits an orientation from the Heegaard surface), and order the
intersection points with the $\betak$--curves according to this
orientation, starting with the point which follows the endpoint
$\delta(1)$ of the curve $\delta$.  Denote these intersection points
by $\{e_i\}_{i=1}^n$. Each intersection point $e_i\in
\alpha_2\cap\beta_{k(i)}$ has a corresponding nearest intersection
point $f_i\in \alpha_1'\cap\beta_{k(i)}$; see Figure~\ref{f:hslide}
for an illustration.

\begin{figure}[ht]
\begin{center}
\input{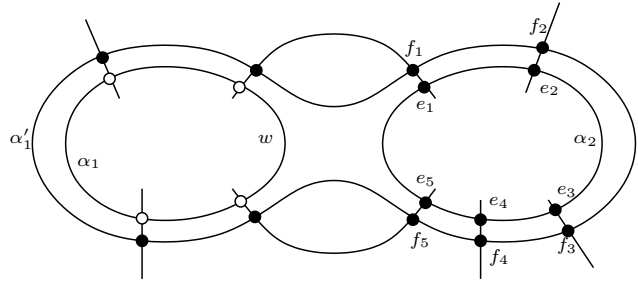}
\end{center}
\caption{{\bf {Nice handle slide.}} We have illustrated the
  pair-of-pants in a nice handle slide (of $\alpha_1$ over
  $\alpha_2$).  The transverse arcs are pieces of $\betak$--curves.  We
  have also shown here the numbering conventions for intersection
  points of the $\betak$--curves with $\alpha_2$, and with part of
  $\alpha'_1$. Intersection points of $\alpha_1$ with the $\betak$--arcs are
  indicated by hollow circles; each has a nearby matching solid circle
  (which is used in the one-to-one correspondence between generators
  for $\DD $ and certain generators in $\DD '$). }
\label{f:hslide}
\end{figure}

Let $\Gen$ and $\Gen '$ denote the set of generators for $\DD$ and $\DD '$.
Generators of $\DD '$ can be partitioned into two types:
\begin{itemize}
\item Those generators which do not contain any coordinate of the form
  $f_i$.  These generators are in one-to-one correspondence with the
  generators $\Gen$ of $\DD$ (via a one-to-one correspondence
  which moves the coordinate on $\alpha_1$ to its nearest intersection
  point on $\alpha_1'$, and which preserves all other coordinates).
  We will suppress this one-to-one correspondence from the notation,
  thinking of $\Gen$ as a subset of $\Gen '$.
\item Those generators which contain a coordinate of the form $f_i$.
  (Note that all the generators contain a coordinate of the form
  $e_j$.) We subdivide the set of these generators into two
  subsets. Let ${\mathcal{K}}$ denote those generators which contain
  $f_i$ and $e_j$ with $i>j$, and let ${\mathcal{L}}$ denote those
  generators which contain $f_i$ and $e_j$ with $i<j$.
\end{itemize}

The map $f_i e_j\x\mapsto f_j e_i\x$ determines a bijection 
$J\colon {\mathcal{K}}\longrightarrow{\mathcal{L}}$ (which, as we
shall see in Lemma~\ref{l:specifhandle}, satisfies the requirements
from Subsection~\ref{ssec:formasp}). There is a rectangle
supported in the pair-of-pants, with corners $f_i$, $e_i$, $e_j$ and
$f_j$, connecting $f_ie_j\x$ with $f_je_i\x$.  Let $K$ denote the
subspace of $\CFaa (\DD ')$ generated by the basis vectors
corresponding to the elements of ${\mathcal {K}}$ together with their
$\partialaa _{\DD '}$--images. By ordering the pairs $f_ie_j$ with the
lexicographic ordering (i.e. first according to the index of $f$,
then according to the index of $e$) we get a filtration on the vector
space spanned by the basis vectors corresponding to the elements of
${\mathcal {L}}$.

In the following we will need a more detailed understanding of the
sets ${\mathfrak{M}}^{\DD '}_{\kaa,\laa }$, leading us to the
appropriate version of Lemma~\ref{l:specif} in the context of handle
slides. Recall that a nice handle slide is defined by an arc $\delta$
contained by a single elementary rectangle $R$, with the assumption
that $D_{1}$, the domain containing $\delta (0)$ on its boundary, but
different from $R$, contains a basepoint. Let $D_{f}$ denote the
domain having $\delta (1) $ on its boundary (and different from $R$).

\begin{lem}\label{l:specifhandle}
  Suppose that $i>j, l>k$ and let $\kaa =f_ie_j\x, \laa =f_ke_l\y$
  denote elements of ${\mathcal {K}}$ and ${\mathcal {L}}$, resp.
  Then the set ${\mathfrak{M}}^{\DD '}_{\kaa , \laa}$ is nonempty if
  and only if either $i=l, j=k$ and $\x =\y$, or if $\laa $ is in a
  higher filtration level than $J(\kaa )=f_je_i\x $.  In addition, the
  set ${\mathfrak{M}}^{\DD '}_{f_i e_j \x , f_j e_i \x }$ contains a
  single element, and for all $\laa \neq f_je_i\y$ any domain
  ${\mathcal D}'\in {\mathfrak {M}}^{\DD '}_{\kaa , \laa }$ contains
  the elementary domain $D_{f}=D_{f}'$ with multiplicity 1.
\end{lem}
\begin{proof}
  We will proceed by a case-by-case analysis of possibilities for a
  domain ${\mathcal D}'\in {\mathfrak{M}}^{\DD '}_{\kaa,\laa}$. Since
  $\kaa\in{\mathcal{K}}$ and $\laa\in{\mathcal{L}}$, one of the
  coordinates (or both) on $\alpha_1'$ and $\alpha_2$ must be
  different in these intersection points.  Notice first that there are
  two arcs on $\alpha_1'$ connecting any two $f_i$ and $f_k$, but one
  of them passes by two bigons on one side and a basepoint on the
  other, hence only one of these two arcs is allowed to appear in the
  boundary of any ${\mathcal D}'\in{\mathfrak{M}}^{\DD '}_{\kaa,\laa}$
  (since ${\mathcal D}'$ contains at most one elementary bigon and no
  basepoint).

  \begin{figure}[ht]
    \begin{center}
      \input{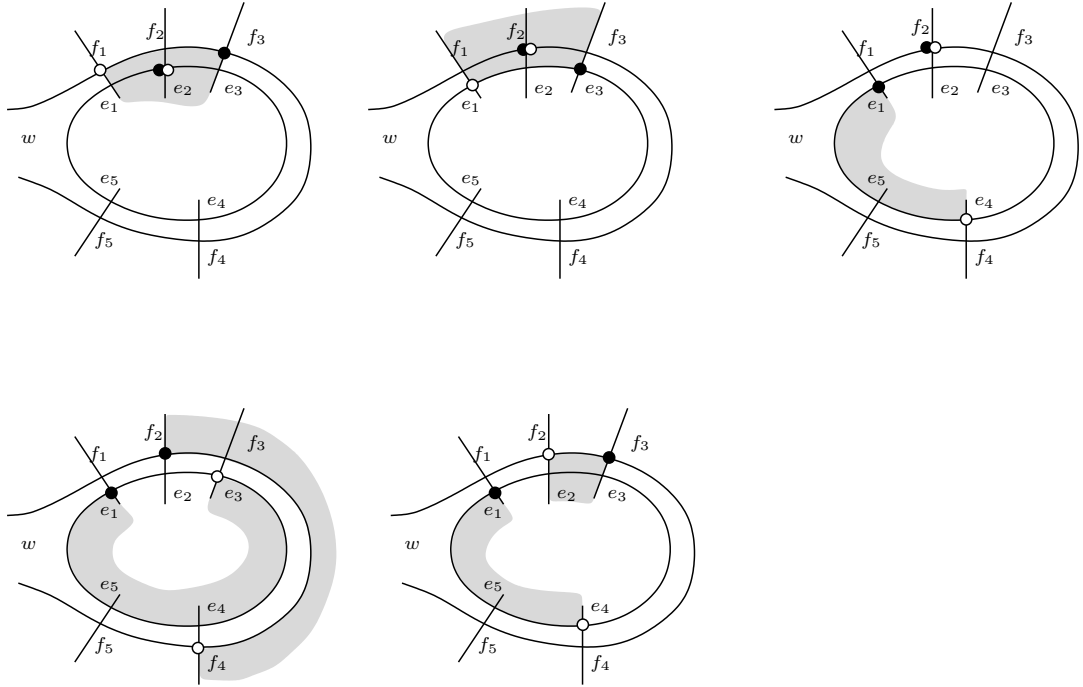}
    \end{center}
    \caption{{\bf An illustration of Lemma~\ref{l:specifhandle}}
    The shaded regions represent parts of the domain ${\mathcal D}'$.}
    \label{f:ExcludeRegions}
  \end{figure}

  Assume first that ${\mathcal D}'$ is a bigon, and the differing
  coordinate is on the curve $\alpha_1'$. The relevant moving
  coordinates are therefore $f_i$ and $f_k$, while $e_j=e_l$, and
  hence $i>j=l>k$. Considering the orientation conventions in the
  picture (and the fact that ${\mathcal D}'$ does not contain two
  elementary bigons or a basepoint), we deduce that any positive bigon
  from $f_i$ to $f_k$ with $i>k$ contains all $e_m$ with $i>m>k$. But
  this violates the condition that ${\mathcal D}'$ is an empty
  bigon. (See the first picture in Figure~\ref{f:ExcludeRegions}.)

  As the next case, assume now that ${\mathcal D}'$ is still a
  bigon, but the moving coordinates are on $\alpha_2$. If the bigon
  contains $f_j$ and $f_l$ on its boundary, then (by the orientation
  convention, together with the fact that ${\mathcal D}'$ does not
  contain the baspoint) we must have $j<l$ and for $\kaa$ and $\laa$
  to be in ${\mathcal {K}}$ and ${\mathcal {L}}$ resp., we need
  $j<i=k<l$. In this case, however, $f_i=f_k$ will be a coordinate
  contained in ${\mathcal D}'$, contradicting the fact that
  it is an empty bigon.  (See the second picture in
  Figure~\ref{f:ExcludeRegions}.)  Otherwise, if ${\mathcal D}'$ does
  not contain $f_j$ and $f_l$ (on its boundary), then either $l<j$ and
  we cannot choose $f_i=f_k$ to satisfy the constraints, or $j<l$. In
  this case, the orientation convention for ${\mathcal D}'$ going from
  $\kaa $ to $\laa$ implies that $D_{f}=D_{f}'$ is in ${\mathcal D}'$, and
  furthermore $j<i=k<l$, hence the filtration level of $\laa$ is
  higher than that of $J(\kaa )$. (See the third picture in
  Figure~\ref{f:ExcludeRegions}.)

  Assume now that ${\mathcal D}'$ is an empty rectangle, hence there
  are two coordinates which move.  If only one of them is on the
  curves $\alpha_1'$ or $\alpha_2$, then the arguments above apply
  verbatim. So consider the case when both coordinates on $\alpha_1'$
  and $\alpha_2$ move.  If $i<k$ then by the assumption on $\kaa$ and
  $\laa$ we have $j<i<k<l$, and by the orientation convention (which
  dictates that we should move from $e_j$ to $e_l$) it follows that
  (in order to keep the domain empty) ${\mathcal D}'$ must contain
  $D_{f}=D_{f}'$. (See the fourth picture in
  Figure~\ref{f:ExcludeRegions}.)  Assume now that $k<i$, so that
  ${\mathcal D}'$ contains the arc in $\alpha_1'$ between $f_k$ and
  $f_i$. This implies that ${\mathcal D}'$ also contains the arc in
  $\alpha_2$ connecting $e_k$ to $e_i$.  The emptyness of ${\mathcal
    D}'$ dictates $j\leq k<i\leq l$. If at one end we have strict
  inequality, then by the fact that ${\mathcal D}'$ has multiplicity 0
  or 1 for each elementary domain, we get that we pass on $\alpha_2$
  from $e_j$ to $e_l$ through the point $\delta(1)$.  Notice that the
  claim on the filtration level also follows at once. (See the fifth
  picture in Figure~\ref{f:ExcludeRegions}.) The last case to examine
  is when $j=k<i=l$. In this case there is a single rectangle in
  ${\mathfrak{M}}^{\DD '}_{\kaa,\laa }$ (any other domain which has
  these four corners must contain two elementary bigons).  This
  completes the proof.
\end{proof}

Notice that Lemma~\ref{l:specifhandle} verifies the property of the
map $J$ required by Equation~\eqref{eq:filt}. As before, the
subspace $K$ defined above is a subcomplex of $\CFaa (\DD ')$, and
therefore we can consider the quotient complex $(Q, \partialaa _Q)$.
The map $F\colon \CFaa (\DD ) \to Q$ is again defined by the simple
formula
\[
\x \mapsto \x+K .
\]
As for nice isotopies, we define the
chains in $\DD '$ as before:
\begin{defn}
\label{def:chain2}
{\whelm For $\x, \y \in \Gen \subset \Gen '$ a sequence $C=({\mathcal
    D}'_1, {\mathcal D}'_2, \ldots , {\mathcal D}'_{n})$ of domains in
  $\DD '$ is a \emph{chain (of length $n$) connecting $\x$ and $\y$}
  if $\kaa _i=f_ie_j\kaa _i' \in {\mathcal {K}}$, $\laa _i=J(\kaa _i)
  = f_je_i\kaa _i'\in {\mathcal {L}}$ ($i=1, \ldots , n-1$), and
\[  
{\mathcal D}_1'\in {\mathfrak{M}}^{\DD '}_{\x,\laa_1}, {\mathcal
    D}_2'\in {\mathfrak{M}}^{\DD '}_{\kaa_1,\laa_2}, \ldots ,
  {\mathcal D}_{n-1}'\in {\mathfrak{M}}^{\DD
    '}_{\kaa_{n-2},\laa_{n-1}}, {\mathcal D}_{n}'\in
  {\mathfrak{M}}^{\DD '}_{\kaa _{n-1}, \y}.
\]
  As before, the definition allows $n=1$, when the chain consists of a
  single element ${\mathcal D}'\in {\mathfrak{M}}^{\DD '}_{\x ,\y}$.
  A domain ${\mathcal D}'_C$ can be associated to a chain $C$ by
  adding the domains ${\mathcal D}'_i$ appearing in $C$ together and
  subtracting the rectangles in ${\mathfrak {M}}^{\DD '}_{\kaa _i , J
    (\kaa _i)}$ for $\kaa _i$ appearing in the chain.}
\end{defn}

The adaptation of Proposition~\ref{p:faktor1} shows that the matrix
element $\langle \partialaa _Q (\x +K), \y +K\rangle$ is determined by
the number of chains connecting $\x$ and $\y$ in $\DD '$:
\begin{prop}\label{p:faktor2}
For $\x,\y \in \CFaa (\DD )$ the matrix element $\langle \partialaa
_Q(\x +K), \y +K\rangle$ in $(Q , \partialaa _Q)$ is equal to the (mod
2) number of chains connecting $\x$ and $\y$. \qed
\end{prop}
There is a map
\[
\Phi \colon \pi_2 ^{\DD '}(\x,\y ) \to
\pi_2 ^{\DD}(\x,\y )
\]
defined analogously to the map $\Phi$ for the case of nice isotopies.
Specifically, in the present case, we have the {\em small domains} for
$\DD '$ which are those elementary domains which are supported inside
the pair-of-pants determined by $\alpha_1$, $\alpha_1'$, and
$\alpha_2$: these are the sequence of rectangles between $\alpha_1'$
and $\alpha_2$, and also the two bigons $B_u'$ and $B_d'$, formed from
the rectangle $R$ in $\DD$ containing the curve $\delta$. All other
elementary domains for $\DD '$ are called {\em large domains}. The
large domains in $\DD '$ are in one-to-one correspondence with the
domains of $\DD$.  

If $\Delta '=\sum m_i D_i'\in \pi_2^{\DD '}(\x,\y)$ is a domain in $\DD
'$, we let $\Phi(\Delta ')$ denote the sum gotten by dropping all the
terms belonging to small domains, taking the special rectangle $R$
with the same multiplicity as $B_u'$ had in $\Delta '$, and viewing
the result as a domain for $\DD$. Note that the multiplicity of $B_u'$
in any $\Delta ' \in\pi_2^{\DD '}(\x,\y)$ (for $\x,\y\in\Gen $)
coincides with the multiplicity of $B_d'$; this remark is analogous to
but somewhat simpler than Lemma~\ref{l:PhiWellDefined}, and is left to
the reader to verify.

\begin{lem}
  \label{l:bijmash}
  The map $\Phi$ is a bijection between $\pi_2 ^{\DD '}(\x,\y )$ and
  $\pi_2^{\DD }(\x,\y )$ and $\mu(\Phi(\Delta '))=\mu (\Delta ')$ for all
  $\Delta '\in\pi_2^{\DD '}(\x,\y)$.
\end{lem}
\begin{proof}
  The proof of bijectivity is analogous ot the proof of
  Lemma~\ref{l:bijmas}. The key point is that the local multiplicities
  of any $\Delta ' \in\pi_2^{\DD '}(\x,\y)$ (with $\x,\y\in\Gen$) at
  the small domains are determined by the local multiplicities of
  $\Delta '$ at the large domains.

  \begin{figure}[ht]
    \begin{center}
      \input{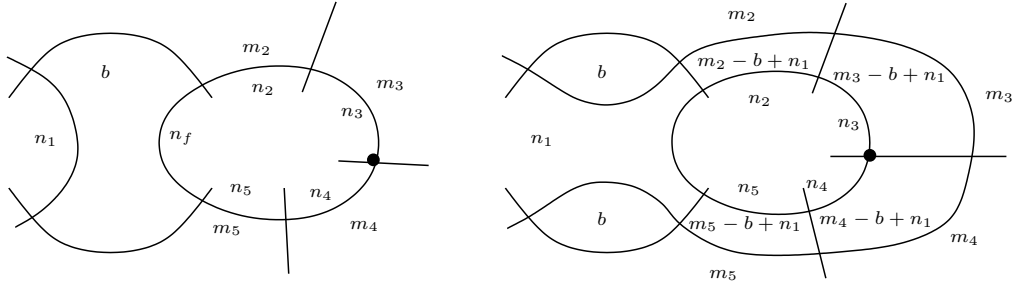}
    \end{center}
    \caption{{\bf Transforming domains under handle slides.}
    \label{f:hbijmas}
    The local multiplicities before the handle slide (on the left)
    determine the local multiplicities at all regions afterwards (on
    the right).  In particular, the generator on the left (indicated
    by the dark circle) which had point measure given by
    $\frac{n_3+n_4+m_3+m_4}{4}$ is taken to a generator (also
    indicated by the dark circle) which has point measure
    $\frac{n_3+n_4+m_3+m_4-2b+2n_1}{4}$.}
  \end{figure}

  The verification of $\mu (\Phi (\Delta '))=\mu(\Delta ')$ needs a
  little more care than was required in Lemma~\ref{l:bijmas}. It is
  not true in general that both the Euler and the point measures
  remain invariant.  Instead, we find that the elementary domain
  $D_{1}$ in $\DD $ is replaced by a new elementary domain $D_1'$ for
  $\DD '$, with $e (D_1')=e (D_1)-1$. Moreover, the rectangle $R$
  containing $\delta$ in $\DD $, which has Euler measure equal to
  zero, is replaced by two elementary bigons $B_u'$ and $B_d'$ with
  Euler measures $\frac{1}{2}$ each. Thus, if $b$ denotes the local
  multiplicity of $\Delta '\in\pi_2^{\DD '}(\x,\y)$ at $B_u'$, and
  $n_1$ is the local multiplicity of $D_1'$ in $\Delta '$, then we
  find that
  \[  
  e(\Phi(\Delta '))=n_1-b+e(\Delta ').
  \]
 Similarly, the point measure of $\Delta '$ at each coordinate of $\x$
 {\em other than the coordinate on $\alpha_2$} coincides with the
 point measure of $\Phi(\Delta ')$ at the corresponding coordinate.
 However, for the coordinate $e_i$ on $\alpha_2$, we find that
 $$n_{e_i}(\Delta ')=n_{e_i}(\Phi(\Delta '))+\left(\frac{n_1-b}{2}\right).$$
 (See Figure~\ref{f:hbijmas}.) Combining this with the analogous statement
 for the $\y$ generator, and adding, we conclude that
 $\Mas(\Phi(\Delta '))=\Mas(\Delta ')$, as claimed.
\end{proof}

Proposition~\ref{prop:IdentifyChains} has the following analogue for
handle slides (though the number of cases is slightly smaller):

\begin{prop}
  \label{prop:hIdentifyChains}
  Given $\x,\y\in\Gen$, there is a (canonical) identification
  between the elements of ${\mathfrak{M}}^{\DD}_{\x,\y}$ and the
  chains connecting $\x$ to $\y$ in $\DD '$, in the sense of
  Definition~\ref{def:chain2}.
\end{prop}
\begin{proof}
  As before, for given $\x,\y\in\Gen \subset \Gen '$, a chain $C$
  connecting $\x$ to $\y$ naturally defines a domain ${\mathcal
    D}'_C\in\pi_2^{\DD '}(\x,\y)$. Consider $\Phi({\mathcal D}'_C)$
  for this chain $C$. By Lemma~\ref{l:bijmash} combined with
  Theorem~\ref{thm:sark}, we see that $\Phi({\mathcal D}'_C)$ is an
  element in ${\mathfrak{M}}^{\DD}_{\x,\y}$.

  Conversely, start with $\Delta\in{\mathfrak{M}}^{\DD}_{\x,\y}$.
  According to Lemma~\ref{l:bijmash}, there is 
  $\Delta '\in\pi_2^{\DD '}(\x,\y)$ with $\Phi(\Delta ')=\Delta$.
  We claim that $\Delta '$ is the domain associated to a chain, and indeed
  that the chain is uniquely determined by its underlying domain.
  
  Continuing with notation from Lemma~\ref{l:bijmash}, there are
  domains $D_{1}$ and $D_{f}$ which contain $\delta
  (0)$ and $\delta (1)$ on their boundary, but are different from
  the rectangle $R$ containing $\delta$. By hypothesis,
  the local multiplicity $n_1$ of $\Delta '$ at $D_{1}'$ vanishes. 
  We will also consider the local multiplicity $b$ at 
  the two bigons $B_u'$ and $B_d'$.

  {\bf {Case 1: $n_f=0$ and $b=0$}.}  The condition that $b=0$ ensures
  that the length of the chain is one. Thus, in this case, $\Delta '$
  is the domain of a chain of length one connecting $\x$ to $\y$.
  
  {\bf {Case 2: $n_f=0$ and $b=1$}.}  Again, $n_f=0$ ensures that the
  length is at most two.  Consider $\Delta$. Letting $R$ be the domain
  in $\DD$ containing the nice arc $\delta$, the fact that $b=1$
  ensures that the local multiplicity of $\Delta $ at $R$ is
  $1$. Moreover, the local multiplicity of $\Delta $ at $D_1$ and
  $D_f$ are both zero. It follows that $\Delta $ is a rectangle with
  boundary on $\alpha_1$ and $\alpha_2$. The top 
  two diagrams of Figure~\ref{f:hIdentifyChains}  illustrate this case.

  Note that $\Delta '$ contains two elementary bigons ($B_u'$ and
  $B_d'$), and hence it follows that it must correspond to a length
  two chain: ${\mathcal D}_1'$ contains one of the bigons and
  ${\mathcal D}_2'$ contains the other one.
  
  {\bf {Case 3: $n_f=1$ and $b=0$}.}
  The condition that $b=0$ ensures that the length of the chain is one
  (i.e. this case is formally just like Case 1).

  {\bf {Case 4: $n_f=1$ and $b=1$}.}  Since $n_f=1$, the corresponding
  domain $D_f$ must be either an elementary bigon or an elementary
  rectange. Assume first that $D_f$ is an elementary bigon. It
  follows that $\Delta$, which contains $D_f$, must also be a
  bigon. Since $n_1=0$, this in fact is a bigon connecting two points
  on $\alpha_1$.  Correspondingly, $\Delta '$ contains three
  elementary bigons: $B_u'$, $B_d'$, and $D_f=D_f'$. Thus, it must
  correspond to a chain of length at least three. The length of the
  chain can be no longer than three, in view of
  Lemma~\ref{l:specifhandle}.
    \begin{figure}[!ht]
    \begin{center}
      \input{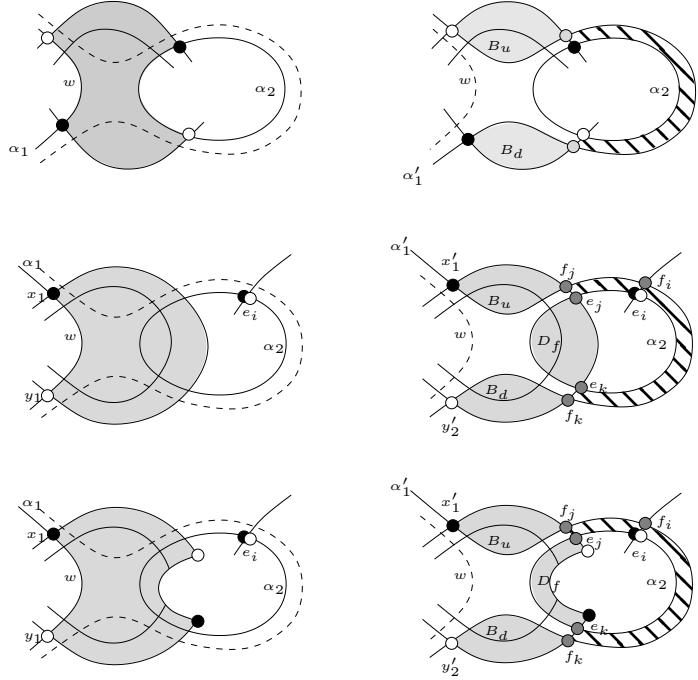}
    \end{center}
    \caption{{\bf An illustration of Proposition~\ref{prop:hIdentifyChains}.}
    \label{f:hIdentifyChains}
    At the left, shaded regions represent the domain $\Delta$ in the
    diagram $\DD$ before the handleslide; these get transformed to
    chains for the diagram $\DD '$ after the handleslide, as indicated
    on the right. (Regions with local multiplicity $-1$ are hatched,
    rather than shaded.) Components of the initial point $\x$ are
    indicated by dark circles, and components of the terminal point
    $\y$ are indicated by white circles. Components of intermediate
    generators appearing in the corresponding chains are indicated by
    gray circles. (For the reader's convenience, we have indicated the
    $\alpha$-circle {\em not} part of the diagram by a dashed arc.)
    The top two diagrams correspond to Case 2 of the proposition,
    the middle two diagrams illustrate Case 4 of
    Proposition~\ref{prop:hIdentifyChains} when $D_f$ is a bigon,
    and the bottom two illustrate Case 4 when $D_f$ is a rectangle.}
  \end{figure}
  Let $x_1$ resp. $y_1$ denote the coordinate of $\x$ resp. $\y$ on
  $\alpha_1$. Let $e_i$ denote the coordinate of $\x$ (and hence also
  $\y$) on $\alpha_2$. Thus, we have some tuple ${\mathbf t}$ with the
  property that $\x=x_1 e_i {\mathbf t}$ and $\y=y_1 e_i {\mathbf t}$,
  cf. the bottom diagrams of Figure~\ref{f:hIdentifyChains} for an
  illustration of this case.

  The $\betak$--arc on the boundary of the bigon $\Delta$ from $\x$ to
  $\y$ also crosses $\alpha_2$ in a pair of points $e_j$ and $e_k$,
  which we order so that $j<k$. Indeed, the fact that the bigon is
  empty ensures that $j<i<k$. There is now a chain:
  \begin{eqnarray*}
    \begin{diagram}
     \x= x_1 e_i {\mathbf t} & \qquad & f_i e_j {\mathbf t} & \qquad \qquad & f_k e_i {\mathbf t} & \qquad & \\
      & \rdTo^{{\mathcal D}'_1} & \dTo & \rdTo^{{\mathcal D}'_2} & \dTo & \rdTo^{{\mathcal D}'_3} & \\
      &  & f_j e_i {\mathbf t} & & f_i e_k {\mathbf t} &  & y_1 e_i {\mathbf t}=\y \\
    \end{diagram}
  \end{eqnarray*}

  Here, by Lemma~\ref{l:specifhandle}, ${\mathcal D}'_2$ must contain
  the bigon $D_f$. Moreover, by orderings, we see that ${\mathcal
    D}'_1$ contains $B_u'$ and ${\mathcal D}'_3$ contains $B_d'$. These 
  properties, along with the fact that ${\mathcal D}'_1$ has an initial
  corner $x_1$ while ${\mathcal D}'_3$ has terminal corner $y_1$, ensure
  that the chain is unqiuely determined by the domain.

  Finally, in the case when $D_f$ is an elementary rectangle, a simple
  adaptation of the above argument provides the result.
\end{proof}

\begin{lem}\label{l:isomegint}
The map $F\colon \CFaa (\DD ) \to Q$ is an isomorphism of chain complexes.
\end{lem}
\begin{proof}
As before, it follows from the construction that $F$ is a vector space
isomorphism. In order to show that it is an isomorphism of chain
complexes, by Proposition~\ref{p:faktor2} it is enough to show that
for generators $\x, \y \in \Gen $ the elements of the set
$\mxy^{\DD}$ are in one-to-one correspondence with the chains
connecting $\x$ and $\y$ in $\DD '$, which is exactly the content of
Proposition~\ref{prop:hIdentifyChains}.
\end{proof}

\begin{prop}\label{p:nicehs}
The homology of $(Q, \partialaa _Q)$ is isomorphic to both
\begin{enumerate}
\item $H_* (\CFaa (\DD ) , \partialaa _{\DD})$ and to
\item $H_* (\CFaa (\DD ') , \partialaa _{\DD '})$.
\end{enumerate}
Consequently, if the nice diagrams $\DD $ and $\DD '$ differ by a
nice handle slide then $\HFaa (\DD )\cong \HFaa (\DD ')$.
\end{prop}
\begin{proof}
Since the property verified by Lemma~\ref{l:isomegint} (together with
the result of Proposition~\ref{p:faktor2}) shows that $F$ is a chain
map, and simple dimension reasons show that it is a vector space
isomorphism, we get that $F$ induces an isomorphism on homologies. On
the other hand, $H_*(Q, \partialaa _Q)$ is isomorphic to $H_* (\CFaa
(\DD '), \partialaa _{\DD '})$, since in the exact triangle of
homologies induced by the short exact sequence $0\to K \to \CFaa (\DD
') \to Q \to 0$ the homology groups of $K$ are obviously 0. This last
observation concludes the proof of the invariance under nice handle
slides.
\end{proof}

\begin{rem}
{\whelm Once again, according to the adaptation of
  Proposition~\ref{p:homeq}, the chain complexes $(\CFaa (\DD ),
  \partialaa _{\DD})$ and $(\CFaa (\DD '), \partialaa _{\DD '})$ are,
  in fact, chain homotopy equivalent complexes.  }
\end{rem}

\subsection{Invariance under nice stabilizations}
\label{ssec:cpxstab}

Recall that we defined two types 
(type-$b$ and type-$g$) of nice stabilizations, depending on 
whether the stabilization increased the number of basepoints 
or the genus of the Heegaard surface. In this subsection we examine
the effect of these operations on the chain complex associated to a 
nice diagram. A nice type-$g$ stabilization is rather simple in this respect,
so we start our discussion with that case.

\begin{thm}\label{t:nicestabg}
Suppose that $\DD$ is a given nice diagram, and $\DD '$ is given 
as a nice type-$g$ stabilization on $\DD$. Then
the chain complexes $(\CFaa (\DD), \partialaa _{\DD})$ and 
$(\CFaa (\DD '), \partialaa _{\DD '})$ 
are isomorphic, and consequently
the Heegaard Floer groups $\HFaa (\DD )$ and $\HFaa (\DD ')$ are
also isomorphic.
\end{thm}
\begin{proof}
  Let $D$ denote the elementary domain in which the nice type-$g$
  stabilization takes place, and denote the newly introduced curves by
  $\alpha _{new}$ and $\beta _{new}$. By the definition of nice
  type-$g$ stabilization, the unique $\betak$--curve intersecting
  $\alpha _{new}$ is $\beta _{new}$, and $\alpha_{new}\cap \beta
  _{new}$ comprises a single point, which we will denote by $x_{new}$.

Suppose now that $\x=\{ x_1, \ldots , x_k \}$ is a generator in $\DD
$. Since on $\alpha_{new}$ of $\DD '$ we can only choose $x_{new}$ as a
coordinate of a point in $\Gen '$, the augmentation map
$\phi \colon {\mathcal  {S}}\to \Gen '$ defined on the
generator $\x =\{ x_1, \ldots , x_k\} $ as
\[
\{ x_1, \ldots , x_k\} \mapsto \{ x_1, \ldots , x_k , x_{new}\} 
\]
provides a bijection between $\Gen $ and ${\mathcal {S}}'$.  Since
all four quadrants meeting at $x_{new}$ contain a basepoint (since all
are part of the domain derived from the chosen $D$ where the
stabilization has been performed), we get that for any $\x,\y \in
\Gen '$ and any ${\mathcal D}\in \mxy$ we have that
$p_{x_{new}}({\mathcal D})=0$, hence the coordinate on $\alpha _{new}$
and $\beta _{new}$ never moves. This verifies that the linear
extension of $\phi $ from the basis $\Gen $ to $\CFaa (\DD )$
provides an isomorphism
\[
f\colon \CFaa (\DD ) \to \CFaa (\DD ')
\]
which, in addition, is a chain map. Consequently the induced map $f
_*\colon \HFaa (\DD ) \to \HFaa (\DD ')$ is an isomorphism,
concluding the proof.
\end{proof}

Suppose finally that $\DD '$ is given by a nice type-$b$ stabilization
of $\DD $.

\begin{thm}\label{thm:stabveg}
If $\DD '$ is given by a nice type-$b$ stabilization on 
$\DD$ then
the homologies of the chain complexes derived from $\DD$ and $\DD '$ 
satisfy the formula
\[
\HFaa (\DD ') \cong \HFaa (\DD )\otimes (\Field \oplus \Field ).
\]
\end{thm}
\begin{proof}
  Recall that a nice type-$b$ stabilization means the introduction of
  a pair of curves $(\alpha_{new}, \beta _{new})$ in an elementary
  domain $D$ of $\DD$ (containing a basepoint $w$) with the property
  that the two new curves are homotopically trivial and intersect each
  other in two points $\{ x_u, x_d \}$, together with the introduction
  of a new basepoint $w_{new}$ in the intersection of the two disks
  $D_{\alpha}, D_{\beta}$, with boundaries $\alpha_{new}$ and $\beta
  _{new}$.  Since $\alpha_{new}$ (and also $\beta _{new}$) contains
  only the two intersection points $x_u$ and $x_d$, any element $\x\in
  \Gen $ gives rise to two elements $\{\x , x_u\}$ and $\{\x , x_d\}$
  of $\Gen '$. In fact, any element of ${\mathcal {S}}'$ arises in
  this way, uniquely specifying the part which originates from
  $\Gen$. This shows that $\CFaa (\DD ')\cong \CFaa (\DD )\otimes
  (\Field \oplus \Field)$.  Now the spaces $\mxy $ considered in
  $\DD$ or $\DD '$ (which will be recoreded in an upper index) can
  be also easily related to each other. Suppose that $\x= \{ \x_1 ,
  x_{n}\} , \y = \{\y_1 , y_n\}\in \Gen '$ with $\x_1, \y_1\in \Gen$
  and $x_n, y_n \in \{ x_u, x_d\}$.
\begin{enumerate}
\item If $x_n = y_n$ then
(since the last coordinate does not move) we have that
$\mxy ^{\DD '}={\mathfrak{M}}_{\x_1, \y_1}^{\DD }$.
\item If $x_n \neq y_n$ then $\x$ and $\y$ can be connected only by a
  bigon with moving coordinates $x_n, y_n$. Hence, if $\mxy^{\DD '}$
  is nonempty, we must have that 
  $\x_1=\y_1$, and indeed
  $\mxy ^{\DD '}={\mathfrak {M}}_{x_u, x_d}$.
\end{enumerate}
Since there are two bigons connecting $x_u$ to $x_d$, the
moduli spaces in case (2) have even cardinality, showing that 
the chain complex $(\CFaa (\DD ' ), \partialaa _{\DD '})$
splits as a tensor product of 
$(\CFaa (\DD ), \partialaa _{\DD})$
 and $(\Field \oplus \Field, 0)$, implying the result. 
\end{proof}

\begin{proof}[Proof of Theorem~\ref{thm:nicemoves}]
The compilation of Propositions~\ref{p:niceisso} and \ref{p:nicehs},
together with Theorems~\ref{t:nicestabg} and \ref{thm:stabveg} provide
the result.
\end{proof}

\section{Heegaard Floer homologies}
\label{sec:hom}
Using the chain complex defined in the previous section for a
convenient diagram, we are ready to define the stable (combinatorial)
Heegaard Floer homology group of a 3--manifold $Y$. The definition
involves two steps, since we can apply our results about convenient
Heegaard diagrams only for 3--manifolds containing no $S^1\times
S^2$--summand.  Recall that we define $b(\DD )$ of a multi-ponted
Heegaard diagram $\DD =(\Sigma , \alphak , \betak , \w)$ as the
cardinality of the basepoint set $\w$.

\begin{defn}\label{def:hf}
{\whelm
\begin{itemize}
\item Suppose that $Y$ is a 3--manifold which contains no $S^1\times
  S^2$--summand. Let $(\Sigma , \alphak , \betak )$ denote an
  essential pair-of-pants diagram for $Y$, and let $\DD$ be a
  convenient diagram derived from $(\Sigma , \alphak , \betak )$ using
  Algorithm~\ref{algo:alg}, having $b (\DD )$ basepoints.  Define the
  stable Heegaard Floer group $\HFast (Y)$ as the equivalence class
\[
[\HFaa (\DD ), b(\DD )]
\]
of the vector space $\HFaa (\DD)$ and the integer
$b(\DD)$.
\item For a general 3--manifold $Y$ consider a decomposition $Y=Y_1\#
  n (S^1\times S^2)$ such that $Y_1$ contains no $S^1\times
  S^2$--summand. The stable Heegaard Floer homology group $\HFast (Y)$
  of $Y$ is then defined as
\[
[ \HFaa (\DD )\otimes (\Field \oplus \Field )^n, b(\DD )],
\]
where $\DD$ is a convenient Heegaard diagram 
derived from an essential pair-of-pants diagram of $Y_1$ using 
Algorithm~\ref{algo:alg}, having $b (\DD )$ basepoints.
\end{itemize}
}
\end{defn}

In order to show that the above definition is valid, first we need to
verify the statement that any 3--manifold admits a convenient Heegaard
diagram. In fact, any genus--$g$ Heegaard diagram with $g$ $\alphak$--
and $g$ $\betak$--curves (the existence of which follows from the
existence of a Morse function on a closed 3--manifold with a unique
minimum and maximum) can be first refined to an essential
pair-of-pants diagram by adding further essential curves to it, from
which the construction of a convenient diagram follows by applying
Algorithm~\ref{algo:alg}.

Next we would like to show that, in fact, the stable Heegaard Floer
homology defined above is a diffeomorphism invariant of the
3--manifold $Y$ and is independent of the chosen convenient Heegard
diagram. 

\begin{thm}\label{thm:invariance}
Suppose that $Y$ is a given closed, oriented 3--manifold.  The stable
Heegaard Floer homology group $\HFast (Y)$ given by
Definition~\ref{def:hf} is a diffeomorphism invariant of $Y$.
\end{thm}
\begin{proof}
According to the Kneser-Milnor Theorem the closed, oriented
3--manifold $Y$ admits a connected sum decomposition $Y=Y_1\#
n(S^1\times S^2)$, where $Y_1$ contains no $S^1\times S^2$--summand.
In addition, the Kneser-Milnor Theorem also shows that both $n$ and
$Y_1$ are (up to diffeomorphism) uniquely determined by $Y$. Since by
definition the stable Heegaard Floer homology group $\HFast (Y)$ of
$Y$ depends only on $\HFast (Y_1)$ and $n$, we only need to verify the
invariance of the stable Heegaard Floer homologies for 3--manifolds
with no $S^1\times S^2$--summand.

Suppose that the closed, oriented 3--manifold $Y$ contains no
$S^1\times S^2$--summand.  Consider two convenient Heegaard diagrams
$\DD_1$ and $\DD_2$ of $Y$ derived from the essential pair-of-pants
diagrams $(\Sigma_1, \alphak_1, \betak_1)$ and $(\Sigma_2, \alphak_2,
\betak_2)$. According to Theorem~\ref{thm:MainConv} any two such
convenient Heegaard diagrams are nicely connected. By
Corollary~\ref{c:nochange}, however, we know that nice moves do not
change stable Heegaard Floer homology. Therefore it implies that
\[
[\HFaa (\DD_1), b(\DD_1)]\cong[\HFaa (\DD_2), b(\DD_2)],
\]
concluding the proof of independence.
\end{proof}

\section{Heegaard Floer homology with twisted coefficients}
\label{s:twist}
It would be desirable to modify the definition of our invariant in such a 
way that we get well-defined vector spaces as opposed to equivalence
classes of pairs of vector spaces and integers. One way to achieve
this goal is to consider homologies with \emph{twisted coefficients},
as we will discuss in this section.

Suppose that $\DD=(\Sigma , \alphak , \betak , \w )$ is a
multi-pointed Heegaard diagram of the 3--manifold $Y$ with $b=b(\DD )$
basepoints. Suppose that $Y$ has no $S^1\times S^2$-summands.
Following \cite[Section~3.4]{OSzlinks}, we define $\pi _2 (\alphak ) $
(and similarly $\pi _2 (\betak )$) as the set of those domains $D=\sum
n _i D_i$ which satisfy that $\partial D=\sum m_i \alpha _i$, i.e. the
boundary of the domain $D$ is a linear combination of entire
$\alphak$--curves.  Elements of $\pi _2(\alphak )$ and $\pi _2 (\betak
)$ are also called $\alphak$-- (and respectively $\betak$--) {\em boundary
degenerations}. The map $m_{\w, \alpha }\colon \pi _2 (\alphak ) \to
\bfz ^b$ (and $m_{\w, \beta }\colon \pi _2 (\betak ) \to \bfz ^b$)
defined on $D \in \pi _2 (\alphak )$ by $m_{\w ,
  \alpha}(D)=(n_{w_1}(D), \ldots , n _{w_b}(D))$ provides an
isomorphism between $\pi _2 (\alphak )$ and $\bfz ^b$. Indeed, by
definition, a domain $D\in \pi _2 (\alphak )$ has constant multiplicity
on an $\alphak$--component, and since this multiplicity can be
arbitrary, and each $\alphak$--component contains a unique basepoint,
the above isomorphism follows.

More generally, for the generators $\x ,\y$ we can consider
\[
m_{\w}\colon \pi _2 (\x , \y ) \to \bfz ^b
\]
by mapping $D\in \pi _2 (\x , \y )$ into $(n_{w_1}(D), \ldots , n
_{w_b}(D))$.  Suppose that $\x = \y$. 
Notice that in this case $\pi _2 (\x , \x )$ admits a natural 
group structure. The kernel ${\mathcal {P}}$ of
the above map is then called the group of \emph{periodic domains}.

A map $\pi _2 (\x , \x )\to H_2 (Y; \bfz )$ can be defined by taking
the 2-chain in $\Sigma$ representing an element $D$ of $\pi _2 (\x ,
\x )$ and then (since its boundary can be written as a linear
combination of entire $\alphak$-- and $\betak$--curves) capping it off
with the handles attached along the $\alphak$-- and $\betak$--curves.
This map fits in the exact sequence
\[
0\to \bfz  \to \pi _2 (\alphak )\oplus \pi _2 (\betak )\to \pi _2 (\x , \x )
\to H_2(Y, \bfz )\to 0 .
\]

In a slightly different manner, distinguish a basepoint $w_1$ (say, in
$D_1$) and then connect the domain of any other basepoint to $D_1$ by
a tube and remove the other baspoint. The resulting once pointed
Heegaard diagram on the $(b-1)$--fold stabilization of $\Sigma$ now
presents the 3--manifold $Y\# _{b-1}S^1\times S^2$, and we get a
simpler version of the above exact sequence:
\[
0\to \bfz \to \pi _2' (\x , \x )\to H_2 (Y\# _{b-1}S^1\times S^2 ; \bfz
) \to 0 .
\]
Here $\pi _2 '(\x , \x )$ is taken in the Heegaard diagram we get
after the stabilizations, and the elements of $\pi _2 '(\x , \x) $
correspond to those elements of $\pi _2(\x, \x )$ which have the same
mulitplicity at the domains containing the basepoints.  The set
${\mathcal {P}}$ of periodic domains is therefore naturally a subset
of $\pi _2'(\x, \x )$, being the collection of those domains for which
the common multiplicity at the basepoints is zero.

Recall that the set $\pi _2 (\x , \y )$ is not always nonempty; in
fact this property induces an equivalence relation on the set of
generators. Let us fix a generator $\x \in {\mathcal {S}}$ in every
equivalence class, and denote the identification of $\pi _2' (\x , \x )$
(i.e. the set of domains in the Heegaard diagram providing $Y\#
_{b-1}S^1\times S^2$) with $H_2(Y\# _{b-1}S^1\times S^2; \bfz )\oplus
\bfz $ by $\phi$.  For any further generator in the same equivalence
class fix a domain $D_{\y}\in \pi _2 (\x , \y )$ with
$(n_{w_i}(D))={\bf 0}$. (By taking any element $D'\in \pi _2(\x ,\y)$
and the element $D''\in \pi _2(\alphak)$, regarded as an element in
$\pi _2 (\x , \x)$, with the property $m_{\w }(D')=-m_{\w , \alpha} (D'')$,
the sum $D'+D''$ will be such a choice.) These choices provide an
identification $\phi _{\y, \z}$ of $\pi '_2 (\y , \z )$ (for all $\y,
\z$ which can be connected to $\x $) with $H_2(Y \# _{b-1}S^1\times
S^2; \bfz )\oplus \bfz $, the last factor is given by $\sum n_{w_i
}(D)$: associate to $D\in \pi '_2 (\y , \z )$ with $(n_{w_i}(D))={\bf
  0}$ the $\phi$--image of the domain $D_{\y }+D-D_{\z}$ (which is
obviously an element of $\pi '_2(\x , \x )$).

In order to define the twisted theory, we need to modify the
definition of both the vector space and the boundary map acting on it.
Suppose that $\DD$ is a nice diagram for $Y$.  Define $\CFa _T(\DD )$
as the free module generated by the generators (the element of the set
${\mathcal {S}}$) over the group-ring $\Field [H_2 (Y\# _{b-1}S^1\times S^2;
  \bfz )]$. In particular, a generator of $\CFa _T (\DD) $, when
regarded as a vector space over $\Field$, is a pair $[\y , a]$, where
$\y\in {\mathcal {S}}$ is an intersection point and $a\in H_2 (Y \#
_{b-1}S^1\times S^2; \bfz )$.

Define
\[
\partiala _{T, \DD}[\y , a]=\sum _{\z \in \Gen}\sum _{D\in {\mathfrak {M}}_{\y \z}} [\z , a+\phi
  _{\y , \z}(D)]
\]
The sum is obviously finite, since there are only at most two elements
in ${\mathfrak {M}}_{\y \z}$, and there are finitely many intersection
points. The simple adaptation of the proof of Theorem~\ref{thm:chaincomplex}
then shows
\begin{prop}
Suppose that $\DD$ is a nice diagram for $Y$. Then
$\partiala _{T, \DD} ^2=0$. \qed
\end{prop}
With this result at hand we have
\begin{defn}
{\whelm Suppose that $Y$ is a given 3-manifold with $Y=Y_1\# _n S^1\times
  S^2$ (and $Y_1$ has no $S^1\times S^2$-summand).  Then define the
  \emph{twisted Heegaard Floer homology} $\HFa _{T}(Y)$ of $Y$ as
  $H_*(\CFa _T (\DD ) , \partiala _{T, \DD})$ for a convenient
  Heegaard diagram $\DD$ of $Y_1$.}
\end{defn}

Two simple examples will be useful in the proof of independence.
\begin{exas}\label{ex:simp}
{\whelm {\bf {(a)}} Suppose that $S^3$ is given by the twice
  pointed Heegaard diagram $\DD =(S^2, \alpha , \beta , w_1, w_2)$ of
  Figure~\ref{f:hd}(a). Then the
generators of $\CFa _T(\DD )$ are of the form $[x,n]$ and $[y,m]$
(where $x,y$ are the two intersection points and $n,m\in \Z$).
By definition $\partial _T [y,n]=0$ and $\partial _T[x,n]=[y,n]+[y,n+1]$,
hence every closed element of $\CFa _T (\DD )$ is homologous either to 0 or to 
$[y,0]$, showing that $\HFa _T(\DD )=\Field$.

\noindent {\bf {(b)}} The once pointed Heegaard diagram of $S^3$ given
by Figure~\ref{f:hd}(b) provides the chain complex $\CFa _T (\DD
)=\Field$, and since $\partiala _{T, \DD }=0$, we get that $\HFa _T
(\DD )=\Field$.}
\end{exas}

\begin{thm}
Suppose that $Y$ is a given 3-manifold. 
Then the combinatorially defined twisted Heegaard Floer homology $\HFa_{T}(Y)$ 
is a topological invariant of $Y$.
\end{thm}
\begin{proof}
By the Kneser-Milnor theorem the decomposition $Y=Y_1\# _n S^1\times
S^2$ is unique, hence we only need to verify the theorem for
3-manifolds with no $S^1\times S^2$-summand.

The independence of the choice of the intersection points $\x$ in their
equivalence classes, and from the choices of the connecting domains
$D_{\y }\in \pi _2 (\x , \y )$ is a simple linear algebra exercise.

Suppose now that $\DD _1$ and $\DD _2$ are two convenient Heegaard
diagrams for a manifold $Y$ with no $S^1\times S^2$-summand.  According to
Theorem~\ref{thm:MainConv} the two diagrams can be connected by a
sequence of nice isotopies, handle slides and the two types of nice
stabilizations. The proof of the invariance of the stable invariant
under nice isotopy and nice handle slide readily applies to show the
invariance of the twisted homology. When a type-$g$ stabilization (the
one increasing the genus, but leaving the number of basepoints
unchanged) is applied, the chain complex does not change, hence the
independence of that move is trivial. 

Finally we have to examine the effect of a type-$b$ stabilization.
Notice that in this case the base ring also changes, so we need to
apply more care. Suppose that we start with a diagram $\DD$.  The
result $\DD _{st}$ of the stabilization can be regarded as the
connected sum of the original diagram $\DD$ with the spherical diagram
$\DD _0$ of $S^3$ shown by Figure~\ref{f:hd}(a).  According to
Example~\ref{ex:simp}(a), the twisted Heegaard Floer homology of that
(nice) spherical Heegaard diagram is $\Field=\Z /2\Z$. Therefore we
get that the chain complex $(\CFa _T(\DD _{st}) , \partiala _{T, \DD
  _{st}})$ is the tensor product of $(\CFa _T (\DD ), \partiala _{T,
  \DD})$ and of $(\CFa _T(\DD _0), \partiala _{T, \DD _0})$ over the
ring $\Field [H_2 (Y\# _b S^1\times S^2; \Z)]$, where the ring acts on
the first chain complex by the requirement that the new element of
$H_2(Y\# _b S^1\times S^2; \Z)$ corresponding to the stabilization
acts trivially, while the new element is the only one with nontrivial
action on the chain complex of the spherical diagram $\DD _0$. Now the
model computation verifies the result.
\end{proof}

The group $H_2(Y\# _{b-1}S^1\times S^2)$ does not split in general
canonically as a sum of $H_2(Y)$ and $H_2(\# _{b-1}S^1\times
S^2)$. The splitting is, however, canonical in the simple case when
$Y$ is a rational homology 3--sphere, implying that $H_2(Y; \bfz
)=0$. In this case the above defined group $\HFa _T (Y)$ is isomorphic
to the conventional Heegaard Floer group $\HFa (Y)$, as it is defined
in \cite{OSzF1}, cf. Theorem~\ref{thm:twisted}. Therefore we get a
combinatorial proof of the following:

\begin{thm}
  For a rational homology spheres 3-manifold $Y$, the
  invariant $\HFa(Y)$ is a topological invariant of $Y$. \qed
\end{thm}

We point out that the twisted group $\HFa _T(Y)$ admits a natural
relative $\bfz$--grading: consider
\[
gr ([\x , a])-gr([\y , b])=\mu (D)
\]
for the domain $D\in \pi _2 (\x , \y )$ with the property
$a+D=b$. (Here $\mu (D) $ is the Maslov index of the domain $D$.
Since $D$ is unique, the above quantity is well-defined.)

\section{Appendix: The relation between $\HFaa (\DD )$ and $\HFa (Y)$}
\label{subsec:Appendix}

In this section we will identify $\HFaa (\DD )$ with an appropriately
stabilized version of $\HFa (Y)$ (which group was defined in
\cite{OSzF1} using the holomorphic theory of Lagrangian Floer
homologies). Notice that in the proof of invariance of $\HFast (Y)$ in
Theorem~\ref{thm:invariance} we used only the
combinatorial/topological arguments discussed in this paper and did
not refer to any parts of the holomorphic theory.

Suppose that $\DD =(\Sigma , \alphak , \betak , \w ) $ is an
admissible, genus-$g$ multi-pointed Heegaard diagram for a 3--manifold
$Y$. (Let $\vert \alphak \vert =\vert \betak \vert =k$ and $\vert \w
\vert =b(\DD )$.)  Following~\cite{OSzlinks} a chain complex
$(\CFa(\DD ), \partiala _{\DD} )$ can be associated to $\DD $ using
Lagrangian Floer homology. Specifically, consider the $k$--fold
symmetric power $\Sym ^k (\Sigma )$ with the symplectic form $\omega$
provided by \cite{perutz} having the property that $\Ta =\alpha _1
\times \ldots \times \alpha _k$ and $\Tb = \beta _1 \times \ldots \times
\beta _k$ are Lagrangian submanifolds of $(\Sym ^k (\Sigma ), \omega
)$.  Then $\CFa (\DD )$ is generated over $\Field=\bfz /2\bfz$ by the
set of intersection points $\Ta \cap \Tb \subset \Sym ^k (\Sigma )$.
Since $\x \in \Ta \cap \Tb$ is an unordered $k$--tuple of points of
$\Sigma$ having exactly one coordinate on each $\alpha _i$ and on each
$\beta _j$, in the case where $\DD$ is a Heegaard diagram, we clearly
have 
\begin{lem}\label{l:izomm}
The $\bfz/2\bfz$-vector spaces $\CFa(\DD)$ and
$\CFaa(\DD)$ are isomorphic under the above identification map. \qed
\end{lem}

Given generators $\x,\y\in\Ta\cap\Tb$, one can consider
pseudo-holomorphic Whitney disks which connect them. To this end, fix
an almost-complex structure $J$ on $\Sym ^k (\Sigma )$ compatible with
the symplectic structure $\omega$, and denote the unit complex disk
$\{ z \in \bfc \mid z{\overline {z}}\leq 1\}$ by ${\mathbb {D}}$. Let
$e_{\alpha}= \{ z \in \bfc \mid z{\overline {z}}= 1, Re (z)\leq 0\}$
and $e_{\beta}= \{ z \in \bfc \mid z{\overline {z}}= 1, Re (z)\geq
0\}$. Define the space ${\mathcal {M}}_{\x ,\y}$ as the set of maps
$u\colon {\mathbb {D}}\to \Sym ^k (\Sigma )$ with the properties
\begin{itemize}
\item $u(i)=\x$ and $u(-i)=\y$,
\item $u(e_{\alpha })\subset \Ta$ and $u(e_{\beta })\subset \Tb$,
\item $u({\mathbb {D}})\cap (\{ w_i \} \times \Sym ^{k-1}(\Sigma ))=\emptyset$ 
for all $w_i \in \w$, and finally 
\item $u$ is $J$--holomorphic, that is, $du(iv)=Jdu(v)$ for all $v\in
  T{\mathbb {D}}$.
\end{itemize}
To each map $u$ as above, one can associate a domain ${\mathcal D}(u)$, which
is a domain connecting $\x$ to $\y$ as in Definition~\ref{def:Domain}
(see~\cite{OSzF1}).  Indeed, it is convenient to consider moduli spaces
${\mathcal M}({\mathcal D})$, the moduli space of pseudo-holomorphic disks $u$
which induce the given domain ${\mathcal D}$. The moduli space ${\mathcal
  M}({\mathcal D})$ has a formal dimension $\Mas({\mathcal D})$ which, as the
notation suggests, depends only on the underlying domain.  For generic $J$ and
$\Mas({\mathcal D})=1$, this moduli space is a smooth 1--manifold with a
free ${\mathbb R}$--action on it.  The number $\#\left(\frac{{\mathcal
    M}({\mathcal D})}{{\mathbb R}}\right)$ denotes the $\pmod{2}$ count of
points in this quotient space (which is compact, and hence a finite collection
of points).

With the help of the moduli spaces ${\mathcal M}({\mathcal D})$ one can now
define a chain complex (provided $J$ is sufficiently generic), as follows.  We
define the boundary map $\partiala \colon \CFa(\DD)\to \CFa (\DD )$ for given
$\x\in \Ta \cap \Tb$ by
\[
\partiala \x =\sum _{\y \in \Ta \cap \Tb }
\sum_{\{{\mathcal D}\in\pi_2(\x,\y)\big|n_{\ws}({\mathcal D})=0, \Mas({\mathcal D})=1\}}
\#\left(\frac{{\mathcal M}({\mathcal D})}{{\mathbb R}}\right)\cdot \y.
\]

In the case where $b(\DD)=1$, the homology of the above chain complex
is the 3--manifold invariant $\HFa(Y)$ from~\cite{OSzF1}. 
More generally, we have the following result from~\cite{OSzlinks}:

\begin{thm}
  \label{thm:MultipleBasepoints}
  If $\DD$ is an admissible, multi-pointed Heegaard diagram for a
  3--manifold $Y$, then the homology of the above complex is related
  to the 3--manifold invariant $\HFa(Y)$ by
  \[ H_*(\CFa(\DD))\cong \HFa(Y)\otimes (\Field\oplus\Field)^{b(\DD)-1}.\]
\qed
\end{thm}

In view of this, the main theorem from~\cite{SW} can quickly be
adapted to prove the following:

\begin{thm}\label{thm:uaz} 
Suppose that $\DD$ is a nice multi-pointed Heegaard diagram of 
$Y$. Then 
\[
\HFaa (\DD)\cong \HFa(Y)\otimes (\Field\oplus\Field)^{b(\DD)-1}.
\]
\end{thm}
\begin{proof}
  In view of Lemma~\ref{l:izomm} and
  Theorem~\ref{thm:MultipleBasepoints}, it suffices to identify the
  boundary operator of $\CFa(\DD)$ with the boundary operator of
  $\CFaa(\DD)$.

  The argument for the above identification uses the following facts:
  \begin{enumerate}
    \item A theorem of Lipshitz~\cite{Lipshitz}, according to which
      the Maslov index $\mu({\mathcal D})$ in the holomorphic theory
      is, indeed, given by Equation~\eqref{eq:MaslovIndex}.
    \item A simple principle, according to which one can choose generic 
      $J$ so that ${\mathcal M}({\mathcal D})$ is empty unless
      ${\mathcal D}\geq 0$.
    \item The fact that, for a nice Heegaard diagram, 
      the nonnegative domains with Maslov index one are precisely
      bigons or rectangles (cf. Proposition~\ref{p:atfog}).
    \item An observation that in the case where ${\mathcal D}$ is a polygon,
      $\#\left(\frac{{\mathcal M}({\mathcal D})}{{\mathbb R}}\right)=1$ (mod 2), 
      see~\cite{KT, Rasmussen}.
  \end{enumerate}
\end{proof}

In addition, the same principle shows that for the twisted
theory we have the following partial identification of the resulting
groups:

\begin{thm}\label{thm:twisted}
Suppose that $Y$ is a rational homology 3-sphere, that is, its first
Betti number $b_1(Y)$ vanishes. The twisted (topological) Heegaard
Floer homology $\HFa _T (Y)$ (as it is defined in
Section~\ref{s:twist}) is isomorphic to $\HFa (Y)$ (as it is defined in 
\cite{OSzF1}, using the holomorphic theory).  \qed
\end{thm}

\section{Appendix: Handlebodies and pair-of-pants decompositions}
\label{appendix}
For the sake of completeness, in this Appendix we verify a slightly
weaker version of Theorem~\ref{t:luo} of Luo, which is still
sufficient for the applications in this paper.  Let us assume that
$\Sigma $ is a genus-$g$ surface with $g>1$, and suppose that
$\alphak$ and $\alphak '$ are two markings of the surface $\Sigma$.
Recall from Definition~\ref{def:flip} that the two pair-of-pants
decompositions $\alphak$ and $\alphak '$ differ by a generalized flip
(or g-flip) if $\alphak = \alphak _0\cup\{ \alpha \}, \alphak =
\alphak _0\cup\{ \alpha '\}$, and $\alpha, \alpha '$ are both
contained by the 4-punctured component of $\Sigma -\alphak _0$.
Decompositions differing by a sequence of g-flips are called
\emph{g-flip equivalent}. Then the main result of this Appendix is:

\begin{thm}\label{thm:luohely}
  Suppose that $\alphak $ and $\alphak '$ are two markings on the
  surface $\Sigma$. The markings determine the same handlebody if and
  only if the markings $\alphak $ and $\alphak '$ are g-flip
  equivalent.
\end{thm}

We start with some preparatory constructions.  Suppose
that $\Sigma $ is of genus $g>1$ and $\alphak$ is a given marking on
$\Sigma$.  Recall that then $\alphak$ contains $3g-3$ curves.  The set
$\{ \alpha _1, \ldots , \alpha _g\} \subset \alphak$ of curves of the
marking is called a \emph{spanning $g$--tuple} for the pair-of-pants
decomposition if the subspace spanned by $\{ \alpha _1, \ldots ,
\alpha _g \}$ in $H_1 (\Sigma ; \Z/2\Z)$ is $g$--dimensional, i.e.
the curves are homologically independent. 

We will prove  Theorem~\ref{thm:luohely} in two steps:
first we assume that $\alphak$ and $\alphak '$ admit a common
spanning $g$--tuple, and in the second step we treat the general
case. (This second argument will be considerably shorter and simpler
than the first.)

\begin{prop}\label{prop:common}
Suppose that $\alphak$ and $\alphak ' $ are two markings with
identical spanning $g$--tuples. Then $\alphak$ and $\alphak'$ can be
connected by a sequence of g-flips and isotopies through markings.
\end{prop}

\begin{proof} 
Let $A=\{ \alpha _1 , \ldots , \alpha _k\}$ and $A'=\{ \alpha _1' ,
\ldots , \alpha _k'\}$ denote the maximal subsets of $\alphak$ and
$\alphak '$, respectively, with the property that $\alpha _i$ and
$\alpha _i'$ are isotopic for $i=1,\ldots , k$. In the following
(after applying the isotopy) we will identify the two sets.  By our
assumption we have that $k\geq g$ and the complement $\Sigma -A$ is the
disjoint union of punctured spheres.

If $k$ is $3g-3$, then all components of $\Sigma -A$ are
pairs-of-pants, hence $\alphak$ and $\alphak '$ are isotopic
decompositions, hence there is nothing to prove. If $k$ is $3g-4$,
then there is a component of $\Sigma -A$ which is a 4-punctured
sphere, the further components are pairs-of-pants.  The 4-punctures sphere
component contains a pair of (nonisotpic)
$\alphak$-- and $\alphak '$--curves. By the definition of g-flip,
these are related by a g-flip move, hence the decompositions are
connected by g-flips. Notice that in the intermediate stages the
appearing curves already were part of $\alphak$ or $\alphak '$, hence
all curves are homologically essential.

Suppose now that the statement is proved for pairs with $\vert A\vert
=k+1$, and consider a pair $\alphak, \alphak '$ which has $k$ as the
size of the corresponding set $A$. Let $F$ be a component of $\Sigma
-A$ which is not a pair-of-pants.  We will concentrate only on those
curves of $\alphak$ and $\alphak '$ which are contained by
$F$. Suppose that $\alpha$ and $\alpha '$ (elements of $\alphak $ and
$\alphak '$, resp.)  are \emph{minimal} curves, in the sense that by
deleting them $F$ falls into two components, one of which is a
pair-of-pants. (By the usual 'innermost circle' argument it is easy to
see that such curves always exist.) Let $a_1, a_2$ denote the two
further boundary circles of the pair-of-pants bounded by $\alpha$ (and
let $a_1', a_2'$ denote the similar two circles for $\alpha '$).

First we would like to present a normalization procedure for these
minimal curves, hence for the coming lemma we only consider the
decomposition $\alphak$ and temporarily forget about $\alphak '$. Let
$a$ be an embedded arc connecting the boundary circle $a_1$ and $a_2$
in the complement $F-\alphak$ in such a way that the boundary of the
tubular neighbourhood of $a\cup a_1\cup a_2$ in $F$ is
$\alpha$. Consider another embedded path $b$ in $F$ joining $a_1$ and
$a_2$ and let $\beta$ denote the boundary of the tubular neighborhood
of $b\cup a_1\cup a_2$. (Notice that now $b$ is not necessarily in the
complement of the $\alphak$--curves.)

\begin{lem}
The marking $\alphak$ is g-flip equivalent to a marking $\betak$
containing all curves of $A$ and $\beta$. The sequence connecting
$\alphak$ and $\betak$ is through markings.
\end{lem}
\begin{proof}
First of all, we can assume that $a$ and $b$ are disjoint: by
considering a curve $c$ which is parallel to $a$ until its first
intersection with $b$, and then parallel with $b$, by chosing the
appropriate side for the parallels we can reduce the number of
intersections of $a$ and $b$ by one, and since being g-flip equivalent
is an equivalence relation, we only need to deal with disjoint $a$ and
$b$.

Consider the surface $F'$ we get by capping off all the boundary
components of $F$ with punctured disks (with punctures $p_i$) except
$a_1$ and $a_2$. In the resulting annulus the two arcs $a$ and $b$ are
obviously isotopic (by allowing to isotope the endpoints of these arcs on the corresponding
boundary components). Suppose that such an  isotopy sweeps through the
marked points $p_1, \ldots , p_n$ of $F'$ (recording the further
boundary components of $F$). Obviously if $n=0$ then $a$ and $b$ were
already isotopic in $F$ and there is nothing to prove. We will show a
g-flip reducing $n$ by one. Indeed, choose an arc $b'$ connecting
$a_1$ and $a_2$ in $F$ such that $b'$ is disjoint from both $a$ and
$b$, and the isotopy in $F'$ from $a$ to $b'$ sweeps through a single
marked point $p$. The boundary of the tubular neighbourhood of $b'\cup
a_1 \cup a_2$ will be denoted by $\beta '$. Let $\gamma$ denote the
boundary component of the tubular neighbourhood of $a\cup b'\cup
a_1\cup a_2$ with the property that its complement in $F$ has a
4-punctured sphere component. The other component of $F-\gamma$ will
be denoted by $G$.  Let $\gammak$ denote a pair-of-pants decomposition
of $G$ containing curves homologically essential in $\Sigma$. This
decomposition $\gammak$ gives rise to two decompositions of $F$: we
add to it $\{\gamma , \alpha \}$ or $\{ \gamma , \beta '\}$.  Now
these two decompositions differ by a g-flip (changing $\alpha$ to
$\beta '$), but $\gammak \cup \{ \gamma , \alpha \}$ is g-flip
equivalent to $\alphak$ by induction (since they share one more common
curve, namely $\alpha$) while $\gammak \cup \{ \gamma , \beta '\}$ is
g-flip equivalent to any decomposition containing $\beta '$ (for the
same reason). By induction on the distance $n$ of $a$ and $b$ (i.e.
the number of $p_i$'s an isotopy in $F'$ sweeps accross), the proof of
the lemma is complete.
\end{proof}

Returning to the proof of Proposition~\ref{prop:common}, therefore we can assume
that $a$ connects the two boundary components in any way we like.  We
will distinguish three cases according to the number $C$ of common
circles of $\{ a_1, a_2\}$ and $\{ a_1', a_2'\}$. If $C=2$, then by
the above lemma we can assume that after a sequence of g-flips
$\alpha$ coincides with $\beta$, hence by induction the two
pairs-of-pants decompositions are g-flip equivalent.  If $C=1$ (i.e.
say $a_2=a_2'$), we can again assume that $a$ and $b$ are disjoint,
and then the curve $\delta$, which is the boundary of $a\cup b\cup
a_1\cup a_1'\cup a_2$ separates a 4-puntured sphere in which a g-flip
moves $\alpha $ to $\beta$ and any extension of it will produce (by
induction) a decomposition which is g-flip equivalent (with $\{ \delta
, \alpha\}$) to $\alphak$ and (with $\{ \delta , \beta \}$) to
$\betak$. Finally if $C=0$ then again first we assume that $a$ and $b$
are disjoint, and consider a curve $\delta$ in $F$ which splits off
$a_1, a_2, a_1', a_2'$ (and the curves $\alpha, \beta $) from $F$. Any
extension of these three curves will produce a decomposition which is
g-flip equivalent to both $\alphak$ and $\betak$ by induction, hence
the proof of Proposition~\ref{prop:common} is complete.
\end{proof}

With the above special case in place, we can now turn to the 
\begin{proof}[Proof of Theorem~\ref{thm:luohely}]
Suppose now that $\alphak$ and $\alphak'$ are given pair-of-pants
decompositions, together with the chosen spanning $g$--tuples. If the
spanning $g$--tuples coincide, then Proposition~\ref{prop:common}
applies and finishes the proof.

Suppose now that $\alphak$ and $\alphak '$ admit spanning $g$--tuples
differing by a single handle slide.  In this case there is a
pair-of-pants decomposition $\alphak _1$ containing both spanning
$g$--tuples: the handle slide $\alpha_1$ on $\alpha _2$ determines a
pair-of-pants bounded by $\alpha _1, \alpha _2$ and $\alpha _1'$ (the
result of the handle slide), and refining this triple (together with
$\alpha _3, \ldots , \alpha _g$) to a pair-of-pants, we get the
desired pair-of-pants decomposition $\alphak _1$.  The application of
Proposition~\ref{prop:common} for the pairs $(\alphak , \alphak _1)$
and for $(\alphak ' , \alphak _1)$ and the fact that being g-flip
equivalent is transitive now shows that $\alphak $ and $\alphak '$ are
g-flip equivalent.

Since (by a classical result) two $g$--tuples determining the same
handlebody can be transformed into each other by a sequence of handle
slides and isotopies, the repeated application of the above argument
completes the proof.
\end{proof}

\end{document}